\documentclass{birkjour}
\usepackage[english]{babel}
\usepackage{amssymb}
\usepackage{amsfonts}
\usepackage{amsmath}
\usepackage[hidelinks]{hyperref}
\usepackage{bm}
\usepackage{latexsym}
\usepackage{epsfig}
\usepackage{faktor}
\usepackage{xfrac}  
\numberwithin{equation}{section}

\newtheorem{theorem}{Theorem}[section]
\newtheorem{cor}[theorem]{Corollary}
\newtheorem{lemma}[theorem]{Lemma}
\newtheorem{proposition}[theorem]{Proposition}
\newtheorem{definition}[theorem]{Definition}

\newtheorem{remark}[theorem]{Remark}
\newtheorem{conj}[theorem]{Conjecture}
\newtheorem{example}[theorem]{Example}

\numberwithin{equation}{section}
\def\beq{\begin{equation}}
\def\eeq{\end{equation}}
\def\ben{\begin{enumerate}}
\def\een{\end{enumerate}}
\def\cA{{\mathcal A}}

\def\cE{{\mathcal E}}

\def\cG{{\mathcal G}}
\def\cH{{\mathcal H}}
\def\cI{{\mathcal I}}
\def\cJ{{\mathcal J}}
\def\cL{{\mathcal L}}

\def\cN{{\mathcal N}}

\def\cR{{\mathcal R}}
 
\def\cT{{\mathcal T}}
\def\cU{{\mathcal U}}
\def\cV{{\mathcal V}}
\def\cW{{\mathcal W}}
\def\cY{{\mathcal Y}}
\def\cZ{{\mathcal Z}}

\def\cW{\mathcal W}

\newcommand{\CC}{{\mathbb C}}

\newcommand{\KK}{{\mathbb K}}
\newcommand{\NN}{{\mathbb N}}

\newcommand{\fR}{{\mathfrak R}}

\newcommand{\mtm}{m\times m}
\newcommand{\ntn}{n\times n}

\newcommand{\mtn}{m\times n}
\newcommand{\ntm}{n\times m}

\newcommand{\sts}{s\times s}

\newcommand{\LTL}{L\times L}
\newcommand{\wt}{\widetilde}

\newcommand{\ul}{\underline}

\newcommand{\bB}{\mathbf{B}}
\newcommand{\bT}{\mathbf{T}}

\newcommand{\bA}{\mathbf{A}}

\newcommand{\fC}{\mathfrak{C}}
\newcommand{\fc}{\mathfrak{c}}
\newcommand{\fo}{\mathfrak{o}}
\newcommand{\fO}{\mathfrak{O}}

\newcommand{\fB}{\mathfrak{B}}
\newcommand{\fA}{\mathfrak{A}}
\newcommand{\fX}{\mathfrak{X}}
\newcommand{\fY}{\mathfrak{Y}}
\newcommand{\fZ}{\mathfrak{Z}}
\newcommand{\fa}{\mathfrak{a}}
\newcommand{\fr}{\mathfrak{r}}
\newcommand{\fS}{\mathfrak{S}}
\newcommand{\fb}{\mathfrak{b}}

\newcommand{\fM}{\mathfrak{M}}
\usepackage{amsmath}

\DeclareMathOperator{\Ima}{Im}
\usepackage{color}
\usepackage[normalem]{ulem}
\makeatletter
\def\moverlay{\mathpalette\mov@rlay}
\def\mov@rlay#1#2{\leavevmode\vtop{
    \baselineskip\z@skip \lineskiplimit-\maxdimen
    \ialign{\hfil$#1##$\hfil\cr#2\crcr}}}
\makeatother
\newcommand{\plangle}{\moverlay{(\cr<}}
\newcommand{\prangle}{\moverlay{)\cr>}}

\newcommand{\introthmname}{}
\newtheorem{introthminn}{\introthmname}
\newenvironment{introthm}[1]
  {\renewcommand{\introthmname}{#1}\begin{introthminn}}
  {\end{introthminn}}
\begin{document}

%+Title
\title[The lost-abbey conditions]{Realizations of Non-Commutative Rational Functions Around a Matrix Centre, II:
The Lost-Abbey Conditions}
\thanks{The research of both authors was 
partially supported by the US--Israel 
Binational Science Foundation (BSF) 
Grant No. 2010432, Deutsche Forschungsgemeinschaft (DFG) 
Grant No. SCHW 1723/1-1, 
and Israel Science Foundation (ISF) 
Grant No. 2123/17.}
\author[M. Porat]{Motke Porat}
\address{Department of Mathematics\\
Ben-Gurion University of the Negev\\
P.O. Box 653,
Beer-Sheva 84105\\ Israel}
\email{motpor@gmail.com}
\author[V. Vinnikov]{Victor Vinnikov}
\address{Department of Mathematics\\
Ben-Gurion University of the Negev\\
P.O. Box 653,
Beer-Sheva 84105\\ Israel}
\email{vinnikov@math.bgu.ac.il}
%-Title
%+Abstract
\begin{abstract}
In a previous paper the authors 
generalized classical results of minimal 
realizations of non-commutative (nc) rational functions, 
using  nc
Fornasini--Marchesini
realizations which are centred at an arbitrary matrix point. 
In particular, it was proved that the domain of regularity of a nc 
rational function is contained in  the invertibility 
set of a corresponding pencil of any minimal realization 
of the function.
In this paper we prove an equality between the domain of a nc rational function and the domain of any of its minimal realizations. 
As for evaluations over stably finite algebras, we show that the domain of the realization w.r.t any such algebra coincides with the so called matrix domain of the function w.r.t the algebra. As a corollary we show that the domain of regularity and the stable extended domain coincide.
In contrary to both the classical case and the scalar case --- where every matrix coefficients which satisfy the controllability and observability conditions can appear in a minimal realization of a nc rational function --- the matrix coefficients in our case have to satisfy certain equations, called linearized lost-abbey conditions, which are related to Taylor--Taylor expansions in nc function theory.
\end{abstract}
%-Abstract
\maketitle
%+Contents
\tableofcontents
%-Contents

\section*{Introduction}
Non-commutative (nc, for short) rational functions 
are a skew field of fractions
--- more precisely, the universal skew field of fractions --- 
of the
ring of nc polynomials, i.e., polynomials in 
noncommuting indeterminates (the free associative algebra).
Essentially, they are obtained
by starting with nc polynomials and applying
successive arithmetic operations;
a considerable amount of technical details is necessary here since
in contrast to the commutative case there is no canonical coprime
fraction representation for a nc rational function. NC rational functions originated from several sources: the general theory of free rings and of skew fields
(see 
\cite{Co61,Hu70,Co71a,Co72,Le74,Linn93,Lich00},
\cite{Co71,Co06,Co95} 
for comprehensive expositions, and
\cite{R99,Linn06} 
for good surveys); 
the theory of rings with
rational identities (see 
\cite{AM66}, 
also 
\cite{Be70} and
\cite[Chapter 8]{Row80}); 
and rational former power series in the
theory of formal languages and finite automata (see
\cite{Kle,Schutz61,Schutz62b,Fliess70,Fliess74a,Fliess74b} 
and
\cite{BR} for a good survey). 

Much like in the case of rational functions of a single variable 
\cite{BGK,KAF}
(and unlike the case of several commuting variables 
\cite{G1,Kac1}),
nc rational functions that are regular at 
$\ul{0}$
admit a good state space realization theory, 
see in particular Theorem 
\ref{thm:realization} below.
This was first established in the context of finite automata and recognizable power series,
and more recently reformulated, with additional details, in the context of transfer functions of multidimensional systems with evolution
along the free monoid 
(see 
\cite{BV1,BGM1,BGM2,BGM3,AK0,BK-V}).
State space realizations of nc rational functions have figured
prominently in work on robust control of linear systems subjected
to structured possibly time-varying uncertainty (see 
\cite{Beck,
BeckDoyleGlover, LuZhouDoyle}). 
Another important application of nc rational functions
appears in the area of Linear Matrix Inequalities (LMIs, see,
e.g., 
\cite{NN,N06,SIG97}). 
Most optimization problems of system theory and
control are dimensionless in the sense that the natural variables
are matrices, and the problem involves nc rational expressions in these matrix
variables which have therefore the same form independent of matrix sizes (see
\cite{AK,CHSY,H1,H3,HM}).
State space realizations are exactly what is
needed to convert (numerically unmanageable) rational matrix
inequalities into (highly manageable) linear matrix inequalities
(see \cite{HKMV,HM2,HMcCV}). 
 
 Coming from a different direction, the method of state space realizations,
also known as the linearization trick, found important recent applications in free
probability, see 
\cite{BMS,HMS,Sp1,Sp3}. 
Here it is crucial to evaluate nc rational expressions
on a general algebra --- which is stably finite in many important cases ---
rather than on matrices of all sizes. Stably finite algebras appeared in this context in the work of Cohn
\cite{Co06} and they play an important and not surprising role in our analysis.

Here is  a full characterization of nc rational functions which are regular at $\ul{0}$
and their (matrix) domains of regularity, 
in terms of their minimal realizations $\nu$
(for the proofs, see 
\cite{BGM1,BGM2,Fliess70,Fliess74a,Fliess74b,KV3,KV4,Kle}).
\begin{introthm}{Theorem}
\label{thm:realization}
If 
$\fR$
is a  nc rational function of 
$x_1,\ldots,x_d$ 
and 
$\fR$ 
is regular at 
$\ul{0}$, 
then
$\fR$ 
admits a unique (up to unique similarity) 
minimal nc Fornasini--Marchesini realization
\begin{equation*}
\fR(x_1,\ldots,x_d)=D+C
\Big(I_L-
\sum_{k=1}^d A_kx_k
\Big)^{-1}
\sum_{k=1}^d B_kx_k,
\end{equation*}
where 
$A_1,\ldots,A_d\in\KK^{L\times L} 
,B_1,\ldots,B_d\in\KK^{L\times 1}, 
C\in\KK^{1\times L},D=\fR(\ul{0})\in\KK$ 
and 
$L\in\NN$. 
Moreover, for all 
$m\in\NN:$ 
$\ul{X}=(X_1,\ldots,X_d)\in(\KK^{\mtm})^d$ 
is in the domain of regularity of 
$\fR$ if and only if  
$\det\left(I_{Lm}-X_1\otimes A_1-\ldots-
X_d\otimes A_d\right)\ne0;$ in that case
\begin{equation*}
\fR(\ul{X})=
I_m\otimes D+(I_m\otimes C)
\Big( I_{Lm}-\sum_{k=1}^d
X_k\otimes A_k\Big)^{-1}\sum_{k=1}^d
X_k\otimes B_k.
\end{equation*}
\end{introthm}
Here a realization is called minimal if the state space dimension 
$L$ 
is as small as possible; this is equivalent to   
the realization being observable, i.e., 
$$\bigcap_{0\le k}
\bigcap_{\,1\le i_1,\ldots,i_k\le d}\ker(CA_{i_1}\cdots A_{i_k})
=\{\ul{0}\},$$ and controllable, i.e.,
$$\bigvee_{0\le k}
\bigvee_{\,1\le i_1,\ldots, i_k,j\le d,} A_{i_1}\cdots A_{i_k}B_j=\KK^L.$$
Theorem 
\ref{thm:realization} 
is strongly related to  
expansions of nc rational functions which are regular at 
$\ul{0}$ 
into formal nc power series around 
$\ul{0}$; 
that is why it is \textbf{not} 
applicable for all nc rational functions. 
For example, the nc rational expression
$R(x_1,x_2)=(x_1x_2-x_2x_1)^{-1}$
is not defined at $\ul{0}$, nor at any 
pair
$(y_1,y_2)\in\KK^2$, 
therefore one can not consider 
realizations of 
$R$
which are centred at 
$\ul{0}$ 
as in Theorem 
\ref{thm:realization}, nor at any scalar point 
(a tuple of scalars). A realization theory for such 
expressions (and hence functions) 
is required  in particular for all of the applications 
mentioned above.  Such a theory is 
presented in our first paper 
\cite{PV1} and continues here, using the ideas of the general 
theory of nc functions. Other types of realizations of 
nc rational functions that are 
not necessarily regular at 
$\ul{0}$ have been considered (see 
\cite{CR1,CR2}, 
\cite{V2}, 
and also the recent papers
\cite{Schr1,Schr2,Schr3,Schr4}).

The theory of \textbf{nc functions} has its roots in the works by Taylor \cite{T1,T2} on noncommutative
spectral theory. 
It was further developed by Voiculescu 
\cite{VDN,Voi04,Voi09} 
and Kalyuzhnyi-Verbovetskyi--Vinnikov \cite{KV1},
including a detailed discussion on nc difference-differential calculus.
The main underlying idea is that
a function of 
$d$ 
non-commuting variables is a function of 
$d-$tuples of square matrices of all sizes
that respects direct sums and simultaneous similarities.
See also
the work of Helton--Klep--McCullough 
\cite{HKMcC1,HKMcC2}, 
of Popescu 
\cite{Po06,Po10}, 
of Muhly--Solel 
\cite{MS}, and of Agler--McCarthy 
\cite{AgMcC1,AgMcC2,AgMcC6}.

A crucial fact 
\cite[Chapters 4-7]{KV1} 
is that nc functions admit power series expansions,
called Taylor--Taylor series in honor of 
Brook Taylor and of Joseph L. Taylor,
around an arbitrary
matrix point in their domain.
However, a main difference between the scalar and the non-scalar centre cases, is that the coefficients in the Taylor--Taylor series are not arbitrary and must satisfy some compatibility conditions, called the lost-abbey conditions
(see 
\cite[equations (4.14)-(4.17)]{KV1}, 
or equations
$(\ref{eq:7May18a})-(\ref{eq:7May18c}$) below). 
This motivates us to generalize realizations 
as in Theorem 
\ref{thm:realization} above,
to the
case where the centre is a 
$d-$tuple of matrices
rather than 
$\ul{0}$ 
or a 
$d-$tuple of scalars.
We refer also to
\cite{KVV}
for a recent application of the lost-abbey conditions for the study of germs of nc functions. 
\\\\
\iffalse
The main goal of the rest of the chapter is to obtain the other inclusion
in Theorem
\ref{thm:30Jan18c} and thereby complete the characterization of nc rational
functions via their minimal realizations
(cf. Theorem
\ref{thm:3Jun19a}). As a byproduct we will also be able to characterize the
collections
of $\ul{\bA},\ul{\bB},C,D,$
with $(\ul{\bA},\ul{\bB})$
controllable and
$(C,\ul{\bA})$
observable, that can appear in a minimal realization of a nc rational function
(cf. Theorem
\ref{thm:5Nov19a}). 
To do so we invoke some tools from the general theory of nc functions.
\fi
This is the second in a series of papers with the goal
of generalizing the theory of Fornasini--Marchesini realizations centred at $\ul{0}$ (or any other scalar point), to the case of Fornasini--Marchesini realizations centred at an arbitrary matrix point in the domain of regularity of a nc rational function. 
In the first paper of the series 
(\cite{PV1}),
we proved the following main result:
\iffalse
a partial generalization of Theorem
\ref{thm:realization} 
(see Theorem  
\ref{theorem5} below), 
namely proving the existence and uniqueness of a minimal realization,
together with the inclusion of the domain of the nc rational function in the domain of any of its minimal realizations:
\fi
\begin{introthm}{Theorem}[\cite{PV1}, Corollary 2.18 and Theorem 3.3]
\label{theorem5}
If $\fR$ 
is a nc rational function of 
$x_1,\ldots,x_d$ 
over 
$\KK$, 
then for every 
$\ul{Y}=(Y_1,\ldots,Y_d)\in dom_s(\fR)$ 
there exists a unique 
(up to unique similarity) minimal (observable and controllable) 
nc Fornasini--Marchesini realization
\begin{align*}
\cR(\ul{X})= D+C
\Big(I_{L}-\sum_{k=1}^d 
\bA_k(X_k-Y_k)\Big)^{-1}
\sum_{k=1}^d 
\bB_k(X_k- Y_k)
\end{align*}
centred at 
$\ul{Y}$, 
such that
\begin{multline*}
dom_{sm}(\fR)\subseteq 
\Omega_{sm}(\cR):=\Big\{ \ul{X}\in
(\KK^{sm\times sm})^d: 
\\ \det\Big( I_{Lm}-\sum_{k=1}^d
(X_k-I_m\otimes Y_k)\bA_k\Big)\ne0\Big\}
\end{multline*}
and 
\begin{multline*}
\fR(\ul{X})=\cR(\ul{X})=I_m\otimes D+(I_m\otimes C)
\Big(I_{Lm}-\sum_{k=1}^d 
(X_k-I_m\otimes Y_k)\bA_k\Big)^{-1} \\
\sum_{k=1}^d (X_k-I_m\otimes Y_k)\bB_k
\end{multline*}
for every 
$\ul{X}=(X_1,\ldots,X_d)\in dom_{sm}(\fR)$
and 
$m\in\NN$.
\end{introthm} 
Here $C,D$
are matrices of appropriate sizes, while $\bA_1\ldots,\bA_d,
\bB_1,\ldots,\bB_d$
are linear mappings from 
$\KK^{\sts}$
to matrices of appropriate sizes, 
see Subsection
\ref{subsec:prelimFirstPaper} 
below for the exact description.
 
This is a partial generalization 
of Theorem 
\ref{thm:realization},
as it only shows the inclusion of the domain of a nc rational function 
in the domain of a minimal nc Fornasini--Marchesini realization that the
function admits. 
The difficulty is that in contrast to Theorem 
\ref{thm:realization}, 
a minimal nc Fornasini--Marchesini
realization 
of a nc rational expression centred at a matrix point  is no longer a nc
rational expression 
by itself.
An equality between 
the two domains does hold (Theorem
\ref{thm:3Jun19a}), but the proof requires more tools from the theory of
nc functions, including their nc generalized power series expansions and
difference-differential calculus,
which are developed in this paper. 

Notice that we made a change of notations from 
\cite{PV1}, where what is denoted here by
$\Omega_{sm}(\cR)$, was denoted by $DOM_{sm}(\cR)$.
\\
\\
\textbf{\uline{Outline and Main Results:}}
In Section
\ref{sec:la} we provide a linearization of the lost-abbey 
($\cL\cA$) conditions,
in the case where the nc generalized 
power series admits a minimal nc 
Fornasini--Marchesini realization (Lemma
\ref{lem:13Dec17a}). Then we prove 
that for any nc Fornasini--Marchesini realization $\cR$ such that its coefficients
satisfy these linearized lost-abbey ($\cL-\cL\cA$)  
conditions, the evaluation of
$\cR$
is a nc function on the invertibility set 
$\Omega(\cR)$ of the corresponding pencil, 
which is an upper admissible nc set
(Theorem 
\ref{thm:19Oct17a}). 

Then, in Section
\ref{subsec:cal},
we use the fact that 
$\cR$
is a nc function on the upper admissible nc set
$\Omega(\cR)$, to 
apply the nc difference-differential calculus and to show that $\Omega(\cR)$
is similarity invariant
(Theorem 
\ref{thm:6Aug19a}), 
under the assumptions that
$\cR$ is controllable and observable, 
and its coefficients satisfy the $\cL-\cL\cA$ 
conditions.
Some formulas such as
(\ref{eq:16Jul19b}), that appear in the proof of Lemma 
\ref{lem:12Ap17a} 
might be of separate interest for applications in free analysis.
 
Finally, in Section \ref{sec:mainR}
we show that if a realization is controllable and observable, and its coefficients
satisfy the $\cL-\cL\cA$ conditions, 
then the realization is actually the restriction of a nc rational
function (Theorem 
\ref{thm:26Jun19a}). In the proof of Theorem 
\ref{thm:26Jun19a} 
we make an extensive use of the results from Sections 
\ref{sec:la}
and
\ref{subsec:cal}, 
as well as of 
\cite[Lemma 3.9]{V1}
(cf. Lemma
\ref{lem:25Jun19a}).

As a corollary of Theorem
\ref{thm:26Jun19a}, 
we prove one of the main results in the paper, that is that the domain of
a nc rational function
\textbf{coincides} with --- and not only contains as in 
Theorem \ref{theorem5} ---
the domain of any of its minimal realizations, centred
at
an arbitrary matrix point, which allows us to evaluate 
the nc rational function on \textbf{all} of its domain using the evaluation
of the realization. The corresponding result when evaluating over stably finite algebras 
is given too, 
while
one has to modify our definition of the $\cA-$domain of the function and
consider its $\cA-$matrix domain. 
\begin{introthm}{Theorem}[Theorem
\ref{thm:3Jun19a},
Theorem
\ref{thm:10Sep19b}]
\label{theorem2}
If 
$\fR$ 
is a nc rational function of 
$x_1,\ldots,x_d$ 
over 
$\KK$, 
then for every 
$\ul{Y}=(Y_1,\ldots,Y_d)\in dom_s(\fR)$, 
there exists a unique 
minimal nc Fornasini--Marchesini realization
$\cR$
centred at 
$\ul{Y}$, 
such that 
$dom_{sm}(\fR)=
\Omega_{sm}(\cR)$
for every 
$m\in\NN$ 
and
\begin{multline*}
\fR(\ul{X})
=I_m\otimes D+(I_m\otimes C)
\Big(I_{Lm}-\sum_{k=1}^d 
(X_k-I_m\otimes Y_k)\bA_k\Big)^{-1} \\
\sum_{k=1}^d (X_k-I_m\otimes Y_k)\bB_k
\end{multline*}
for every 
$\ul{X}\in dom_{sm}(\fR)$.
Moreover, for every $n\in\NN$, 
\begin{align*}
dom_n(\fR)=
\big\{\ul{X}\in (\KK^{\ntn})^d:
I_s\otimes\ul{X}\in\Omega_{sn}(\cR)
\big\}
\end{align*}
and
$I_s\otimes\fR(\ul{X})=\cR(I_s\otimes\ul{X})$
for every
$\ul{X}\in dom_n(\fR)$. Furthermore, 
\begin{align*}
dom_{\cA}^{Mat}(\fR)=
\big\{\ul{\fa}\in\cA^d:
I_s\otimes\ul{\fa}\in 
\Omega_{\cA}(\cR)
\big\}
\end{align*}
and $I_s\otimes\fR(\ul{\fa})=\cR^{\cA}(I_s\otimes\ul{\fa})$
for every
$\ul{\fa}\in dom_{\cA}^{Mat}(\fR)$,
whenever $\cA$
ia a stably finite
$\KK-$unital algebra.
\end{introthm}
Here 
$\Omega_{\cA}(\cR)$
is the invertibility set of the corresponding matrix pencil over the algebra $\cA$ 
(see
(\ref{eq:12Apr20c}) 
for the definition), 
$dom_{\cA}^{Mat}(\fR)$
denotes a variation of the domain of the rational function $\fR$ over the algebra 
$\cA$ called the matrix domain (see
(\ref{eq:5Nov19b})
for the precise definition),
and an algebra is called stably finite if every square left invertible matrix over the algebra is also right invertible (see 
Definition
\ref{def:9Mar20a}).
We also present in Corollary
\ref{cor:25Mar21a}
a version of this result using realizations which are independent of a matrix centre, in the spirit of Cohn, Reutenauer and Fliess.

We finish this section by giving a full characterization of all 
the minimal nc Fornasini--Marchesini realizations of a nc rational function,
which
are centred at $\ul{Y}\in dom(\fR)$, 
using the 
$\cL-\cL\cA$ 
conditions:
\begin{introthm}{Theorem}[Theorem
\ref{thm:5Nov19a}]
\label{theorem3}
Let  
$\cR$ 
be a nc Fornasini--Marchesini realization centred at 
$\ul{Y}\in(\KK^{\sts})^d$ 
that is described by 
$(L;D,C,\ul{\bA},\ul{\bB})$, and suppose  $\cR$ is both controllable and
observable. 
\iffalse
given by
\begin{align*}
\cR(X)= D+C
\left(I_{L}-\sum_{i=1}^d 
\bA_i(X_i-Y_i)\right)^{-1}
\sum_{i=1}^d \bB_i(X_i- Y_i).
\end{align*}
\fi
The following are equivalent:
\begin{enumerate}
\item[1.]
There exists a nc rational function 
$\fR\in\KK\plangle \ul{x}\prangle$
regular at 
$\ul{Y}$, 
such that
$dom_{sm}(\fR)=\Omega_{sm}(\cR)$
and 
$\fR(\ul{X})=\cR(\ul{X})$,
for every 
$\ul{X}\in dom_{sm}(\fR)$ 
and 
$m\in\NN$.
\item[2.]
The coefficients of $\cR$
satisfy the 
$\cL-\cL\cA$ 
conditions (cf. equations 
(\ref{eq:12Jun17a})--(\ref{eq:12Jun17d})).
\end{enumerate}
\end{introthm}
There is a different notion of domain for nc rational functions,
the so called extended domain, which is based on evaluations on generic matrices.
It was discovered in
\cite{KV1} and a priori it contains the usual domain of regularity of the
function.
The extended domain by itself is problematic as it is not closed w.r.t direct
sums, but this can be fixed by introducing the stable extended domain
(see 
\cite{V1}). 

In Section
\ref{sec:stable} 
we use the crucial fact
that if the coefficients of a nc Fornasini--Marchesini realization, that
is controllable and observable, satisfy the $\cL-\cL\cA$ conditions, then
the realization defines a
nc function on an upper admissible nc set and we can 
apply to this function the
 difference-differential calculus. 
The main result ---
which is a generalization of 
\cite[Theorem 3.10]{V1}, in which the function is assumed to be regular at some
scalar point --- then
is that the stable extended domain of any nc 
rational function coincides
with its usual domain of regularity:
\begin{introthm}{Theorem}[Corollary 
\ref{cor:10Nov19b}]
For every nc rational function 
$\fR\in\KK\plangle\ul{x}\prangle$, 
we have
$$edom^{st}(\fR)=dom(\fR).$$
Moreover, for every 
$n\in\NN$ 
we have
$$edom^{st}_n(\fR)
=dom_n(\fR)=
\big\{\ul{X}\in(\KK^{\ntn})^d:I_s\otimes
\ul{X}\in\Omega_{sn}(\cR)\big\},$$
whenever 
$\cR$ 
is a minimal realization of 
$\fR$, 
centred at a point from
$dom_s(\fR)$.
\end{introthm}
In the next paper
\cite{PV3}, 
we will  
use the theory of 
realizations with a matrix centre developed in
\cite{PV1}
and in the present paper, 
to
characterize explicitly the ring of 
rational nc generalized power series with matrix centre 
$\ul{Y}$.
More precisely, we will establish a 
generalization of the Fliess-Kronecker 
theorem in which the characterization is 
given in terms of two conditions: the 
$\cL\cA$ conditions 
and a finiteness condition 
on the rank of an infinite Hankel matrix.

As part of the tools 
we will construct a functional model and use it 
to provide a different one step proof 
(based on Hankel realizations, see
\cite{KAF} for the classical case and
\cite{BGM1} for the scalar nc case) 
for the existence of a (minimal) 
realization formula for nc rational functions, 
without using synthesis. 
Furthermore, we present
an explicit construction of the free skew field
$\KK\plangle\ul{x}\prangle$,
with a self-contained proof that it 
is the universal skew field of fractions of 
the ring of nc polynomials. 
\\\\
\textbf{Acknowledgments.}
The authors would like to thank
Dmitry Kalyuzhnyi-Verbovetskyi, Roland Speicher, and Juri Vol\v{c}i\v{c}
for their helpful comments and discussions. 
A special gratitude is due to Igor Klep and Juri Vol\v{c}i\v{c}
for the details of
Example
\ref{ex:10Nov19a}. 
The idea of developing 
realization theory around 
a matrix point, 
along the lines presented here,
was first explored by the 
second author at
a talk at MFO workshop on free 
probability theory in 2015 
(\cite{KV2}); 
the second author would like to 
thank MFO and the workshop 
organizers for their hospitality. 
Finally, we would like to thank the anonymous referee for helpful comment.

\section{Preliminaries}
\label{sec:preliminaries}
\textbf{Notations:} 
$d$ 
will stand for the number of noncommuting variables, 
which will be usually denoted by
$x_1,\ldots,x_d$, we often abbreviate non-commuting by nc.
For a positive integer 
$d$, 
we denote by 
$\cG_d$ 
the free monoid generated by 
$d$ 
generators 
$g_1,\ldots,g_d$,
we say that a word 
$\omega=g_{i_1}\ldots g_{i_\ell}\in\cG_d$ 
is of length 
$|\omega|=\ell$
if $\ell\ge1$ and 
$\omega=\emptyset$ 
is of length 
$0$.
For a field
$\KK$ 
and 
$n\in\NN$, let 
$\KK^{\ntn}$ 
be the vector space of 
$\ntn$ 
matrices over 
$\KK$, 
let 
$\{\ul{e}_1,\ldots,\ul{e}_n\}$ 
be the standard basis of 
$\KK^n$ 
and let 
$\cE_n=\left\{ E_{ij}=\ul{e}_i\ul{e}_j^T: 
1\le i,j\le n\right\}$ 
be the standard basis of 
$\KK^{\ntn}$.
The tensor (Kronecker) product of two matrices 
$P\in\KK^{n_1\times n_2}$
and
$Q\in\KK^{n_3\times n_4}$ 
is  the 
$n_1n_3\times n_2n_4$ 
block matrix 
$P\otimes Q=
\big[p_{ij}Q
\big]_{1\le i\le n_1,1\le j\le n_2}.$  
The range of a matrix 
$P$, 
that is the span of all of its columns, denoted by 
$\Ima(P)$.
For any two square matrices 
$X$ 
and 
$Y$ of the same size, their commutator is defined by $[X,Y]=XY-YX$.

We denote operators on matrices by bold letters such as 
$\bA,\bB$, 
and the action of 
$\bA$ 
on 
$X$ by 
$\bA(X)$. If $\bA$ is defined on 
$\sts$ matrices we extend 
$\bA$ to act on
$sm\times sm$ 
matrices for any 
$m\in\NN$, 
by viewing an
$sm\times sm$ 
matrix 
$X$ 
as an 
$\mtm$ 
matrix with 
$\sts$ 
blocks and by evaluating
$\bA$ 
on the 
$\sts$ 
blocks (cf. equation
(\ref{eq:26May20a})); 
in that case we denote the evaluation by 
$(X)\bA$. 
If 
$C$ 
is a constant matrix and 
$\bA$ 
is an operator, then 
$C\cdot\bA$ 
and 
$\bA\cdot C$ 
are two operators, defined by  
$(C\cdot\bA)(X):=C\bA(X)$ 
and 
$(\bA\cdot C)(X):=\bA(X) C$.
For every 
$n_1,n_2\in\NN$, we define the permutation matrix
$$E(n_1,n_2)=
\big[
E_{ij}^T
\big]_{1\le i\le n_1,1\le j\le n_2}\in\KK^{n_1n_2
\times n_1n_2}$$
and use these matrices to change the order of factors in the Kronecker product of two matrices by the following rule 
\begin{align}
\label{eq:24Jan19a}
P\otimes Q=E(n_1,n_3)(Q\otimes P)E(n_2,n_4)^T,
\end{align}
for all
$n_1,n_2,n_3,n_4\in\NN,\,
Q\in\KK^{n_1\times n_2},$
and
$P\in\KK^{n_3\times n_4}$;
for more details
see 
\cite[pp. 259--261]{HJ}. 
If 
$P=
\big[
P_{ij}\big]_{1\le i,j\le m},Q=
\big[
Q_{ij}\big]_{1\le i,j\le m}
\in(\KK^{\sts})^{\mtm}$, then we use the notation
$$P\odot_s Q:=
\Big[
\sum_{k=1}^m P_{ik}\otimes Q_{kj}\Big]_{1\le i,j\le m}$$
for the so-called faux product of 
$P$ and $Q$,
viewed as
$\mtm$ matrix over the tensor algebra of 
$\KK^{\sts}$, 
see 
\cite{Ef} 
for its origins in operator spaces. 
If 
$\ul{X}=(X_1,\ldots, X_d)\in(\KK^{sm\times sm})^d$ 
and 
$\omega=g_{i_1}\ldots g_{i_{\ell}}\in\cG_d$, 
then
\begin{align}
\label{eq:21Mar19a}
\ul{X}^{\odot_s\omega}:=X_{i_1}\odot_s\cdots\odot_s X_{i_{\ell}}.
\end{align}
Recall that 
$\ul{X}$ 
is called jointly nilpotent, if there exists 
$\kappa\in\NN$ such that
$\ul{X}^{\odot_s\omega}=0$ for every $\omega\in\cG_d$ satisfying $|\omega|\ge\kappa$.
For every matrix
$Z$ over
$\KK$
and a unital
$\KK-$algebra
$\cA$, we 
use the notation
$Z_{\cA}:=Z\otimes \bf{1}_{\cA}$, 
where
$\bf{1}_{\cA}$
is the unit element in 
$\cA$. 
Notice that if 
$\ul{X}^{i,j}=(X^{i,j}_1,
\ldots,X^{i,j}_d)\in\big(\KK^{n_i\times n_j}\big)^d$ 
for 
$1\le i,j\le 2$, we use the notation
\begin{align*}
\begin{pmatrix}
\ul{X}^{1,1}&\ul{X}^{1,2}\\
\ul{X}^{2,1}&\ul{X}^{2,2}\\
\end{pmatrix}:=
\left(\begin{bmatrix}
X^{1,1}_1&X^{1,2}_1\\
X^{2,1}_1&X^{2,2}_1\\
\end{bmatrix},\ldots,
\begin{bmatrix}
X^{1,1}_d&X^{1,2}_d\\
X^{2,1}_d&X^{2,2}_d\\
\end{bmatrix}\right)
\in\Big(\KK^{(n_1+n_2)\times(n_1+n_2)}\Big)^d.
\end{align*}
We use
$\fR,\cR,R$, and
$\fr$ 
for nc rational function, nc 
Fornasini--Marchesini realization, nc
rational expression, and matrix valued 
nc rational function, respectively.
Likewise, we use 
$\fa$ 
to denote elements in an algebra 
$\cA$
and
$\fA$ 
to denote matrices over 
$\cA$. 
All over the paper, we use underline 
to denote vectors or 
$d-$tuples.
\subsection{NC rational expressions \& functions}
\label{sec:NCREF} 
Let 
$\KK$ 
be a field, 
$d\ge1$ 
an integer, and 
$x_1,\ldots,x_d$ 
noncommuting variables. 
 We denote by
$\KK\langle x_1,\ldots,x_d\rangle$ 
the 
$\KK-$algebra of nc polynomials in the 
$d$ 
nc variables 
$x_1,\ldots,x_d$ 
over 
$\KK$.
We obtain nc rational expressions by applying 
successive arithmetic operations 
(addition, multiplication and taking inverse)
on 
$\KK\langle x_1,\ldots,x_d\rangle$.

For a nc rational expression 
$R$ 
and 
$n\in\NN$, 
the set 
$dom_n(R)$
consists of all 
$d-$tuples of 
$\ntn$ 
matrices over 
$\KK$ 
for which all the inverses in 
$R$ 
exist. 
The \textbf{domain of regularity} of 
$R$ 
is then defined by
\begin{align*}
dom(R):=
\coprod_{n=1}^\infty dom_n(R).
\end{align*}
For example, 
$R(x_1,x_2)=(x_1x_2-x_2x_1)^{-1}$
is a nc rational expression in 
$x_1,x_2$, with 
$dom_1(R)=\emptyset$ 
and 
$dom_n(R)=\big\{ (X_1,X_2)\in(\KK^{\ntn})^2:
\det(X_1X_2-X_2X_1)\ne0\big\}\ne\emptyset$
for any $n>1$.

A nc rational expression 
$R$
is called \textbf{non-degenerate} if
$dom(R)\ne\emptyset$ 
(for instance the nc rational expression 
$(x_1-x_1)^{-1}$ is degenerate).
Let 
$R_1$ 
and 
$R_2$ 
be nc rational expressions in 
$x_1,\ldots,x_d$ 
over 
$\KK$. 
We say that 
$R_1$ 
and 
$R_2$ 
are 
$(\KK^d)_{nc}-$\textbf{evaluation equivalent},
if
$R_1(\ul{X})=R_2(\ul{X})$
for every 
$\ul{X}\in 
dom(R_1)\cap dom(R_2)$.

For example the nc rational expressions 
$R_1(x_1,x_2)=x_2(x_1x_2)^{-1}x_1+x_1x_2,\,
R_2(x_1,x_2)=x_2^{-1}(x_2+x_2x_1x_2)$
and
$R_3(x_1,x_2)=1+x_1x_2$,
are 
$(\KK^3)_{nc}-$evaluation equivalent. 
\begin{definition}[NC Rational Functions]
\label{def:21Apr20c} 
A \textbf{nc rational function} 
of 
$x_1,\ldots,x_d$ 
over 
$\KK$
is an 
equivalence class of non-degenerate nc 
rational expressions, w.r.t the 
$(\KK^d)_{nc}-$evaluation equivalence relation. 
For every nc rational function 
$\fR$ of $x_1,\ldots,x_d$, 
define its 
\textbf{domain of regularity}
\begin{align}
\label{eq:9Mar19a}
dom(\fR):=\bigcup_{R\in\fR} dom(R).
\end{align}
For every 
$\ul{X}\in dom(\fR)$, 
we say that $\fR$
is \textbf{regular} at 
$\ul{X}$ and
its evaluation is given by 
$\fR(\ul{X})=R(\ul{X})$
for every nc rational expression  
$R\in\fR$ 
such that 
$\ul{X}\in dom(R)$.
\end{definition}
The 
$\KK-$algebra of all nc rational functions of 
$x_1,\ldots,x_d$ 
over 
$\KK$ is denoted by 
$\KK\plangle x_1,\ldots,x_d\prangle$ 
and it is a skew field, called the 
\textbf{free 
skew field}. Moreover,
$\KK\plangle x_1,\ldots, x_d\prangle$  
is the universal skew field of fractions of 
$\KK\langle x_1,\ldots,x_d\rangle$. 
See 
\cite{AM66,Be70,Co71a,Co72,Row80} 
for the original proofs and  
\cite{Co06} 
for a more modern reference, 
while a proof of the equivalence with the 
evaluations over matrices is presented in 
\cite{KV3,KV4}. 
\\\\
\textbf{$\cA-$Domains and Evaluations.}
Let
$\cA$ be a unital 
$\KK-$algebra. If
$\ul{\fa}=(\fa_1,\ldots,\fa_d)\in\cA^d$ 
and 
$\omega=g_{i_1}\ldots g_{i_\ell}\in\cG_d$, 
then we use the notations 
$\ul{\fa}^\omega:=\fa_{i_1}\cdots 
\fa_{i_\ell}$ 
and
$\ul{\fa}^{\emptyset}=\bf{1}_{\cA}$, 
where 
$\bf{1}_{\cA}$ 
is the unit element in 
$\cA$. 
Evaluations and 
domains of nc rational expressions over 
$\cA$ are defined in a natural way, see 
\cite{HMS}
or
\cite[subsection 1.2]{PV1}  
for more details.
We are interested in a certain family of algebras, called stably finite algebras, also called weakly finite in
\cite{Co06}. 
\begin{definition}
\label{def:9Mar20a}
A unital 
$\KK-$algebra 
$\cA$ 
is called 
\textbf{stably finite} 
if for every 
$m\in\NN$ 
and 
$\fA,\fB\in\cA^{\mtm}$, we have
$\fA\fB=I_m\otimes 
{\bf{1}}_{\cA}$
if and only if
$\fB\fA=I_m\otimes 
\bf{1}_{\cA}$.
\end{definition}
\subsection{Realization theory around a matrix centre}
\label{subsec:prelimFirstPaper}
We summarize now the main definitions and 
main results from 
\cite{PV1} about realizations of nc rational 
expressions and nc rational functions 
around a matrix centre. 
Notice we denote here by 
$\Omega_{sm}(\cR)$ 
what was denoted in 
\cite{PV1} 
by 
$DOM_{sm}(\cR)$. 
\begin{definition}[\cite{PV1}, Definition 2.1]
\label{def:13Apr19a}
Let 
$s\in\NN,\,
L\in\NN,\,
\ul{Y}=(Y_1,\ldots,Y_d)\in(\KK^{\sts})^d,\,
\bA_1,\ldots,\bA_d:\KK^{\sts}
\rightarrow\KK^{L\times L}$
and
$\bB_1,\ldots,\bB_d:\KK^{\sts}
\rightarrow\KK^{L\times s}$
be linear mappings, 
$C\in\KK^{s\times L}$, 
and
$D\in\KK^{\sts}$. 
Then
\begin{equation}
\label{eq:18Aug18a}
\cR(X_1,\ldots,X_d)= 
D+C\Big(I_{L}-\sum_{k=1}^d 
\bA_k(X_k-Y_k)\Big)^{-1}
\sum_{k=1}^d 
\bB_k(X_k- Y_k)
\end{equation}
is called a 
\textbf{nc Fornasini--Marchesini realization} 
centred at 
$\ul{Y}$
and it is defined for every 
$\ul{X}=(X_1,\ldots,X_d)\in \Omega_s(\cR)$,
where
\begin{equation*}
\Omega_s(\cR):=\Big\{ 
\ul{X}\in(\KK^{\sts})^d:\det
\Big(
I_L-\sum_{k=1}^d
\bA_k(X_k-Y_k)
\Big)\ne0
\Big\}.
\end{equation*}
\end{definition}
In that case we say that the realization 
$\cR$ 
is \textbf{described by} 
$(L;D,C,\ul{\bA},\ul{\bB})$ 
and the corresponding generalized 
linear pencil centred at 
$\ul{Y}$ is
\begin{align}
\label{eq:9Mar20c}
\Lambda_{\ul{\bA},\ul{Y}}(\ul{X}):=I_L-\sum_{k=1}^d
\bA_k(X_k-Y_k).
\end{align}
Let
$s_1,s_2,s_3,s_4\in\NN$. 
If 
$\bT:\KK^{s_1\times s_2}
\rightarrow
\KK^{s_3\times s_4}$ 
is a linear mapping and 
$m\in\NN$, 
then 
$\bT$ 
can be naturally extended to a linear mapping 
$\bT:\KK^{s_1m\times s_2m}
\rightarrow
\KK^{s_3m\times s_4m}$, 
by the following rule:
\begin{equation}  
\label{eq:26May20a}
X=
\big[
X_{ij}\big]_{1\le i,j\le m}
\in\KK^{s_1m\times s_2m}
\Longrightarrow
(X)\bT=
\big[
\bT(X_{ij})
\big]_{1\le i,j\le m},
\end{equation}
i.e., 
$(X)\bT$  
is an 
$\mtm$ 
block matrix with entries in 
$\KK^{s_3\times s_4}$.
Therefore, we can
extend the pencil 
$\Lambda_{\ul{\bA},\ul{Y}}$ 
to act on 
$d-$tuples of $sm\times sm$ matrices,
by
$$\Lambda_{\ul{\bA},\ul{Y}}
(\ul{X})=
I_{Lm}-\sum_{k=1}^d 
(X_k-I_m\otimes Y_k)\bA_k,\, 
\forall\ul{X}=(X_1,\ldots,X_d)
\in(\KK^{sm\times sm})^d,$$
hence we
 extend the realization
(\ref{eq:18Aug18a}) 
to act on 
$d-$tuples of 
$sm\times sm$ 
matrices: for every 
\begin{equation*}
\ul{X}=(X_1,\ldots,X_d)\in
\Omega_{sm}(\cR):=
\Big\{\ul{X}
\in(\KK^{sm\times sm})^d:
\det\big(\Lambda_{\ul{\bA},\ul{Y}}(\ul{X})\big)\ne0
\Big\},
\end{equation*}
we define
\begin{equation*}
\cR(\ul{X}):=
I_m\otimes D+
(I_m\otimes C)
\Lambda_{\ul{\bA},\ul{Y}}(\ul{X})^{-1}
\sum_{k=1}^d 
(X_k-I_m\otimes Y_k)\bB_k.
\end{equation*}
In addition, if
$\cA$
is a unital 
$\KK-$algebra, a linear mapping
$\bT:\KK^{s_1\times s_2}\rightarrow
\KK^{s_3\times s_4}$
can be also naturally extended to a linear mapping
$\bT^{\cA}:\cA^{s_1\times s_2}
\rightarrow\cA^{s_3\times s_4}$, 
by the following rule:
\begin{equation}
\label{eq:9Apr20b}
\fA
=\sum_{i=1}^{s_1}
\sum_{j=1}^{s_2} 
E_{ij}\otimes \fa_{ij}
\in\cA^{s_1\times s_2}
\Longrightarrow (\fA)\bT^{\cA}
=\sum_{i=1}^{s_1}
\sum_{j=1}^{s_2}
\bT(E_{ij})
\otimes \fa_{ij}\in
\cA^{s_3\times s_4},
\end{equation}
where
$E_{ij}=\ul{e}_i\ul{e}_j^T
\in\KK^{s_1\times s_2}$
and
$\fa_{ij}\in\cA$.
Thus we extend the pencil 
$\Lambda_{\ul{\bA},\ul{Y}}$ to act on $d-$tuples of $\sts$ matrices over $\cA$, by 
$$\Lambda_{\ul{\bA},\ul{Y}}^{\cA}
(\ul{\fA})=
I_L\otimes {\bf1}_{\cA}-\sum_{k=1}^d
(\fA_k-Y_k\otimes {\bf{1}}_{\cA})
\bA_k^{\cA},\,
\forall\ul{\fA}=(\fA_1,\ldots,
\fA_d)\in(\cA^{\sts})^d.$$
The 
\textbf{$\cA-$domain} 
of a nc Fornasini--Marchesini realization 
$\cR$ centred at 
$\ul{Y}$, 
as in 
(\ref{eq:18Aug18a}), 
is then defined to be the subset of 
$(\cA^{\sts})^d$
given by
\begin{equation}
\label{eq:12Apr20c}
\Omega_{\cA}(\cR)
:=\Big\{
\ul{\fA}\in(\cA^{\sts})^d:
\Lambda^{\cA}_{\ul{\bA},\ul{Y}}(\ul{\fA})
\text{ is invertible in }
\cA^{L\times L}\Big\}
\end{equation}
and for every 
$\ul{\fA}=(\fA_1,\ldots,\fA_d)\in 
\Omega_{\cA}(\cR)$, 
the evaluation of 
$\cR$ 
at 
$\ul{\fA}$
is defined by 
\begin{align}
\label{eq:12Apr20d}
\cR^{\cA}(\ul{\fA}):=
D\otimes {\bf1}_\cA
+(C\otimes {\bf1}_\cA)
\Lambda^{\cA}_{\ul{\bA},\ul{Y}}(\ul{\fA})^{-1}
\sum_{k=1}^d 
\big[(\fA_k)\bB_k^{\cA}-
\bB_k(Y_k)\otimes {\bf1}_\cA
\big].
\end{align}
\begin{definition}[\cite{PV1}, Definition 2.3]
\label{def:25Sep18a}
Let 
$R$ 
be a nc rational expression in
$x_1,\ldots,x_d$ 
over 
$\KK,\,\ul{Y}
\in dom_s(R)$,
$\cR$ 
be a nc Fornasini--Marchesini realization centred at
$\ul{Y}$, 
and 
$\cA$ 
be a unital
$\KK-$algebra. 
We say that:
\\1.
$R$ 
\textbf{admits the realization} 
$\cR$, 
or that 
$\cR$ 
is a realization of 
$R$, 
if 
\begin{align*}
dom_{sm}(R)\subseteq \Omega_{sm}(\cR)
\text{ and }R(\ul{X})=\cR(\ul{X}),\,
\forall \ul{X}\in dom_{sm}(R),
\, m\in\NN.
\end{align*} 
2.
$R$
\textbf{admits the realization 
$\cR$ 
w.r.t}
$\cA$,
or that
$\cR$ 
is a realization of 
$R$ 
w.r.t 
$\cA$, 
if
\begin{multline*}
dom_{\cA}(R)\subseteq
\{\ul{\fa}\in\cA^d: 
I_s\otimes 
\ul{\fa}\in\Omega_{\cA}(\cR) \}
\text{ and } 
\\ I_s\otimes R^{\cA}(\ul{\fa})
=\cR^{\cA}(I_s\otimes\ul{\fa}),\,
\forall\ul{\fa}\in dom_{\cA}(R).
\end{multline*}
\end{definition}
\begin{definition}[\cite{PV1}, Definition 2.7]
\label{def:25Jan19a}
Let
$\bA_1,\ldots,\bA_d:\KK^{\sts}
\rightarrow\KK^{L\times L}$
and
$\bB_1,\ldots,\bB_d:
\KK^{\sts}\rightarrow\KK^{L\times s}$ 
be linear mappings, and 
$C\in\KK^{s\times L}$. 
\\1. $(\ul{\bA},\ul{\bB})$ 
is called 
\textbf{controllable}, if
$$\bigvee_{\omega\in\cG_d,\,
X_1,\ldots,X_{|\omega|+1}
\in\KK^{s\times s},\,
1\le k\le d} 
\Ima
\left(\ul{\bA}^\omega
(X_1,\ldots,
X_{|\omega|})
\bB_k
(X_{|\omega|+1})
\right)=\KK^L.$$
\\2. $(C,\ul{\bA})$ 
is called \textbf{observable},
if
$$\bigcap_{\omega\in\cG_d,\,
X_1,\ldots,X_{|\omega|}
\in\KK^{s\times s}}
\ker \left(C\ul{\bA}^\omega
\left(X_1,\ldots,
X_{|\omega|}\right)\right)=\{\ul{0}\}.$$
\end{definition}
If
$\cR$
is a nc Fornasini--Marchesini realization of a nc rational expression
$R$,
that is centred at $\ul{Y}$, 
then it is said to be
\textbf{minimal}
if the dimension 
$L$ 
is the smallest integer for which 
$R$ 
admits such 
a realization, i.e., if 
$\cR^\prime$
is a nc Fornasini--Marchesini realization of 
$R$ 
centred at 
$\ul{Y}$ 
of dimension 
$L^\prime$, 
then 
$L\le L^\prime$. 
\begin{theorem}[\cite{PV1}, Theorem 2.16]
\label{thm:minObsCont}
Let 
$R$
be a nc rational expression in 
$x_1,\ldots,x_d$ 
over 
$\KK$ 
and
$\cR$
be a nc Fornasini--Marchesini realization of
$R$
centred at 
$\ul{Y}\in(\KK^{\sts})^d$.
Then 
$\cR$ 
is minimal if and only if 
$\cR$ is controllable and observable. 
\end{theorem}
Let 
$\cR_1$
and
$\cR_2$
be two nc Fornasini--Marchesini realizations,
described by 
$(L_1; D^1,C^1,\ul{\bA}^1,\ul{\bB}^1)$ and
$(L_2; D^2,C^2,\ul{\bA}^2,\ul{\bB}^2)$,
respectively, both centred at 
$\ul{Y}\in(\KK^{\sts})^d$.  
Then $\cR_1$ 
and 
$\cR_2$
are said to be \textbf{uniquely similar}, if
$L_1=L_2,\,D^1=D^2$,
and there exists a unique invertible matrix 
$T\in\KK^{L_1\times L_1}$ 
such that 
\begin{align*}
C^2=C^1T^{-1},\, 
\bB^2_k=T\cdot\bB^1_k, 
\text{ and }\bA_k^2=
T\cdot\bA_k^1\cdot T^{-1},\,
1\le k\le d.
\end{align*}
\begin{cor}[\cite{PV1}, Corollary 2.18]
\label{cor:ExistAndUnique}
Let
$R$ 
be a nc rational expression in 
$x_1,\ldots,x_d$ 
over 
$\KK$ 
and 
$\ul{Y}\in dom_s(R)$.
Then
$R$ 
admits a unique (up to unique similarity) 
minimal nc Fornasini--Marchesini realization 
centred at 
$\ul{Y}$, 
that is also a realization of 
$R$ 
w.r.t any 
unital stably finite $\KK-$algebra.
\end{cor}
\begin{theorem}[\cite{PV1}, Theorem 3.3]
\label{thm:MainThmFirstPaper}
Let 
$\fR\in\KK\plangle x_1,\ldots,x_d\prangle$. 
For every
two integers
$s,n\in\NN$,
a point 
$\ul{Y}\in dom_{s}(\fR)$,
a minimal nc Fornasini--Marchesini realization 
$\cR$ 
centred at 
$\ul{Y}$ 
of 
$\fR$,
and a unital stably finite 
$\KK-$algebra
$\cA$,
we have the following properties:
\begin{itemize}
\item[1.]
$dom_n(\fR)\subseteq
\{\ul{Z}\in(\KK^{\ntn})^d:
I_s\otimes \ul{Z}\in \Omega_{sn}(\cR)\}$
and
$I_s\otimes
\fR(\ul{Z})=\cR(I_s\otimes \ul{Z}),\,\forall\ul{Z}\in dom_n(\fR).$
\item[2.]
If 
$s\mid n$,  
then
$dom_{n}(\fR)\subseteq \Omega_{n}(\cR)$ 
and 
$\fR(\ul{Z})=\cR(\ul{Z}),\,
\forall 
\ul{Z}\in dom_{n}(\fR)$.
\item[3.]
$dom_{\cA}(\fR)
\subseteq
\{ \ul{\fa}\in\cA^d:
I_s\otimes\ul{\fa}\in \Omega_{\cA}(\cR)\}$
and 
$I_s\otimes
\fR^{\cA}(\ul{\fa})
=\cR^{\cA}(I_s\otimes\ul{\fa}),
\,\forall\ul{\fa}\in dom_{\cA}(\fR).$
\end{itemize}
\end{theorem}
\subsection{NC functions}
\label{sec:NCfunctions}
The definitions and results mentioned in this subsection are taken from the book 
\cite{KV1}. We do not quote those in their full generality, as in the book they appear in the framework of a module  
over a commutative ring, 
while here we consider the framework of a vector space 
$\cV$ 
over a field 
$\KK$. 
Moreover, we will mostly consider the vector space 
$\cV=\KK^d$ 
for an integer 
$d\in\NN$.
\smallskip

If 
$\cV$ 
is a vector space over a field
$\KK$, 
then  
$\cV_{nc}:=\coprod_{n=1}^\infty 
\cV^{\ntn}$
is called the nc space over 
$\cV$ 
and consists of all square matrices over 
$\cV$.
For every 
$X\in\cV^{\ntn}$
and
$Y\in\cV^{\mtm}$, 
where
$m,n\in\NN$, we define their direct sum as
$$X\oplus Y:=
\begin{bmatrix}
X&0\\0&Y\\
\end{bmatrix}\in\cV^{(n+m)\times(n+m)}.$$ 
For every 
$\Omega\subseteq \cV_{nc}$ 
and
$n\in\NN$ 
we use the notation 
$\Omega_n:=\Omega\cap \cV^{\ntn}$.
\begin{definition}
\label{def:12May20a}
1. A subset $\Omega\subseteq\cV_{nc}$
is called a \textbf{nc set} if it is 
closed under direct sums, 
i.e., if
$X\in\Omega_n$
and
$Y\in\Omega_m$,
then 
$X\oplus Y\in \Omega_{n+m}$,
for every two integers $n,m\in\NN$. 
\\2.
A nc set $\Omega\subseteq\cV_{nc}$ 
is called \textbf{upper} 
(resp., \textbf{lower}) \textbf{admissible} 
if for every
$n,m\in\NN,\,
X\in\Omega_n,\,
Y\in\Omega_m$, 
and 
$Z\in\cV^{n\times m}$
(resp., $Z\in\cV^{\mtn}$), there exists
$0\ne c\in\KK$ such that 
$$\begin{bmatrix}
X&cZ\\0&Y\\
\end{bmatrix}\in \Omega_{n+m}\,
\Big(resp.,\,
\begin{bmatrix}
X&0\\
cZ&Y\\
\end{bmatrix}\in\Omega_{n+m}\Big).$$
3.
A subset
$\Omega\subseteq\cV_{nc}$
is called \textbf{similarity invariant}, 
if for every 
$n\in\NN,\,X\in\Omega_n$, and invertible 
$T\in\KK^{\ntn}$, 
we have 
$T\cdot X\cdot T^{-1}\in\Omega_n$.
\end{definition}
Notice that if 
$X\in\cV^{\ntn}$ 
and 
$T\in\KK^{\ntn}$, 
by the products 
$T\cdot X$ 
and 
$X\cdot T$ 
we mean the standard matrix multiplication and we use the action of 
$\KK$ 
on 
$\cV$. 
In the special and most common case where 
$\cV=\KK^d$, 
we have the identification
\begin{equation*}
\big(\KK^d\big)_{nc}
=\coprod_{n=1}^\infty
\big(\KK^d\big)^{\ntn}\cong
\coprod_{n=1}^\infty 
\big(\KK^{\ntn}\big)^d,
\end{equation*}
that is the nc space of all 
$d-$tuples of square matrices over 
$\KK$.
Thus, in this particular case, for every 
$\ul{X}=(X_1,\ldots,X_d)\in(\KK^{\ntn})^d$
and
$T\in\KK^{\ntn}$, 
the products
$T\cdot\ul{X}$ and $\ul{X}\cdot T$ are given by
\begin{equation*}
T\cdot\ul{X}:=(TX_1,\ldots,TX_d)
\text{ and }\ul{X}\cdot 
T:=(X_1T,\ldots,X_dT).
\end{equation*}
For every nc rational expression $R$
in $x_1,\ldots,x_d$ over 
$\KK$, the domain
of  regularity of $R$ 
is an upper admissible, similarity invariant nc set. We will show 
eventually that the domain of any 
nc rational function of 
$x_1,\ldots,x_d$ over 
$\KK$ is an upper admissible, similarity invariant nc set as well
(cf. Corollary
\ref{cor:12May20a}). 
If
$\KK=\CC$, 
the domains of nc rational expressions or functions are open in the uniformly-open topology (cf. 
\cite[Chapter 4.2]{KV1} 
for a discussion on the uniformly-open topology and
\cite[Lemma A.5]{HMcCV}
for the proof in the case where regularity at $0$ is assumed).

Another important example of a similarity invariant, upper admissible nc set is a nilpotent ball:
for every 
$s\in\NN$ 
and 
$\ul{Y}=(Y_1,\ldots,Y_d)\in(\KK^{\sts})^d$, 
we define the \textbf{nilpotent ball around} 
$\ul{Y}$ as follows 
\begin{align}
Nilp(\ul{Y}):=\coprod_{m=1}^\infty
Nilp_{sm}(\ul{Y})\subseteq (\KK^d)_{nc},
\end{align}
where
\begin{align}
Nilp_{sm}(\ul{Y})
:=\big\{
\ul{X}\in(\KK^{sm\times sm})^d: 
(\ul{X}-I_m\otimes \ul{Y})
\text{ is jointly nilpotent}
\big\}
\end{align}
for every 
$m\in\NN$.
\begin{definition}
\label{def:21Apr20e}
Let 
$\cV$
and
$\cW$ 
be vector spaces over a field 
$\KK$ 
and 
$\Omega\subseteq\cV_{nc}$ 
be a nc set, then 
$f:\Omega\rightarrow\cW_{nc}$ 
is called a
\textbf{nc function}, if
\\1. $f$ 
is graded, i.e., if
$n\in\NN$
and
$X\in\Omega_n$, 
then
$f(X)\in\cW^{\ntn}$;
\\2.
$f$ 
respects direct sums, i.e., if 
$X,Y\in\Omega$, 
then
$f(X\oplus Y)=f(X)\oplus f(Y)$; and
\\3.
$f$ respects similarities, i.e.,  
if
$n\in\NN
,\,X\in\Omega_n$,
and
$T\in\KK^{\ntn}$ 
is invertible such that 
$T\cdot X\cdot T^{-1}\in\Omega_n$, 
then
$f(T\cdot X\cdot T^{-1})=
T\cdot f(X)\cdot T^{-1}$.
\end{definition}
Conditions $2$ and
$3$ in Definition
\ref{def:21Apr20e} are equivalent to a single one: 
$f$ 
{\bf respects intertwining}, namely if 
$XS=SY$, 
then $f(X)S=Sf(Y)$, where
$X\in\Omega_n,\,
Y\in\Omega_m,$
and
$S\in\KK^{\ntm}$ ($n,m\in\NN$). 

Every nc rational expression 
$R$ in $x_1,\ldots,x_d$ over $\KK$
is a nc function on
$dom(R)$, with $\cV=\KK^d$ and $\cW=\KK$.
It is also true that for any nc rational function 
$\fR$ of 
$x_1,\ldots,x_d$ over
$\KK$,
$\fR\restriction_{\Omega}$ 
is a nc function on 
$\Omega$, 
for every nc set 
$\Omega\subseteq dom(\fR)$.
For instance we can take
$\Omega$
to be the nilpotent ball around a point 
$\ul{Y}\in dom(\fR)$. 
Since we will show later that 
$dom(\fR)$
is itself an nc set,
it follows also that 
$\fR$
is a nc function on
$dom(\fR)$ 
(cf. Corollary
\ref{cor:12May20a}). 
\begin{remark}
If 
$\Omega\subseteq\cV_{nc}$ 
is a nc set, then by 
$\wt{\Omega}$ 
we denote the smallest nc set that contains
$\Omega$ 
and that is similarity invariant; 
$\wt{\Omega}$
is called the \textbf{similarity invariant envelope}
of 
$\Omega$.
If $\Omega$ is upper admissible, then so is 
$\wt{\Omega}$. Another important result that we will use later on is the following: 
if
$\Omega\subseteq\cV_{nc}$
is a nc set and 
$f:\Omega\rightarrow\cW_{nc}$ is a nc function, then there exists
a unique nc function 
$\wt{f}:\wt{\Omega}\rightarrow\cW_{nc}$
such that  
$\wt{f}\mid_{\Omega}=f$
(cf.
\cite[Appendix A]{KV1}). 
\end{remark}
It is shown in
\cite[Theorem 5.8]{KV1}
that every nc function
$f$ on an upper admissible nc set containing the point 
$\ul{Y}\in(\KK^{\sts})^d$,
admits a Taylor--Taylor power series expansion around
$\ul{Y}$, i.e., an expansion of the form
\begin{align}
\label{eq:12May20c}
f(\ul{X})=\sum_{\omega\in\cG_d}
(\ul{X}-I_m\otimes \ul{Y}
)^{\odot_s\omega}f_\omega,
\end{align}
where for every $\omega\in\cG_d$, 
$f_{\omega}:(\KK^{\sts})^{|\omega|}\rightarrow
\KK^{\sts}$
is a 
$|\omega|-$linear mapping, or alternatively a linear mapping from 
$\big(\KK^{\sts}\big)^{\otimes|\omega|}$
to
$\KK^{\sts}$. 
Notice that 
$$(\ul{X}-I_m\otimes 
\ul{Y})^{\odot_s\omega}\in 
\Big(\big(\KK^{\sts}\big)^{\otimes|\omega|}
\Big)^{\mtm},$$
hence we can apply 
$f_{\omega}$
to every entry of this matrix yielding 
a matrix in 
$\big(\KK^{\sts}\big)^{\mtm}\cong 
\KK^{sm\times sm}$,
 which is where the value 
$f(\ul{X})$ 
lies; this is how we extend 
$f_{\omega}$ to a mapping from
$(\KK^{sm\times sm})^{|\omega|}$
into $\KK^{sm\times sm}$.
The equality in 
(\ref{eq:12May20c}) holds for every 
$\ul{X}\in Nilp(\ul{Y})$, 
as this ensures that the series in 
(\ref{eq:12May20c}) is actually finite.

One important difference with the case 
of a scalar centre ($s=1$) is that the coefficients 
$(f_{\omega})_{\omega\in\cG_d}$ 
are not arbitrary multilinear mappings,  
as they have to satisfy certain compatibility 
conditions w.r.t
$\ul{Y}$, 
called the
\textbf{lost-abbey ($\cL\cA$) conditions} 
(see 
\cite[equations (4.14)-(4.17)]{KV1}):
\begin{align}
\label{eq:7May18a}
Sf_\emptyset-f_\emptyset S=
\sum_{k=1}^d 
f_{g_k}
([S,Y_k])
\end{align}
and for every 
$\omega=g_{i_1}\ldots 
g_{i_\ell}\ne\emptyset$
in
$\cG_d:$
\begin{align}
\label{eq:7May18b}
Sf_\omega
(Z_1,\ldots,Z_\ell)
-f_\omega(SZ_1,\ldots,Z_\ell)
=
\sum_{k=1}^d
f_{g_k\omega}([S,Y_k],Z_1,\ldots,
Z_\ell),
\end{align}
\begin{align}
\label{eq:14Apr20e}
f_\omega(Z_1,\ldots,Z_\ell S)
-
f_\omega(Z_1,\ldots,Z_\ell) S=
\sum_{k=1}^d 
f_{\omega g_k}(Z_1,\ldots,Z_\ell,[S,Y_k])
\end{align}
and
\begin{multline}
\label{eq:7May18c}
f_\omega(Z_1,\ldots,Z_{j-1},Z_jS,Z_{j+1},
\ldots,Z_\ell)-
f_\omega(Z_1,\ldots,Z_j,
SZ_{j+1},\ldots,Z_\ell)
\\=\sum_{k=1}^d 
f_{g_{i_1}\ldots g_{i_j}
g_kg_{i_{j+1}}\ldots g_{i_\ell}}
(Z_1,\ldots,Z_j,[S,Y_k],Z_{j+1},\ldots,Z_\ell),
\end{multline}
for every $Z_1,\ldots,Z_{\ell}\in\KK^{\sts},\,
1\le j<\ell$, 
and
$S\in\KK^{\sts}$.
Conversely, in
\cite[Theorem 5.1.5]{KV1} it is shown
that given a sequence of multilinear mappings
$(f_{\omega})_{\omega\in\cG_d}$ 
which satisfy the $\cL\cA$ conditions, the power series in
(\ref{eq:12May20c})
is a nc function on
the nilpotent ball around
$\ul{Y}$.
\begin{remark}
\label{rem:16May20ab}
If a nc function
$f$ is locally bounded, then its Taylor--Taylor series
around
$\ul{Y}$, as in
(\ref{eq:12May20c}), 
converges (absolutely and uniformly) on 
a neighborhood of
$\ul{Y}$ in an appropriate topology,
see
\cite[Corollary 7.5]{KV1}
\end{remark}

\section{The Linearized Lost-Abbey Conditions}
\label{sec:la}
One of the purposes of this section is to reformulate the 
$\cL\cA$ 
conditions in the case of 
\textbf{nc rational functions}, in terms of the coefficients of a minimal
nc Fornasini--Marchesini realization of the function.
\begin{definition}
\label{def:14Apr20a}
Let 
$f$ be a generalized nc power series around
$\ul{Y}$
of the form
(\ref{eq:12May20c})
and let 
$\cR$ be a
nc Fornasini--Marchesini
realization 
that is centred at
$\ul{Y}$ 
and described by
$(L;D,C,\ul{\bA},\ul{\bB})$.
We say that \textbf{the series
$f$
admits the realization}
$\cR$,
if
$$f_{\emptyset}=D
\text{ and }f_{\omega}(Z_1,\ldots,Z_{\ell})
=C\bA_{i_1}(Z_1)\cdots\bA_{i_{\ell-1}}
(Z_{\ell-1})\bB_{i_\ell}(Z_{\ell}),$$
for every
$\emptyset\ne\omega=g_{i_1}
\ldots g_{i_\ell}\in\cG_d$
and
$Z_1,\ldots,Z_{\ell}\in\KK^{\sts}$. 
This is equivalent
to say that 
$f(\ul{X})=\cR(\ul{X})$ 
for every 
$\ul{X}\in Nilp(\ul{Y})$.
\end{definition}
If 
$\fR$ 
is a nc rational function that is regular at 
$\ul{Y}=(Y_1,\ldots,Y_d)\in(\KK^{\sts})^d$, 
then
Corollary
\ref{cor:ExistAndUnique}
and Theorem 
\ref{thm:MainThmFirstPaper}
guarantee that 
$\fR$ 
admits a minimal nc 
Fornasini--Marchesini realization 
$\cR$
that is centred at 
$\ul{Y}$
and described by
$(L;D,C,\ul{\bA},\ul{\bB})$. 
Then, a direct computation 
(cf. 
\cite[Lemma 2.12]{PV1}) shows that 
the
Taylor--Taylor power series expansion of 
$\fR$ 
around
$\ul{Y}$ 
is 
\begin{align}
\label{eq:5Ma17a}
\fR(\ul{X})=
\sum_{\omega\in\cG_d}
(\ul{X}-I_m\otimes \ul{Y}
)^{\odot_s\omega}\fR_\omega,
\,\ul{X}\in(\KK^{sm\times sm})^d
\end{align} 
where the coefficients 
$\fR_\omega:(\KK^{\sts})^{|\omega|}
\rightarrow \KK^{\sts}$ 
are the multilinear mappings given by 
\begin{align}
\label{eq:14Apr20c}
\fR_{\omega}(Z_1,\ldots,Z_{\ell})=
C\bA_{i_1}(Z_1)\cdots
\bA_{i_{\ell-1}}(Z_{\ell-1})
\bB_{i_\ell}(Z_\ell)
\end{align} 
for 
$\emptyset\ne\omega\in\cG_d$, 
and 
$\fR_{\emptyset}=D$,
while they can also be seen as multilinear mappings on 
$(\KK^{sm\times sm})^{\odot_s|\omega|}$ by
\begin{align*}
(Z_1\odot_s\cdots\odot_s Z_\ell)\fR_\omega=
(I_m\otimes C)(Z_1)\bA_{i_1}\cdots
(Z_{\ell-1})\bA_{i_{\ell-1}}
(Z_\ell)\bB_{i_\ell},
\end{align*}
if 
$\omega=g_{i_1}\cdots
g_{i_\ell}\ne\emptyset$ 
and 
$\fR_\emptyset=I_m\otimes D$,
for every $m\in\NN$ 
and 
$Z_1,\ldots,Z_{\ell}\in(\KK^{sm\times sm})^d$.
Therefore, it follows immediately that the series in   
(\ref{eq:5Ma17a}), 
which comes from a nc rational function 
$\fR$ (or even a nc rational expression), 
admits the minimal nc Fornasini--Marchesini realization $\cR$.

As mentioned in the discussion above, 
the coefficients 
$(\fR_\omega)_{\omega\in\cG_d}$
in
(\ref{eq:14Apr20c}) 
must satisfy the 
$\cL\cA$ 
conditions 
and hence we get the 
$\cL\cA$ 
conditions 
in terms of 
$\bA_1,\ldots,\bA_d,\bB_1,\ldots,\bB_d,C,$ 
and
$D$, 
in the case where 
$\fR$ 
admits a minimal realization centred at 
$\ul{Y}$, as formulated next:
\begin{lemma}
\label{lem:13Dec17a}
Let 
$f$ 
be a generalized nc power series around 
$\ul{Y}=(Y_1,\ldots,Y_d)\in(\KK^{\sts})^d$ 
of the form
(\ref{eq:12May20c}).
Suppose that 
$f$ 
admits a nc 
Fornasini--Marchesini realization
$\cR$ that is centred at
$\ul{Y}$ and described by 
$(L;D,C,\ul{\bA},\ul{\bB})$, and that 
$\cR$ is controllable and observable.
Then the series
$f$
is a nc function on 
$Nilp(\ul{Y})$
if and only if 
the following equations hold
\begin{align}
\label{eq:12Jun17a}
SD-DS=C
\sum_{k=1}^d \bB_k([S,Y_k]),
\end{align}
\begin{align}
\label{eq:12Jun17b}
SC\bB_{i_1}(Z_1)-C\bB_{i_1}(SZ_1)
=C\Big(\sum_{k=1}^d
\bA_k([S,Y_k])\Big)\bB_{i_1}(Z_1),
\end{align}
\begin{align}
\label{eq:12Jun17c}
SC\bA_{i_1}(Z_1)-C\bA_{i_1}(SZ_1)
=C\Big(\sum_{k=1}^d
\bA_k([S,Y_k])\Big)\bA_{i_1}(Z_1),
\end{align}
\begin{align}
\label{eq:12Jun17f}
\bB_{i_1}(Z_1S)-\bB_{i_1}(Z_1)S
=\bA_{i_1}(Z_1)
\sum_{k=1}^d
\bB_k([S,Y_k]),
\end{align}
\begin{align}
\label{eq:12Jun17e}
\bA_{i_1}(Z_1S)\bB_{i_2}(Z_2)
-\bA_{i_1}(Z_1)
\bB_{i_2}(SZ_2)=\bA_{i_1}(Z_1)
\Big( 
\sum_{k=1}^d
\bA_k([S,Y_k])\Big)\bB
_{i_2}(Z_2),
\end{align}
and
\begin{align}
\label{eq:12Jun17d}
\bA_{i_1}(Z_1S)\bA_{i_2}(Z_2)
-\bA_{i_1}(Z_1)
\bA_{i_2}(SZ_2)=\bA_{i_1}(Z_1)
\Big( 
\sum_{k=1}^d
\bA_k([S,Y_k])\Big)\bA_{i_2}(Z_2),
\end{align}
for every 
$S,Z_1,Z_2\in\KK^{\sts}$ 
and 
$1\le i_1,i_2\le d$. 
\end{lemma}
We call equations 
$(\ref{eq:12Jun17a})-(\ref{eq:12Jun17d})$ 
the \textbf{linearized lost-abbey ($\cL-\cL\cA$) conditions}.
\begin{proof}
From the assumption that
$f$ 
admits the realization 
$\cR$, 
we know that 
\begin{align}
\label{eq:14Apr20d}
f_{\emptyset}=D
\text{ and } 
f_{\omega}(Z_1,\ldots,Z_{\ell})
=C\bA_{i_1}(Z_1)\cdots\bA_{i_{\ell-1}}
(Z_{\ell-1})\bB_{i_\ell}(Z_{\ell})
\end{align}
for every
$\emptyset\ne\omega=g_{i_1}\ldots
g_{i_\ell}\in\cG_d$
and 
$Z_1,\ldots,Z_{\ell}
\in\KK^{s\times s}$.
Therefore, the series $f$ is a nc function in
$Nilp(\ul{Y})$ if and only if the coefficients $(f_{\omega})_{\omega\in\cG_d}$
in
(\ref{eq:14Apr20d}) satisfy  
equations  
$(\ref{eq:7May18a})-(\ref{eq:7May18c})$.
It is only left to show 
that this is equivalent to the fact 
that equations
$(\ref{eq:12Jun17a})-(\ref{eq:12Jun17d})$ hold.

This part of the proof is mainly technical, 
so we show it only for one of the equations, 
while stressing where the
minimality of 
$\cR$ 
is coming into play.
While skipping the computations, we mention briefly how to obtain all the
other equations:
\begin{enumerate}
\item[$\bullet$] 
(\ref{eq:12Jun17a}) 
is obtained from 
(\ref{eq:7May18a});
\item[$\bullet$]
(\ref{eq:12Jun17b}) is obtained from 
(\ref{eq:7May18b}), 
by taking 
$\ell=|\omega|=1$;
\item[$\bullet$]
(\ref{eq:12Jun17c}) 
is obtained from
(\ref{eq:7May18b}), 
by taking 
$\ell=|\omega|>1$ 
and using the controllability of
$(\ul{\bA},\ul{\bB})$;
\item[$\bullet$]
(\ref{eq:12Jun17f}) is obtained from
(\ref{eq:14Apr20e}), by taking any 
$\ell\ge1$ 
and using the observability of 
$(C,\ul{\bA})$; 
\item[$\bullet$]
(\ref{eq:12Jun17e}) 
is obtained from 
(\ref{eq:7May18c}), 
by taking 
$j=\ell-1$ 
and using the observability of
$(C,\ul{\bA})$;
\item[$\bullet$]
(\ref{eq:12Jun17d}) 
is obtained from 
(\ref{eq:7May18c}), 
by taking 
$j<\ell-1$ 
and using both the observability of
$(C,\ul{\bA})$ and the controllability of 
$(\ul{\bA},\ul{\bB})$.
\end{enumerate}
\iffalse
Let us consider equation
(\ref{eq:7May18b}):
\begin{multline*}
Sf_\omega
(Z_1,\ldots,Z_\ell)
-f_\omega(SZ_1,\ldots,Z_\ell)
=
\sum_{k=1}^d
f_{g_k\omega}([S,Y_k],Z_1,\ldots,
Z_\ell)
\iff
\\SC\bA_{i_1}(Z_1)\cdots\bA_{i_{\ell-1}}
(Z_{\ell-1})\bB_{i_\ell}(Z_{\ell})
-C\bA_{i_1}(SZ_1)\cdots\bA_{i_{\ell-1}}
(Z_{\ell-1})\bB_{i_\ell}(Z_{\ell})
\\=\sum_{k=1}^d
C\bA_k([S,Y_k])\bA_{i_1}(Z_1)\cdots\bA_{i_{\ell-1}}
(Z_{\ell-1})\bB_{i_\ell}(Z_{\ell})
\\\iff 
\Big(SC\bA_{i_1}(Z_1)-C\bA_{i_1}(SZ_1)-\sum_{k=1}^d
C\bA_k([S,Y_k]\bA_{i_1}(Z_1)\Big)
\bA_{i_2}(Z_2)\cdots\bA_{i_{\ell-1}}
(Z_{\ell-1})\bB_{i_\ell}(Z_{\ell})=0,
\end{multline*}
as the tuple
$(\ul{\bA},\ul{\bB})$ is controllable, it follows that the last equality
holds if and only if
\begin{align*}
SC\bA_{i_1}(Z_1)-C\bA_{i_1}(SZ_1)=
C\Big(\sum_{k=1}^d
\bA_k([S,Y_k]\Big)
\bA_{i_1}(Z_1),
\end{align*}
i.e.,
(\ref{eq:7May18b}) holds if and only if 
(\ref{eq:12Jun17c}) holds.
\fi
We now show the last equivalence in this list: as
\begin{multline*}
f_\omega(Z_1,\ldots,Z_{j-1},Z_jS,Z_{j+1},
\ldots,Z_\ell)
%=C\prod_{p=1}^{j-1}\bA_{i_p}(Z_p)\bA_{i_j}(Z_jS)
%\prod_{q=j+1}^{\ell-1}\bA_{i_{q}}(Z_{q})
%\bB_{i_{\ell}}(Z_{\ell}),
\\=C\bA_{i_1}(Z_1)\cdots
\bA_{i_{j-1}}(Z_{j-1})
\bA_{i_j}(Z_jS)
\bA_{i_{j+1}}(Z_{j+1})
\cdots\bA_{i_{\ell-1}}
(Z_{\ell-1})\bB_{i_{\ell}}(Z_{\ell}),
\end{multline*}
\begin{multline*}
f_\omega(Z_1,\ldots,Z_j,
SZ_{j+1},\ldots,Z_\ell)
%=C
%\prod_{p=1}^{j}\bA_{i_p}(Z_p)\bA_{i_{j+1}}(SZ_{j+1})
%\prod_{q=j+2}^{\ell-1}\bA_{i_{q}}(Z_q)\bB_{i_\ell}(Z_\ell)
\\=C\bA_{i_1}(Z_1)\cdots
\bA_{i_j}(Z_{j})
\bA_{i_{j+1}}(SZ_{j+1})
\bA_{i_{j+2}}(Z_{j+2})
\cdots\bA_{i_{\ell-1}}
(Z_{\ell-1})\bB_{i_{\ell}}(Z_{\ell}),
\end{multline*}
and
\begin{multline*}
f_{g_{i_1}\ldots g_{i_j}
g_kg_{i_{j+1}}\ldots g_{i_\ell}}
(Z_1,\ldots,Z_j,[S,Y_k],Z_{j+1},\ldots,Z_\ell)
\\=C\bA_{i_1}(Z_1)\cdots
\bA_{i_j}(Z_{j})\bA_k([S,Y_k])
\bA_{i_{j+1}}(Z_{j+1})
\cdots\bA_{i_{\ell-1}}
(Z_{\ell-1})\bB_{i_{\ell}}(Z_{\ell})
\end{multline*}
for every $1<j<\ell-1$,
equation
(\ref{eq:7May18c}) holds if and only if
\begin{multline*}
C\bA_{i_1}(Z_1)\cdots
\bA_{i_{j-1}}(Z_{j-1})
\bA_{i_j}(Z_jS)
\bA_{i_{j+1}}(Z_{j+1})
\cdots\bA_{i_{\ell-1}}
(Z_{\ell-1})\bB_{i_{\ell}}(Z_{\ell})
\\-C\bA_{i_1}(Z_1)\cdots
\bA_{i_j}(Z_{j})
\bA_{i_{j+1}}(SZ_{j+1})
\bA_{i_{j+2}}(Z_{j+2})
\cdots\bA_{i_{\ell-1}}
(Z_{\ell-1})\bB_{i_{\ell}}(Z_{\ell})
\\=\sum_{k=1}^d
C\bA_{i_1}(Z_1)\cdots
\bA_{i_j}(Z_{j})\bA_k([S,Y_k])
\bA_{i_{j+1}}(Z_{j+1})
\cdots\bA_{i_{\ell-1}}
(Z_{\ell-1})\bB_{i_{\ell}}(Z_{\ell}),
\end{multline*}
which is equivalent to
\begin{multline*}
C\bA_{i_1}(Z_1)\cdots
\bA_{i_{j-1}}(Z_{j-1})
\\ \Big(
\bA_{i_j}(Z_jS)
\bA_{i_{j+1}}(Z_{j+1})
-\bA_{i_j}(Z_{j})
\bA_{i_{j+1}}(SZ_{j+1})
\\-
\sum_{k=1}^d \bA_{i_j}(Z_{j})\bA_k([S,Y_k])
\bA_{i_{j+1}}(Z_{j+1})
\Big) \\
\bA_{i_{j+2}}(Z_{j+2})
\cdots\bA_{i_{\ell-1}}
(Z_{\ell-1})\bB_{i_{\ell}}(Z_{\ell})=0.
\end{multline*}
Due to the controllability of 
$(\ul{\bA},\ul{\bB})$ and the observability of
$(C,\ul{\bA})$, the last equation is equivalent to the the vanishing of the
middle piece, i.e., to
\begin{multline*}
\bA_{i_j}(Z_jS)
\bA_{i_{j+1}}(Z_{j+1})
-\bA_{i_j}(Z_{j})
\bA_{i_{j+1}}(SZ_{j+1})
\\ =\bA_{i_j}(Z_{j})
\Big(\sum_{k=1}^d\bA_k([S,Y_k])\Big)
\bA_{i_{j+1}}(Z_{j+1}),
\end{multline*}
which is exactly equation
(\ref{eq:12Jun17d}).
\end{proof}

\begin{remark}
\label{rem:29Jul20a}
If 
$\cH$
is an infinite dimensional Hilbert space,
$\bA_k:\KK^{\sts}\rightarrow\cL(\cH)$
and
$\bB_k:\KK^{\sts}\rightarrow\cL(\KK^{s},\cH)$
for every 
$1\le k\le d,\, 
C\in\cL(\cH,\KK^{s})$ 
and 
$D\in\KK^{\sts}$, 
such that 
conditions 
$(\ref{eq:12Jun17a})-(\ref{eq:12Jun17d})$
hold, then the coefficients 
$(f_{\omega})_{\omega\in\cG_d}$ ---
which are given in 
(\ref{eq:14Apr20d})
--- satisfy the 
$\cL\cA$ 
conditions, hence the power series in 
(\ref{eq:12May20c}) defines a nc function on 
$Nilp(\ul{Y})$ and it admits the realization described by 
$(D,C,\ul{\bA},\ul{\bB})$.
This follows from using the same arguments as 
in the proof of Lemma 
\ref{lem:13Dec17a}, where (formally) $\cH$ replaces the space 
$\KK^L$.
\end{remark}
\begin{remark}
\label{rem:10Sep19a}
Notice that even if we remove 
the assumptions on the realization being controllable and observable, we
get that satisfying the
$\cL-\cL\cA$ conditions implies that the series is a nc function, however
these assumptions on the realization are  required for the other direction.
\end{remark}

\begin{remark} 
\label{rem:12Jan20a}
Recall that for every
$n,m\in\NN$
and
$1\le  k\le d$, 
the linear mappings
$\bA_k:\KK^{\sts}\rightarrow\KK^{L\times L}$
and
$\bB_k:\KK^{\sts}\rightarrow\KK^{L\times s}$ 
are naturally extended to mappings 
$\bA_k:\KK^{sn\times sm}\rightarrow
\KK^{Ln\times Lm}$ 
and
$\bB_k:\KK^{sn\times sm}\rightarrow
\KK^{Ln\times sm}$, 
just by acting on the block matrices of size $\sts$. 
It takes some  straightforward computations (which are omitted here) 
to show that if the 
$\cL-\cL\cA$ conditions
hold, then they can be extended to 
$sn\times sm$ 
matrices, i.e., 
we have 
\begin{align}
\label{eq:12Jun17a2}
S(I_m\otimes D)-(I_n\otimes D)S=(I_n\otimes C)
\sum_{k=1}^d 
\big(S(I_m\otimes Y_k)-(I_n\otimes Y_k)S\big)\bB_k,
\end{align}
\begin{multline}
\label{eq:12Jun17b2}
S(I_m\otimes C)
(Z_1)\bB_{i_1}-(I_n\otimes C)
(SZ_1)\bB_{i_1}
\\ =(I_n\otimes C)\Big(\sum_{k=1}^d
(S(I_m\otimes Y_k)-(I_n\otimes Y_k)S)\bA_k\Big)
(Z_1)\bB_{i_1},
\end{multline}
\begin{multline}
\label{eq:12Jun17c2}
S(I_m\otimes C)(Z_1)
\bA_{i_1}-(I_n\otimes C)(SZ_1)\bA_{i_1}
\\ =(I_n\otimes C)\Big(\sum_{k=1}^d
\big(S(I_m\otimes Y_k)-(I_n\otimes Y_k)
S\big)\bA_k\Big)(Z_1)\bA_{i_1},
\end{multline}
\begin{align}
\label{eq:12Jun17f2}
(Z_2S)\bB_{i_2}-(Z_2)\bB_{i_2}S
=(Z_2)\bA_{i_2}\sum_{k=1}^d
\big(S(I_m\otimes Y_k)-
(I_n\otimes Y_k)S\big)\bB_k,
\end{align}
\begin{multline}
\label{eq:12Jun17e2}
(Z_2S)\bA_{i_2}(Z_1)\bB_{i_1}
-(Z_2)\bA_{i_2}
(SZ_1)\bB_{i_1} \\ =(Z_2)\bA_{i_2}
\Big( 
\sum_{k=1}^d
\big(S(I_m\otimes Y_k)-(I_n\otimes Y_k)S
\big)\bA_k\Big)(Z_1)\bB_{i_1},
\end{multline}
and
\begin{multline}
\label{eq:12Jun17d2}
(Z_2S)\bA_{i_2}(Z_1)\bA_{i_1}
-(Z_2)\bA_{i_2}
(SZ_1)\bA_{i_1}
\\ =(Z_2)\bA_{i_2}
\Big( 
\sum_{k=1}^d
\big(S(I_m\otimes Y_k)-(I_n\otimes Y_k)S\big)\bA_k
\Big)(Z_1)\bA_{i_1},
\end{multline}
for every 
$S\in\KK^{sn\times sm},Z_1\in\KK^{sm\times sm},
Z_2\in\KK^{sn\times sn}$,
and 
$1\le i_1,i_2\le d$.
\end{remark}

As a corollary of Lemma
\ref{lem:13Dec17a}
and
Remark
\ref{rem:10Sep19a}, if equations 
$(\ref{eq:12Jun17a})-(\ref{eq:12Jun17d})$ 
hold,
 then the coefficients in the power series expansion of the realization around
$\ul{Y}$ must satisfy the $\cL\cA$ 
conditions, hence the nc 
Fornasini--Marchesini realization
$\cR$ defines a nc function on 
$Nilp(\ul{Y})$.

Moreover, we now show that the realization is a nc function on a larger set,
that is
$\Omega(\cR)$, 
the invertibility set of the realization. 
This will be the first step of our approach in which we start from 
an arbitrary nc Fornasini--Marchesini realization, that is controllable and
observable, for which its coefficients satisfy the
$\cL-\cL\cA$ 
conditions (cf. equations 
$(\ref{eq:12Jun17a})-(\ref{eq:12Jun17d})$) 
and eventually find a nc rational function which this realization admits.
\begin{theorem}
\label{thm:19Oct17a}
Let $\cR$ be a nc Fornasini--Marchesini realization centred at 
$\ul{Y}$ that is described by
$(L;D,C,\ul{\bA},\ul{\bB})$. 
If the coefficients of $\cR$ satisfy the 
$\cL-\cL\cA$ 
conditions, then
$\Omega(\cR)$
is an upper admissible nc subset of
$(\KK^d)_{nc}$ 
and the function
$\cR:\Omega(\cR)\rightarrow 
\KK_{nc}$ that is given by
\begin{align}
\label{eq:26Jun19c}
\cR(\ul{X})=I_m\otimes D+(I_m\otimes C)
\Lambda_{\ul{\bA},\ul{Y}}(\ul{X})^{-1}
\sum_{k=1}^d 
(X_k-I_m\otimes Y_k)\bB_k,
\end{align}
$\forall\ul{X}
\in\Omega_{sm}(\cR),\,
m\in\NN$,
is a nc function 
(that is defined only in levels which are multiples of 
$s$, i.e., if $s\nmid n$ then the domain of $\cR$ at the level of $d-$tuples
of matrices in $\KK^{\ntn}$ is empty), with
$\ul{Y}\in \Omega_s(\cR)$.
\end{theorem}

\begin{proof}
\uline{\textbf{$\Omega(\cR)$ 
is an upper admissible nc subset of
$(\KK^d)_{nc}$:}}
If 
$m,\wt{m}\in\NN,\,
\ul{X}\in\Omega_{sm}(\cR),\,
\ul{\wt{X}}\in\Omega_{s\wt{m}}(\cR)$,
and
$\ul{Z}\in(\KK^{sm\times s\wt{m}})^d$,
then
\begin{align*}
\Lambda_{\ul{\bA},\ul{Y}}
\Big(\begin{bmatrix}
\ul{X}&\ul{Z}\\
\ul{0}&\ul{\wt{X}}\\
\end{bmatrix}\Big)
=\begin{bmatrix}
\Lambda_{\ul{\bA},\ul{Y}}(\ul{X})&-\sum_{k=1}^d (Z_k)\bA_k\\
0&\Lambda_{\ul{\bA},\ul{Y}}(\ul{\wt{X}})\\
\end{bmatrix}
\end{align*}
is invertible, i.e., 
$\begin{bmatrix}
\ul{X}&\ul{Z}\\
\ul{0}&\ul{\wt{X}}\\
\end{bmatrix}\in\Omega_{s(m+\wt{m})}(\cR)$.

\iffalse
$\bullet$  
\textbf{$\cR$ 
respects direct
sums.} 
If 
$m,\wt{m}\in\NN,\,
\ul{X}\in(\Omega_{\cR})_{m}$
and
$\ul{\wt{X}}\in(\Omega_{\cR})_{\wt{m}}$,
then
\begin{multline*}
\cR
\begin{pmatrix}
\ul{X}&0\\
0&\ul{\wt{X}}\\
\end{pmatrix}
=I_{n}\otimes D+
(I_n\otimes C)
\Lambda\Big(\begin{pmatrix}
\ul{X}&0\\
0&\ul{\wt{X}}\\
\end{pmatrix}\Big)^{-1}
\sum_{k=1}^d 
\Big(\begin{pmatrix}
X_k-I_m\otimes Y_k&0\\
0&\wt{X_k}-I_{\wt{m}}\otimes Y_k\\
\end{pmatrix}\Big)\bB_k
\\=\begin{pmatrix}
I_m\otimes D&0\\
0&I_{\wt{m}}\otimes D\\
\end{pmatrix}+
\begin{pmatrix}
I_m\otimes C&0\\
0&I_{\wt{m}}\otimes C\\
\end{pmatrix}
\begin{pmatrix}
\Lambda(\ul{X})&0\\
0&\Lambda(\ul{\wt{X}})\\
\end{pmatrix}^{-1}
\\\begin{pmatrix}
\sum_{k=1}^d
(X_k-I_m\otimes Y_k)\bB_k&0\\
0&\sum_{k=1}^d
(\wt{X_k}-I_{\wt{m}}\otimes Y_k)\bB_k\\
\end{pmatrix}=
\begin{pmatrix}
\cR(\ul{X})&0\\
0&\cR(\ul{\wt{X}})\\
\end{pmatrix}
\end{multline*}
where $n=m+\wt{m}$, i.e., $\cR(\ul{X}\oplus
\ul{\wt{X}})=\cR(\ul{X})\oplus\cR(\ul{\wt{X}})$.
\fi
\uline{\textbf{$\cR$ 
is a nc function on 
$\Omega(\cR)$:}} 
$\cR$ is clearly graded, 
so it is left to show the crucial part which is that \textbf{$\cR$
respects intertwining.} 
Let 
$m,\wt{m}\in\NN,\,  
\ul{X}\in\Omega_{sm}(\cR),\,
\wt{\ul{X}}\in\Omega_{s\wt{m}}(\cR)$,
and
$T\in\KK^{sm\times s\wt{m}}$ 
such that
$\ul{X}\cdot T=T\cdot\wt{\ul{X}}
$, 
i.e., 
that 
$X_kT=T\wt{X}_k$ 
for all 
$1\le k\le d$.
For simplifications let us define the matrices
\begin{align*}
& M_B:=\sum_{k=1}^d
\big(T(I_{\wt{m}}\otimes Y_k)-(I_m\otimes Y_k)T\big)\bB_k,\,
\\ & M_A:=\sum_{k=1}^d
\big(T(I_{\wt{m}}\otimes Y_k)-(I_m\otimes Y_k)T\big)\bA_k
\end{align*}
and recall that as equations 
$(\ref{eq:12Jun17a})-(\ref{eq:12Jun17d})$ 
hold, we know that equations
$(\ref{eq:12Jun17a2})-(\ref{eq:12Jun17d2})$ 
hold as well
(cf. Remark 
\ref{rem:12Jan20a}).  
From
$(\ref{eq:12Jun17f2})$ 
we get 
\begin{multline*}
\sum_{k=1}^d(X_k-I_m\otimes Y_k)\bB_kT
=\sum_{k=1}^d
\big[\big((X_k-I_m\otimes Y_k)T\big)
\bB_k-(X_k-I_m\otimes Y_k)\bA_k
M_B\big]
\\=\sum_{k=1}^d
\big[\big(
T\wt{X}_k-T(I_{\wt{m}}\otimes Y_k)+T(I_{\wt{m}}\otimes Y_k)
-(I_m\otimes Y_k)T
\big)\bB_k
\big]
-\sum_{k=1}^d 
(X_k-I_m\otimes Y_k)\bA_k
M_B
\\=\sum_{k=1}^d 
\big(
T(\wt{X}_k-I_{\wt{m}}\otimes Y_k)
\big)\bB_k+
\Lambda_{\ul{\bA},\ul{Y}}(\ul{X})M_B
\end{multline*}
and hence, using
$(\ref{eq:12Jun17a2})$
\begin{multline*}
\cR(\ul{X})T=(I_m\otimes D)T+(I_m\otimes C)M_B
+(I_m\otimes C)
\Lambda_{\ul{\bA},\ul{Y}}(\ul{X})^{-1}
\sum_{k=1}^d 
\big(
T(\wt{X}_k
-I_{\wt{m}}\otimes Y_k)\big)
\bB_k
\\=T(I_{\wt{m}}\otimes D)+
(I_m\otimes C)
\Lambda_{\ul{\bA},\ul{Y}}(\ul{X})
^{-1}
\sum_{k=1}^d 
\big(
T(\wt{X}_k-I_{\wt{m}}\otimes Y_k)
\big)\bB_k.
\end{multline*}
Next, from 
$(\ref{eq:12Jun17e2})$ it follows that
\begin{multline*}
\Big(\sum_{k=1}^d
\big(
(X_k-I_m\otimes Y_k)T
\big)\bA_k
\Big)
\sum_{k=1}^d
(\wt{X}_k-I_{\wt{m}}\otimes Y_k)\bB_k
\\-
\Big(\sum_{k=1}^d
(X_k-I_m\otimes Y_k)\bA_k
\Big)M_A
\sum_{k=1}^d
(\wt{X}_k-I_{\wt{m}}\otimes Y_k)\bB_k
\\=\Big(\sum_{k=1}^d 
(X_k-I_m\otimes Y_k)\bA_k\Big)
\sum_{k=1}^d 
\big(T(\wt{X}_k-I_{\wt{m}}\otimes Y_k)
\big)\bB_k
\\=-\Lambda_{\ul{\bA},\ul{Y}}(\ul{X})\sum_{k=1}^d 
\big(T(\wt{X}_k-I_{\wt{m}}\otimes Y_k)\big)\bB_k
+\sum_{k=1}^d 
\big(T(\wt{X}_k-I_{\wt{m}}\otimes Y_k)\big)\bB_k,
\end{multline*}
so
\begin{multline*}
\sum_{k=1}^d 
\big(T(\wt{X}_k-I_{\wt{m}}\otimes Y_k)\big)\bB_k
\\=
\Big(\sum_{k=1}^d
\big(
(X_k-I_m\otimes Y_k)T
\big)\bA_k\Big)
\sum_{k=1}^d
(\wt{X}_k-I_{\wt{m}}\otimes Y_k)\bB_k
\\-M_A
\sum_{k=1}^d(\wt{X}_k
-I_{\wt{m}}\otimes Y_k)\bB_k+
\Lambda_{\ul{\bA},\ul{Y}}(\ul{X})
M_A
\sum_{k=1}^d
(\wt{X}_k-I_{\wt{m}}\otimes Y_k)\bB_k
\\ +\Lambda_{\ul{\bA},\ul{Y}}(\ul{X})
\sum_{k=1}^d 
\big(T(\wt{X}_k-I_{\wt{m}}\otimes Y_k)\big)\bB_k
\\=
\Lambda_{\ul{\bA},\ul{Y}}(\ul{X})
\Big( \sum_{k=1}^d 
\big(
T(\wt{X}_k-I_{\wt{m}}\otimes Y_k)
\big)\bB_k
+M_A
\sum_{k=1}^d
(\wt{X}_k-I_{\wt{m}}\otimes Y_k)\bB_k\Big)
\\+\Big(\sum_{k=1}^d
\big((X_k-I_m\otimes Y_k)T
\big)\bA_k
-M_A\Big)
\sum_{k=1}^d
(\wt{X}_k-I_{\wt{m}}\otimes Y_k)\bB_k
\end{multline*}
and hence, using
$(\ref{eq:12Jun17b2})$,
\begin{multline*}
(I_m\otimes C)
\Lambda_{\ul{\bA},\ul{Y}}(\ul{X})^{-1}
\sum_{k=1}^d 
\big(T(\wt{X}_k-I_{\wt{m}}\otimes Y_k)
\big)\bB_k
\\=(I_m\otimes C)\Big[ 
\sum_{k=1}^d 
\big(T(\wt{X}_k-I_{\wt{m}}\otimes Y_k)\big)\bB_k
+M_A
\sum_{k=1}^d
(\wt{X}_k-I_{\wt{m}}\otimes Y_k)
\bB_k\Big]
\\+(I_m\otimes C)\Lambda_{\ul{\bA},\ul{Y}}(\ul{X})^{-1}
\Big(\sum_{k=1}^d
\big((X_k-I_m\otimes
Y_k)T\big)\bA_k-M_A\Big)
\sum_{k=1}^d
(\wt{X}_k-I_{\wt{m}}\otimes Y_k)\bB_k
\\=T(I_{\wt{m}}\otimes C)
\sum_{k=1}^d 
(\wt{X}_k-I_{\wt{m}}\otimes Y_k)\bB_k
\\ +(I_m\otimes C)
\Lambda_{\ul{\bA},\ul{Y}}(\ul{X})^{-1}
\Big(\sum_{k=1}^d
\big(
T(\wt{X}_k-I_{\wt{m}}\otimes Y_k)
\big)\bA_k\Big)
\sum_{k=1}^d
(\wt{X}_k-I_{\wt{m}}\otimes Y_k)\bB_k.
\end{multline*}
From $(\ref{eq:12Jun17d2})$ we know that
\begin{multline*}
\sum_{k=1}^d
\big(
T(\wt{X}_k-I_{\wt{m}}\otimes Y_k)
\big)\bA_k
\\=
\Big(\sum_{k=1}^d
(X_k-I_m\otimes Y_k)\bA_k\Big)
\sum_{k=1}^d
\big(
T(\wt{X}_k-I_{\wt{m}}\otimes Y_k)
\big)\bA_k
+\Lambda_{\ul{\bA},\ul{Y}}(\ul{X})
\sum_{k=1}^d
\big(
T(\wt{X}_k-I_{\wt{m}}\otimes Y_k)
\big)\bA_k
\\=\Big(\sum_{k=1}^d
\big(
(X_k-I_m\otimes Y_k)T
\big)\bA_k\Big)
\sum_{k=1}^d
(\wt{X}_k-I_{\wt{m}}\otimes Y_k)\bA_k\\
-\Big(\sum_{k=1}^d(X_k-I_m\otimes Y_k)\bA_k\Big)
M_A
\sum_{k=1}^d
(\wt{X}_k-I_{\wt{m}}\otimes Y_k)\bA_k
+\Lambda_{\ul{\bA},\ul{Y}}(\ul{X})
\sum_{k=1}^d
\big(
T(\wt{X}_k-I_{\wt{m}}\otimes Y_k)
\big)\bA_k\\
=\Big(\sum_{k=1}^d
\big(
(X_k-I_m\otimes Y_k)T
\big)\bA_k\Big)
\sum_{k=1}^d
(\wt{X}_k-I_{\wt{m}}\otimes Y_k)\bA_k
+\Lambda_{\ul{\bA},\ul{Y}}(\ul{X})M_A
\sum_{k=1}^d
(\wt{X}_k-I_{\wt{m}}\otimes Y_k)\bA_k\\
-M_A
\sum_{k=1}^d
(\wt{X}_k-I_{\wt{m}}\otimes Y_k)\bA_k
+\Lambda_{\ul{\bA},\ul{Y}}(\ul{X})
\sum_{k=1}^d
\big(
T(\wt{X}_k-I_{\wt{m}}\otimes Y_k)
\big)\bA_k
\\=\Big(\sum_{k=1}^d
\big(
(X_k-I_m\otimes Y_k)T
\big)\bA_k
-M_A\Big)
\sum_{k=1}^d
(\wt{X}_k-I_{\wt{m}}\otimes Y_k)\bA_k
\\+\Lambda_{\ul{\bA},\ul{Y}}(\ul{X})
\Big[M_A
\sum_{k=1}^d
(\wt{X}_k-I_{\wt{m}}\otimes Y_k)\bA_k
+\sum_{k=1}^d
\big(
T(\wt{X}_k-I_{\wt{m}}\otimes Y_k)
\big)\bA_k\Big]
\end{multline*}
therefore, by
$(\ref{eq:12Jun17c2})$,
\begin{multline*}
(I_m\otimes C)
\Lambda_{\ul{\bA},\ul{Y}}(\ul{X})^{-1}
\sum_{k=1}^d
\big(
T(\wt{X}_k-I_{\wt{m}}\otimes Y_k)
\big)\bA_k
\\=(I_m\otimes C)
\Lambda_{\ul{\bA},\ul{Y}}(\ul{X})^{-1}
\Big(\sum_{k=1}^d
\big(
T(\wt{X}_k-
I_{\wt{m}}\otimes Y_k)
\big)\bA_k
\Big)\sum_{k=1}^d
(\wt{X}_k-I_{\wt{m}}\otimes Y_k)\bA_k
\\+(I_m\otimes C)\Big[M_A
\sum_{k=1}^d
(\wt{X}_k-I_{\wt{m}}\otimes Y_k)\bA_k
+\sum_{k=1}^d
\big(
T(\wt{X}_k
-I_{\wt{m}}\otimes Y_k)\big)
\bA_k
\Big]
\\=(I_m\otimes C)
\Lambda_{\ul{\bA},\ul{Y}}(\ul{X})^{-1}
\Big(\sum_{k=1}^d
\big(
T(\wt{X}_k-I_{\wt{m}}\otimes Y_k)
\big)\bA_k
\Big)
\sum_{k=1}^d
(\wt{X}_k-I_{\wt{m}}\otimes Y_k)\bA_k
\\ +T(I_{\wt{m}}\otimes C)
\sum_{k=1}^d 
(\wt{X}_k-I_{\wt{m}}\otimes Y_k)\bA_k.
\end{multline*}
Then,
\begin{multline*}
(I_m\otimes C)
\Lambda_{\ul{\bA},\ul{Y}}(\ul{X})^{-1}
\Big(\sum_{k=1}^d
\big(
T(\wt{X}_k-I_{\wt{m}}\otimes Y_k)
\big)\bA_k
\Big)
\Lambda_{\ul{\bA},\ul{Y}}(\ul{\wt{X}})
\\ =T(I_{\wt{m}}\otimes C)
\sum_{k=1}^d 
(\wt{X}_k-I_{\wt{m}}\otimes Y_k)\bA_k,
\end{multline*}
which implies that
\begin{multline*}
(I_m\otimes C)\Lambda_{\ul{\bA},\ul{Y}}(\ul{X})^{-1}
\sum_{k=1}^d
\big(
T(\wt{X}_k-I_{\wt{m}}\otimes Y_k)
\big)\bA_k
\\ =T(I_{\wt{m}}\otimes C)
\Big(
\sum_{k=1}^d 
(\wt{X}_k-I_{\wt{m}}\otimes Y_k)\bA_k\Big)
\Lambda_{\ul{\bA},\ul{Y}}(\ul{\wt{X}})^{-1}.
\end{multline*}
Therefore,
\begin{multline*}
T(I_{\wt{m}}\otimes D)+T(I_{\wt{m}}\otimes C)
\sum_{k=1}^d 
(\wt{X}_k-I_{\wt{m}}\otimes Y_k)\bB_k
\\ +T(I_{\wt{m}}\otimes C)\Big(
\sum_{k=1}^d 
(\wt{X}_k-I_{\wt{m}}\otimes Y_k)\bA_k\Big)
\Lambda_{\ul{\bA},\ul{Y}}(\ul{\wt{X}})^{-1}
\sum_{k=1}^d
(\wt{X}_k-I_{\wt{m}}\otimes Y_k)\bB_k
\\ =T(I_{\wt{m}}\otimes D)+T(I_{\wt{m}}\otimes C)
\Lambda_{\ul{\bA},\ul{Y}}(\ul{\wt{X}})^{-1}
\sum_{k=1}^d
(\wt{X}_k-I_{\wt{m}}\otimes Y_k)\bB_k,
\end{multline*}
i.e.,
$\cR(\ul{X})T=T\cR(\wt{\ul{X}})$.
\end{proof}
\begin{remark}
The idea of that part of the proof, where we showed that $\cR$
respects intertwining, is very similar to the proof of
\cite[Lemma 5.12]{KV1}, with the only difference that instead of a realization
(centred at $\ul{Y}$) 
they had a power series (around
$\ul{Y}$). 
\end{remark}
\section{NC Difference-Differential Calculus}
\label{subsec:cal}
In this section we prove that if the coefficients of a nc Fornasini--Marchesini
realization 
$\cR$, that is controllable and observable,
satisfy
the $\cL-\cL\cA$ conditions, then the set
$\Omega(\cR)$ 
is similarity invariant (cf.
Theorem \ref{thm:6Aug19a}).  It turns out to be an important ingredient 
in obtaining one of our main results (cf. Theorem \ref{thm:3Jun19a}).
\\\\
In the proof of Theorem \ref{thm:6Aug19a} below we use some techniques from the
general theory of nc functions and their difference-differential calculus;
we hereby recall some facts. 

Let $\bA_1,\ldots,\bA_d,\bB_1,\ldots,\bB_d,C,$
and
$D$ satisfy the
$\cL-\cL\cA$ conditions. For every
$1\le j\le d$,
the $j$-th right partial nc difference-differential operator 
$\Delta_j$ 
is defined using the nc difference-differential operator 
$\Delta:=\Delta_R$, by
$$\Delta_{j}f(\ul{X}^1,\ul{X}^2)(Z)
=\Delta f(\ul{X}^1,\ul{X}^2)
(\ul{Z}^{(j)}),$$ 
for every nc function 
$f,\,n_1,n_2\in\NN,\,
\ul{X}^1\in (\KK^{n_1\times n_1})^d$
and
$\ul{X}^2\in(\KK^{n_2\times n_2})^d$ 
in the domain of $f$,
and
$Z\in\KK^{n_1\times n_2}$,
where 
$\ul{Z}^{(j)}:=
(0,\ldots,0,Z,0,\ldots,0)\in (\KK^{n_1\times n_2})^d$ 
is the 
$d-$tuple which consists of all zero matrices except for the matrix $Z$ at
the $j$-th position (see 
\cite[Subsections 2.2-2.6]{KV1} for the precise definitions and properties
of 
$\Delta$ and $\Delta_j$).
To continue, 
we need the following technical lemma.
\begin{lemma}
\label{lem:16Jul19d}
Let $\cR$
be a nc Fornasini--Marchesini realization centred at 
$\ul{Y}\in(\KK^{\sts})^d$ 
that is described by 
$(L;D,C,\ul{\bA},\ul{\bB})$
and
suppose its coefficients satisfy the 
$\cL-\cL\cA$ conditions.
Let 
$\cU=(\KK^{\sts})^d$
and
define 
\begin{align*}
F_1:
\cU_{nc}
\rightarrow
(\KK^{\sts})_{nc},
\text{ by }
F_1(\ul{X})=I_m\otimes D,\,
\forall\ul{X}\in(\KK^{sm\times sm})^d\\ 
F_2:\cU_{nc}
\rightarrow 
(\KK^{s\times L})_{nc},
\text{ by }
F_2(\ul{X})=I_m\otimes C,
\,\forall\ul{X}\in(\KK^{sm\times sm})^d \\
F_3:\Omega(\cR)
\rightarrow
(\KK^{L\times L})_{nc}
,\text{ by }
F_3(\ul{X})=
\Lambda_{\ul{\bA},\ul{Y}}(\ul{X})^{-1},
\,\forall\ul{X}\in\Omega_{sm}(\cR)
\end{align*}
and
\begin{align*}
F_4:
\cU_{nc}
\rightarrow
(\KK^{L\times s})_{nc}
,\text{ by }
F_4(\ul{X})=
\sum_{k=1}^d 
(X_k-I_m\otimes Y_k)\bB_k,
\,\forall\ul{X}\in (\KK^{sm\times sm})^d.
\end{align*}
Then $F_1,F_2$, and 
$F_4$
are nc functions on 
$\cU_{nc}$,
and
$F_3$
is a nc function on
$\Omega(\cR)\subseteq\cU_{nc}$, 
where the nc space is considered 
as a module over $\KK$
and not over $\KK^{\sts}$. 
Moreover, for every
$1\le j\le d$, 
we have $\Delta_j F_1\equiv0,\,
\Delta_j F_2\equiv0$,
$$\Delta_jF_3(\ul{X}^1,
\ul{X}^2)(Z)=F_3(\ul{X}^1)(Z)
\bA_jF_3(\ul{X}^2)
\text{ and }
\Delta_jF_4(\ul{\wt{X}}^1,
\ul{\wt{X}}^2)(Z)=(Z)\bB_j$$
for every 
$m_1,m_2\in\NN,\,
\ul{X}^1\in \Omega_{sm_1}(\cR),\,
\ul{X}^2\in\Omega_{sm_2}(\cR),\,
\ul{\wt{X}}^1\in 
\cU^{m_1\times m_1},\,
\ul{\wt{X}}^2\in
\cU^{m_2\times m_2}$,
and
$Z\in\KK^{sm_1\times sm_2}$.
\end{lemma}
It is easily seen that
$\cR=F_1+F_2F_3F_4,$ 
so the calculations from Lemma
\ref{lem:16Jul19d}  will be used later (cf. Lemma  
\ref{lem:6Aug19b}) to calculate the (higher order) nc derivatives of $\cR$.
Notice that 
while $\cR$ 
is a nc function on a nc subset of
$(\KK^d)_{nc}$
to
$\KK_{nc}$ and is only defined on levels of the form $n=sm$, 
yet the functions $F_1,\ldots,F_4$ are nc functions on nc subsets of
$((\KK^{\sts})^d)_{nc}$ to
$(\KK^{\sts})_{nc}$.
\begin{proof}
In Theorem 
\ref{thm:19Oct17a}
we proved that 
$\Omega(\cR)$ is an upper admissible nc subset of 
$(\cU^d)_{nc}$, while it  is easily seen that 
$F_1, F_2, F_3,$
and
$F_4$ are graded and respect direct sums, 
it is only left to show that they
respect similarities.

Let 
$\ul{X}\in\cU_m$, 
i.e., 
$\ul{X}=(X_1,\ldots,X_d)\in 
(\KK^{sm\times sm})^d$
and let $S\in\KK^{\mtm}$ 
be invertible. We recall the actions of 
$S$ 
(on the left) and of 
$S^{-1}$ (on the right)  on a tuple 
$\ul{X}$, these 
are given by
$$S\cdot\ul{X}\cdot S^{-1}=
\big((S\otimes I_s)X_1
(S^{-1}\otimes I_s),\ldots,
(S\otimes I_s)X_1(S^{-1}\otimes I_s)\big).$$
Therefore
\begin{align*}
S\cdot F_1(\ul{X})\cdot S^{-1}
=(S\otimes I_s)(I_m\otimes D)
(S^{-1}\otimes I_s)
=I_m\otimes D=
F_1(S\cdot\ul{X}\cdot S^{-1}),
\end{align*}  
while from linearity of 
$\bB_k$ 
we get
\begin{multline*}
S\cdot F_4(\ul{X})\cdot S^{-1}=
(S\otimes I_s) 
\Big(\sum_{k=1}^d 
(X_k-I_m\otimes Y_k)\bB_k\Big)
(S^{-1}\otimes I_s)
\\ =\sum_{k=1}^d
((S\otimes I_s)(X_k
-I_m\otimes Y_k)(S^{-1}\otimes I_s))\bB_k
=\sum_{k=1}^d 
(S\cdot X_k\cdot S^{-1}-I_m\otimes Y_k)\bB_k
\\ =F_4(S\cdot\ul{X}\cdot S^{-1})
\end{multline*}
and similarly for 
$F_2$ 
and 
$F_3$, 
with the only difference that for 
$F_2$ we begin with $\ul{X}\in\Omega(\cR)$. 
Moreover, let $\ul{X}^1\in 
\Omega_{sm_1}(\cR),
\ul{X}^2\in\Omega_{sm_2}(\cR)$, 
and 
$Z\in\cU^{m_1\times m_2}$,
thus
\begin{multline*}
F_3\Big(\begin{bmatrix}
\ul{X}^1&\ul{Z}^{(j)}\\
\ul{0}&\ul{X}^2\\
\end{bmatrix}\Big) \\ =
\Big(I_{L(m_1+m_2)}-\sum_{k=1}^d
\Big(\begin{bmatrix}
X^1_k-I_{m_1}\otimes Y_k&\delta_{jk}Z\\
0&X_2^k-I_{m_2}\otimes Y_k\\
\end{bmatrix}\Big)\bA_k\Big)^{-1}
\\=\begin{bmatrix}
I_{Lm_1}-\sum_{k=1}^d
(X_k^1-I_{m_1}\otimes Y_k)
\bA_k&-(Z)\bA_j\\
0&I_{Lm_2}-\sum_{k=1}^d
(X_k^2-I_{m_2}\otimes Y_k)\bA_k
\end{bmatrix}^{-1}
\\=\begin{bmatrix}
F_3(\ul{X}^1)&F_3(\ul{X}^1)(Z)
\bA_jF_3(\ul{X}^2)\\
0&F_3(\ul{X}^2)\\
\end{bmatrix} \\ \Longrightarrow 
\Delta_j F_3(\ul{X}^1,\ul{X}^2)(Z)
=F_3(\ul{X}^1)(Z)\bA_jF_3(\ul{X}^2)
\end{multline*}
and similar calculations hold for $F_4$.
\end{proof}

\begin{definition}
\label{def:12Apr20a}
Let
$\cW$ be a vector space over
$\KK$ and
let 
$2\le k\in\NN$.
For every
$0\le j\le k$, let
$\Gamma_j$ 
be an upper admissible nc subset of
$(\cW^d)_{nc}$,
let
$f_j:\Gamma_j\rightarrow \cW_{nc}$ 
be a nc function, and for every 
$1\le j\le k$, 
let
$h_j:\cW\rightarrow\cW$ 
be a linear mapping.
If 
$m_0,\ldots,m_k\in\NN$ and
$\ul{X}^0\in (\Gamma_0)_{m_0},\ldots, \ul{X}^k\in (\Gamma_k)_{m_k}$, 
we define
$$(f_0\otimes\ldots\otimes f_k)_{h_1,\ldots,h_{k}}
(\ul{X}^0,\ldots,\ul{X}^k):\cW^{m_0\times m_1}\otimes
\cdots\otimes\cW^{m_{k-1}\times m_k}\rightarrow 
\cW^{m_0\times m_k}$$
by
\begin{multline*}
(f_0\otimes\ldots\otimes f_k)_{h_1,\ldots,h_{k}}
(\ul{X}^0,\ldots,\ul{X}^k)(Z^1,\ldots,Z^{k})
\\ :=f_0(\ul{X}^0)h_1(Z^1)\cdots f_{k-1}
(\ul{X}^{k-1})h_{k}
(Z^{k})f_k(\ul{X}^k)
\end{multline*}
for every 
$Z^1\in \cW^{m_0\times m_1},
\ldots,Z^{k}\in\cW^{m_{k-1}\times m_k}$.
\end{definition}
The following useful  
lemma is a generalization of an idea from
\cite[pp. 41--54]{KV1},
while over there only the special case --- in which 
$h_1,\ldots,h_k$ 
are all equal to the identity mapping
--- is treated. 
For the definition of nc functions of order $k$, that is the class 
$\cT^k$, their derivatives and more properties, see
\cite[Subsection 3.1]{KV1}.
\begin{lemma}
\label{lem:12Ap17a}
Let
$\cW$ 
be a vector space over
$\KK$, let $2\le k\in\NN$,
and let 
$f_0,\ldots,f_k,h_1,\ldots,h_k$
and
$\Gamma_0,\ldots,\Gamma_k$ 
be as in Definition 
\ref{def:12Apr20a}.
Then
\begin{align}
\label{eq:16Jul19c}
(f_0\otimes\ldots\otimes f_k)_{h_1,\ldots,h_{k}}
\in\cT^{k}(\Gamma_0,\ldots,\Gamma_k;\, 
\cW_{nc},\ldots,\cW_{nc})
\end{align} 
and
\begin{multline}
\label{eq:16Jul19b}
\Delta_j (f_0\otimes\ldots\otimes f_k)_{h_1,\ldots,
h_{k}}(\ul{X}^0,
\ldots,\ul{X}^{k-1},
c{\ul{X}^k}^{\prime\prime})
(Z^1,\ldots,Z^{k-1},
{Z^k}^{\prime\prime},Z)
\\=
f_0(\ul{X}^0)h_1(Z^1)\cdots
h_{k-1}(Z^{k-1})f_{k-1}(\ul{X}^{k-1}) h_{k}
({Z^k}^{\prime})\Delta_j
f_k({\ul{X}^k}^{\prime},{\ul{X}^k}^{\prime\prime})(Z)
\end{multline}
for every $1\le j\le d$,
$\ul{X}^0\in (\Gamma_0)_{m_0},\ldots,
\ul{X}^{k-1}\in(\Gamma_{k-1})_{m_{k-1}}$,
${\ul{X}^{k}}^{\prime}\in (\Gamma_k)_{m^\prime_k}$,
${\ul{X}^k}^{\prime\prime}\in(\Gamma_k)_{m^{\prime\prime}_k}$,
$Z^1\in \cW^{m_0\times m_1},
\ldots,Z^{k-1}\in\cW^{m_{k-2}\times m_{k-1}}$,
${Z^k}^{\prime}\in\cW^{m_{k-1}\times m^\prime_k}$,
and
$Z\in\cW^{m^\prime_k\times 
m^{\prime\prime}_k}$.
\end{lemma}
Notice that the notation introduced in the lemma can be naturally extended
to the case of tensor products of  higher order nc functions, so that the
equality proven in 
(\ref{eq:16Jul19b}) 
can be written as
$$\Delta_j(f_0\otimes\ldots\otimes f_k)_{h_1,\ldots,h_{k}}
=(f_0\otimes\ldots\otimes \Delta_j f_k)_{h_1,\ldots,h_{k}},$$
similarly to
\cite[equations (3.38)--(3.39)]{KV1} . 
We leave the notational details for the reader.
\begin{proof}
Let
$g:=(f_0\otimes\ldots\otimes f_k)_{h_1,\ldots,h_{k}}.$
It is not hard to check  that 
$g
\in\cT^{k}(\Gamma_0,\ldots,\Gamma_k;\, 
\cW_{nc},\ldots,\cW_{nc}),$ 
so we will only prove 
(\ref{eq:16Jul19b}).
Let 
${Z^k}^{\prime\prime}\in\cW^{m_{k-1}
\times m^{\prime\prime}_k}$.
As 
(\ref{eq:16Jul19c})
holds, one can apply \cite[equation (3.19)]{KV1} to obtain that
$$
g\Big(\ul{X}^0,\ldots,\ul{X}^{k-1},
\begin{bmatrix}
{\ul{X}^k}^{\prime}
&\ul{Z}^{(j)}\\ 
0&{\ul{X}^k}^{\prime\prime}\\
\end{bmatrix}
\Big)\Big(Z^1,\ldots,Z^{k-1},
\begin{bmatrix}
{Z^k}^{\prime}&{Z^k}^{\prime\prime}\\
\end{bmatrix}\Big)=
row\begin{bmatrix}
I_{11}&I_{12}\\
\end{bmatrix},$$
where $\ul{Z}^{(j)}=\ul{e}_j^T\otimes Z$,
$
I_{11}=g(\ul{X}^0,\ldots,
\ul{X}^{k-1},{\ul{X}^k}^{\prime})
(Z^1,\ldots,Z^{k-1},{Z^k}^{\prime})$,
and
\begin{multline*}I_{12}=g(\ul{X}^0,\ldots,\ul{X}^{k-1},
{\ul{X}^k}^{\prime\prime})
(Z^1,\ldots,Z^{k-1},{Z^k}^{\prime\prime})
\\+\Delta_j g(\ul{X}^0,\ldots,\ul{X}^{k-1},
{\ul{X}^k}^{\prime},{\ul{X}^k}^{\prime\prime})
(Z^1,\ldots,Z^{k-1},{Z^k}^{\prime},Z).
\end{multline*}
On the other hand, we know that
\begin{multline*}
g\Big(\ul{X}^0,\ldots,\ul{X}^{k-1},
\begin{bmatrix}
{\ul{X}^k}^{\prime}&\ul{Z}^{(j)}\\ 
0&{\ul{X}^k}^{\prime\prime}\\
\end{bmatrix}
\Big)\Big(Z^1,\ldots,Z^{k-1},
\begin{bmatrix}
{Z^k}^{\prime}&{Z^k}^{\prime\prime}\\
\end{bmatrix}\Big)
\\ =f_0(\ul{X}^0)h_1(Z^1)\cdots f_{k-1}(\ul{X}^{k-1})
h_k(\begin{bmatrix}
{Z^k}^{\prime}&{Z^k}^{\prime\prime}\\
\end{bmatrix})f_k\Big(\begin{bmatrix}
{\ul{X}^k}^{\prime}&\ul{Z}^{(j)}\\ 
0&{\ul{X}^k}^{\prime\prime}\\
\end{bmatrix}\Big)
\\ =f_0(\ul{X}^0)h_1(Z^1)\cdots f_{k-1}(\ul{X}^{k-1})
\begin{bmatrix}
h_k({Z^k}^{\prime})&h_k({Z^k}^{\prime\prime})\\
\end{bmatrix}
\begin{bmatrix}
f_k({\ul{X}^k}^{\prime})&\Delta_j 
f_k({\ul{X}^k}^{\prime},
{\ul{X}^k}^{\prime\prime})(Z)\\ 
0&f_k({\ul{X}^k}^{\prime\prime})\\
\end{bmatrix}
\\ =row
\begin{bmatrix}
J_{11}&J_{12}\\
\end{bmatrix},
\end{multline*}
where it is easily seen that $J_{11}=I_{11}$ and
\begin{multline*}
J_{12}
=f_0(\ul{X}^0)h_1(Z^1)\cdots f_{k-1}(\ul{X}^{k-1})
\\ \Big(h_k({Z^k}^{\prime})\Delta_j 
f_k({\ul{X}^k}^{\prime},{\ul{X}^k}^{\prime\prime})(Z)
+h_k({Z^k}^{\prime\prime})f_k({\ul{X}^k}^{\prime\prime})\Big)
\\=g(\ul{X}^0,\ldots,\ul{X}^{k-1},
{\ul{X}^k}^{\prime\prime})
(Z^1,\ldots,Z^{k-1},{Z^k}^{\prime\prime})
\\ +f_0(\ul{X}^0)h_1(Z^1)\ldots
h_{k-1}(Z^{k-1})f_{k-1}(\ul{X}^{k-1}) 
h_{k}
({Z^k}^{\prime})\Delta_j
f_k({\ul{X}^k}^{\prime},
{\ul{X}^k}^{\prime\prime})(Z).
\end{multline*}
As $I_{12}=J_{12}$, we get 
equation (\ref{eq:16Jul19b}).
\end{proof}
As
$\cR:\Omega(\cR)\rightarrow \KK_{nc}$
is a nc function and 
$\Omega(\cR)$ 
is an upper admissible nc set, it makes sense to consider the difference-differential
operators on $\cR$, these are 
$\Delta_j\cR$ 
for
$1\le j\le d$ 
and even higher order derivatives 
$\Delta^{\omega}\cR=\Delta_{j_k}\Delta_{j_{k-1}}
\ldots\Delta_{j_1}\cR$, where 
$\omega=g_{i_1}\ldots g_{i_\ell}\in\cG_d$.
Recall that
$$\Delta_j(fg)(\ul{X}^1,\ul{X}^2)(Z)
=f(\ul{X}^1)\Delta_jg(\ul{X}^1,\ul{X}^2)(Z)+
\Delta_jf(\ul{X}^1,\ul{X}^2)(Z)g(\ul{X}^2),$$
whenever 
$\ul{X}^1,\ul{X}^2$, and
$Z$
are of appropriate sizes and $f,g$ are nc functions; see
\cite[pp. 23]{KV1} for the proof.
Using Lemma
\ref{lem:16Jul19d}
and Lemma
\ref{lem:12Ap17a}, we get the following important lemma.
\begin{lemma}
\label{lem:6Aug19b}
Let $\cR$
be a nc Fornasini--Marchesini realization centred at 
$\ul{Y}\in\KK^{\sts}$ 
that is described by
$(L;D,C,\ul{\bA},\ul{\bB})$
and
suppose its coefficients satisfy the 
$\cL-\cL\cA$ conditions.
For every 
$m\in\NN,\,\ul{X}\in \Omega_{sm}(\cR),\, 
k\in\NN,\,
Z^1,\ldots,Z^k\in\KK^{sm\times sm}$, 
and 
$1\le j_1,\ldots,j_k\le d$,
we have 
\begin{multline}
\label{eq:6Aug19b}
\Delta_{j_k}\ldots\Delta_{j_1}
\cR(\ul{X},I_m\otimes \ul{Y},\ldots,I_m\otimes \ul{Y},I_m\otimes \ul{Y})
(Z^1,\ldots,Z^k)
\\=(I_{m}\otimes C)F_3(\ul{X})
(Z^1)\bA_{j_1}\cdots(Z^{k-1})\bA_{j_{k-1}}(Z^k)\bB_{j_k}
\end{multline}
and
\begin{multline}
\label{eq:6Aug19c}
\Delta_{j_k}\ldots
\Delta_{j_1}F_3(I_m\otimes\ul{Y},\ldots,
I_m\otimes\ul{Y},\ul{X})(Z^1,\ldots,Z^k)
\\ =(Z^1)\bA_{j_1}\cdots
(Z^k)\bA_{j_k}F_3(\ul{X}),
\end{multline}
where $F_3$ 
is given in Lemma
\ref{lem:16Jul19d}.
\end{lemma}
\begin{proof}
Let  
$F_1,F_2,F_3$, and $F_4$ be as defined in Lemma 
\ref{lem:16Jul19d}, so
$$\cR(\ul{X})=F_1(\ul{X})+
F_2(\ul{X})F_3(\ul{X})F_4(\ul{X}),\,
\forall\ul{X}\in\Omega(\cR).$$
For every 
$\ul{X}^1\in\Omega_{sm_1}(\cR),\, 
\ul{X}^2\in\Omega_{sm_2}(\cR),\, 
Z^1\in\KK^{sm_1\times sm_2}$, 
and 
$1\le j_1\le d$,
\begin{multline*}
\Delta_{j_1}\cR(\ul{X}^1,\ul{X}^2)(Z^1)=
\Delta_{j_1}F_1(\ul{X}^1,\ul{X}^2)(Z^1)
+\Delta_{j_1}(F_2F_3F_4)(\ul{X}^1,\ul{X}^2)(Z^1)
\\=F_2(\ul{X}^1)
\Delta_{j_1}(F_3F_4)
(\ul{X}^1,\ul{X}^2)(Z^1)
+
\Delta_{j_1}F_2(\ul{X}^1,\ul{X}^2)(Z^1) 
F_3(\ul{X}^2)F_4(\ul{X}^2)
\\=F_2(\ul{X}^1)\big[F_3(\ul{X}^1)
\Delta_{j_1}
F_4(\ul{X}^1,\ul{X}^2)(Z^1)
+\Delta_{j_1}
F_3(\ul{X}^1,\ul{X}^2)(Z^1)
F_4(\ul{X}^2)\big]
\\=F_2(\ul{X}^1)\big[
F_3(\ul{X}^1)(Z^1)\bB_{j_1}
+F_3(\ul{X}^1)
(Z^1)\bA_{j_1}F_3(\ul{X}^2)
F_4(\ul{X}^2)\big]
\\=F_2(\ul{X}^1)F_3(\ul{X}^1)
(Z^1)\bB_{j_1}+F_2(\ul{X}^1)
F_3(\ul{X}^1)(Z^1)\bA_{j_1}
F_3(\ul{X}^2)F_4(\ul{X}^2)
\end{multline*}
and in particular, as 
$I_m\otimes \ul{Y}\in 
\Omega_{sm}(\cR)$
and
$F_4(I_m\otimes \ul{Y})=0$, 
$$\Delta_{j_1}\cR(\ul{X}^1,I_m\otimes \ul{Y})(Z^1)=
F_2(\ul{X}^1)F_3(\ul{X}^1)(Z^1)\bB_{j_1}.$$
Using the notations in Lemma 
\ref{lem:12Ap17a}, 
we have
\begin{align*}
\Delta_{j_1}\cR(\ul{X}^1,\ul{X}^2)(Z)=
\big[(f_1\otimes f_2)_{\bB_{j_1}}
+(f_1\otimes f_3)_{\bA_{j_1}}\big]
(\ul{X}^1,\ul{X}^2)(Z),
\end{align*}
i.e., 
$\Delta_{j_1}\cR
=(f_1\otimes f_2)_{\bB_{j_1}}
+(f_1\otimes f_3)_{\bA_{j_1}}$,
where 
$f_1=F_2F_3,
f_2=I_s$ 
and 
$f_3=F_3F_4$ 
are nc functions. So
\begin{multline*}
\Delta_{j_2}\Delta_{j_1}
\cR(\ul{X}^1,\ul{X}^2,\ul{X}^3)(Z^1,Z^2) \\
=\Delta_{j_2}
((f_1\otimes f_2)_{\bB_{j_1}}+
(f_1\otimes f_3)_{\bA_{j_1}})
(\ul{X}^1,\ul{X}^2,\ul{X}^3)(Z^1,Z^2)
\\=f_1(\ul{X}^1)(Z^1)\bB_{j_1}
\Delta_{j_2}f_2
(\ul{X}^2,\ul{X}^3)(Z^2)+f_1(\ul{X}^1)(Z^1)
\bA_{j_1}\Delta_{j_2}f_3
(\ul{X}^2,\ul{X}^3)(Z^2)
\\=f_1(\ul{X}^1)
(Z^1)
\bA_{j_1} \big[F_3(\ul{X}^2)
\Delta_{j_2}F_4(\ul{X}^2,\ul{X}^3)(Z^2)
+\Delta_{j_2}F_3(\ul{X}^2,
\ul{X}^3)(Z^2)F_4(\ul{X}^3)\big]
\\=f_1(\ul{X}^1)(Z^1)\bA_{j_1}
\big[F_3(\ul{X}^2)
(Z^2)\bB_{j_2}+F_3(\ul{X}^2)(Z^2)
\bA_{j_2}
F_3(\ul{X}^3)F_4(\ul{X}^3)\big]
\\=\big[(f_1\otimes F_3
\otimes f_2)_{\bA_{j_1},\bB_{j_2}}
+(f_1\otimes F_3\otimes 
f_3)_{\bA_{j_1},\bA_{j_2}}\big]
(\ul{X}^1,\ul{X}^2,\ul{X}^3)(Z^1,Z^2),
\end{multline*}
i.e., 
\begin{align*}
\Delta_{j_2}\Delta_{j_1}\cR
=(f_1\otimes F_3\otimes 
f_2)_{\bA_{j_1},\bB_{j_2}}
+(f_1\otimes F_3\otimes 
f_3)_{\bA_{j_1},\bA_{j_2}}.
\end{align*}
We proceed the proof by induction (on 
$k$). 
Suppose that
\begin{multline*}
\Delta_{j_k}\ldots\Delta_{j_1}
\cR(\ul{X}^1,\ldots,\ul{X}^{k+1})
(Z^1,\ldots,Z^k)
\\=(f_1\otimes F_3
\otimes\ldots\otimes F_3
\otimes f_2)_{\bA_{j_1},\ldots,
\bA_{j_{k-1}},
\bB_{j_k}}(\ul{X}^1,\ldots,\ul{X}^{k+1})
(Z^1,\ldots,Z^k)
\\+(f_1\otimes F_3
\otimes\ldots\otimes
F_3\otimes f_3)_{\bA_{j_1},\ldots,
\bA_{j_k}}(\ul{X}^1,\ldots,\ul{X}^{k+1})
(Z^1,\ldots,Z^k),
\end{multline*}
i.e., that
\begin{multline*}
\Delta_{j_k}\ldots
\Delta_{j_1}\cR
=(f_1\otimes F\otimes\ldots
\otimes F_3
\otimes f_2)_{\bA_{j_1},\ldots,
\bA_{j_{k-1}},
\bB_{j_k}} \\ +(f_1\otimes F_3
\otimes\ldots\otimes
F_3\otimes f_3)_{\bA_{j_1},
\ldots,\bA_{j_k}}
\end{multline*}
and hence --- using 
(\ref{eq:16Jul19b}) in Lemma 
\ref{lem:12Ap17a} ---
\begin{multline*}
\Delta_{j_{k+1}}\Delta_{j_k}\ldots
\Delta_{j_1}
\cR(\ul{X}^1,\ldots,\ul{X}^{k+2})
(Z^1,\ldots,Z^{k+1})
\\
=\Delta_{j_{k+1}}\big[(f_1\otimes F_3
\otimes\ldots\otimes F_3\otimes  
f_2)_{\bA_{j_1},\ldots,\bA_{j_{k-1}},\bB_{j_k}}
\\+
(f_1\otimes F_3
\otimes\ldots\otimes F_3\otimes f_3)_{\bA_{j_1},
\ldots,\bA_{j_k}}\big]
(\ul{X}^1,\ldots,\ul{X}^{k+2})
(Z^1,\ldots,Z^{k+1})
\\=
f_1(\ul{X}^{1})(Z^1)\bA_{j_1}
\ldots(Z^{k-1})\bA_{j_{k-1}} 
F_3(\ul{X}^k)
\big[(Z^k)
\bB_{j_k}\Delta_{j_{k+1}} f_2(\ul{X}^{k+1},\ul{X}^{k+2})(Z^{k+1})
\\+(Z^k)\bA_{j_k}\Delta_{j_{k+1}}f_3(\ul{X}^{k+1},
\ul{X}^{k+2})(Z^{k+1})\big]
\\=f_1(\ul{X}^{1})(Z^1)\bA_{j_1}
\cdots
(Z^{k-1})\bA_{j_{k-1}}F_3(\ul{X}^k)
(Z^k)\bA_{j_k}\\
\big[F_3(\ul{X}^{k+1})
(Z^{k+1})\bB_{j_{k+1}}+F_3(\ul{X}^{k+1})(Z^{k+1})
\bA_{j_{k+1}}F_3(\ul{X}^{k+2})F_4(\ul{X}^{k+2})\big]
\\=
\big[(f_1\otimes F_3\otimes\ldots\otimes F_3\otimes 
f_2)_{\bA_{j_1},\ldots,\bA_{j_k},\bB_{j_{k+1}}}
\\
+(f_1\otimes F_3\otimes\ldots\otimes F_3\otimes 
f_3)_{\bA_{j_1},\ldots,\bA_{j_{k+1}}}\big]
(\ul{X}^1,\ldots,\ul{X}^{k+2})
(Z^1,\ldots,Z^{k+1}),
\end{multline*}
i.e.,
\begin{multline*}
\Delta_{j_{k+1}}\ldots\Delta_{j_1}
\cR=(f_1\otimes F_3\otimes\ldots\otimes F_3\otimes 
f_2)_{\bA_{j_1},\ldots,\bA_{j_k},\bB_{j_{k+1}}}
\\
+(f_1\otimes F_3\otimes\ldots\otimes F_3\otimes 
f_3)_{\bA_{j_1},\ldots,\bA_{j_{k+1}}},
\end{multline*}
which ends the proof by induction.
Since 
$F_3(I_m\otimes\ul{Y})=I_{Lm}$, 
we know 
\begin{multline*}
\Delta_{j_k}\ldots\Delta_{j_1}
\cR(\ul{X}^1,I_m\otimes \ul{Y},\ldots,I_m\otimes \ul{Y},\ul{X}^{k+1})
(Z^1,\ldots,Z^k)
\\=f_1(\ul{X}^1)(Z^1)
\bA_{j_1}\cdots(Z^{k-1})
\bA_{j_{k-1}}(Z^k)\bB_{j_k}
\\
+f_1(\ul{X}^1)(Z^1)\bA_{j_1}\ldots(Z^k)
\bA_{j_k}f_3(\ul{X}^{k+1}).
\end{multline*}
Therefore, as
$f_3(I_m\otimes \ul{Y})=0$, 
we obtain formula 
(\ref{eq:6Aug19b}).
\iffalse
\begin{align*}
\Delta_{j_k}\ldots\Delta_{j_1}
\cR(\ul{X}^1,\ul{Y},\ldots,\ul{Y},\ul{Y})
(Z^1,\ldots,Z^{k-1},\ul{Z})
&=(I_{m_1}\otimes C)F_3(\ul{X}^1)
\\(Z^1,\ldots,Z^{k-1})\ul{\bA}^{
\omega_1}(Z_{j_k})\bB_{j_k}
\end{align*}
where $\omega_1=g_{j_1}
\ldots g_{j_{k-1}}$
and  
\begin{multline*}
\Delta_{j_k}\ldots\Delta_{j_1}
\cR(I_m\otimes \ul{Y},I_m\otimes 
\ul{Y},\ldots,I_m\otimes 
\ul{Y},\ul{X}^{k+1})
(Z^1,\ldots,Z^k)
\\-\Delta_{j_k}\ldots\Delta_{j_1}
\cR(I_m\otimes \ul{Y},I_m\otimes \ul{Y},
\ldots,I_m\otimes \ul{Y},
I_m\otimes \ul{Y})
(Z^1,\ldots,Z^k)
\\=(I_{m_1}\otimes C)
(Z^1,\ldots,Z^k)
\ul{\bA}^{\nu}F_3(\ul{X}^{k+1})
F_4(\ul{X}^{k+1})
\end{multline*}
where 
$\nu=g_{j_1}\ldots g_{j_k}.$ 
\fi
Next, similarly to the above computations, one easily gets
\begin{multline*}
\Delta_{j_k}\ldots\Delta_{j_1}
F_3(\ul{X}^1,\ldots,\ul{X}^{k+1})(Z^1,\ldots,Z^k)
\\
=
F_3(\ul{X}^1)(Z^1)\bA_{j_1}F_3(\ul{X}^2)
\cdots F_3(\ul{X}^k)(Z^k)\bA_{j_k}F_3(\ul{X}^{k+1})
\end{multline*}
and hence by taking 
$\ul{X}^1=\ldots=\ul{X}^k=I_m\otimes\ul{Y},$ we obtain
(\ref{eq:6Aug19c}).
\iffalse
$$\Delta_{j_k}\ldots
\Delta_{j_1}(F_2F_3)(\ul{Y},\ldots,\ul{Y},\ul{X}^{k+1})
=(I_m\otimes C)(Z^1,\ldots,Z^k)
\bA^{\omega}F_3(X^{k+1})$$
\fi
\end{proof}
Finally, we are in a position to take an advantage of the nc difference-differential
calculus analysis developed in this section and prove that 
$\Omega(\cR)$ is similarity invariant,
a property that will be quite useful for us later on (cf. the proof of Theorem
\ref{thm:26Jun19a}).
\begin{theorem}
\label{thm:6Aug19a}
Let $\cR$
be a nc Fornasini--Marchesini realization centred at 
$\ul{Y}\in\KK^{\sts}$ 
that is described by 
$(L;D,C,\ul{\bA},\ul{\bB})$. 
Suppose 
$\cR$ 
is controllable and observable, and 
its coefficients satisfy the 
$\cL-\cL\cA$ conditions.
Then
$\Omega(\cR)$ is similarity invariant, i.e., 
$\ul{X}\in\Omega_{sm}(\cR)\Longrightarrow 
S\cdot\ul{X}\cdot S^{-1}\in\Omega_{sm}(\cR)$
for any invertible 
$S\in\KK^{sm\times sm}$ and $m\in\NN$. 
\end{theorem}
For the purpose of the proof,
we recall the definitions of the controllability and observability matrices, as they appear in 
\cite[Subsection 2.2]{PV1}.
The infinite block matrix 
\begin{align*}
\fC_{\ul{\bA},\ul{\bB}}:=
row\begin{bmatrix}
\fC_{\ul{\bA},\ul{\bB}}^{(\omega,k)}\\
\end{bmatrix}_{(\omega,k)\in
\cG_d\times\{1,\ldots,d\}}
\end{align*}
is the controllability matrix
associated with 
$(\ul{\bA},\ul{\bB})$,
where 
\begin{align*}
\fC_{\ul{\bA},\ul{\bB}}^{(\omega,k)}
\in\KK^{L\times s^3(|\omega|+1)}
\text{ is given by }
\fC_{\ul{\bA},\ul{\bB}}^{(\omega,k)}
:=row
\begin{bmatrix}
(\ul{\bA}^\omega
\cdot\bB_k)(\ul{Z})
\end{bmatrix}_{\ul{Z}
\in\cE_s^{|\omega|+1}}
\end{align*}
for each 
$(\omega,k)\in
\cG_d\times \{1,\ldots,d\}$.
\iffalse i.e.,
\begin{align*}
\fC_{\ul{\bA},\ul{\bB}}^{(\omega,k)}=
\begin{pmatrix}
\left(\ul{\bA}^\omega\cdot\bB_k\right)\left(
E_{11}^{\otimes(|\omega|+1)}\right)
&\ldots
&(\ul{\bA}^\omega\cdot\bB_k)
\left(E_{ss}^{\otimes{|\omega|+1}} 
\right)
\end{pmatrix}.
\end{align*}
\fi
The infinite block matrix
\begin{align*}
\fO_{C,\ul{\bA}}:=col
\begin{bmatrix}
\fO_{C,\ul{\bA}}^{(\omega)}
\end{bmatrix}_{\omega\in\cG_d}
\end{align*}
is the observability matrix 
associated to
$(C,\ul{\bA})$, where
\begin{align*}
\fO_{C,\ul{\bA}}^{(\omega)}\in
\KK^{s^{3}|\omega|\times L}
\text{ is given by }
\fO_{C,\ul{\bA}}^{(\omega)}:=col
\begin{bmatrix}
C\cdot\ul{\bA}^\omega(\ul{Z})
\end{bmatrix}_{\ul{Z}\in\cE_s^{|\omega|}}
\end{align*}
 for each
$\omega\in\cG_d$.
\begin{proof}
If 
$m\in\NN$ 
and
$\ul{X}\in\Omega_{sm}(\cR)$,  
then 
$\Lambda_{\ul{\bA},\ul{Y}}(\ul{X})$ is invertible.
As 
$(C,\ul{\bA})$ is observable and $(\ul{\bA},\ul{\bB})$
is controllable, it follows from 
\cite[Proposition
2.10]{PV1}
that there exist
$\ell\in\NN$ 
for which the matrix 
$$\fc_{\ell}:=
row 
\begin{bmatrix}
\fC_{\ul{\bA},\ul{\bB}}^{(\omega,k)}\\
\end{bmatrix}_{|\omega|\le\ell,1\le k\le d}$$ 
is right invertible and the matrix 
$$\fo_{\ell}:=
col
\begin{bmatrix}\fO_{C,\ul{\bA}}^{(\omega)}\\
\end{bmatrix}_{|\omega|\le\ell}$$
is left invertible. Let us denote a right inverse of 
$\fc_{\ell}$ by 
$\fc^{(R)}$ 
and a left inverse of 
$\fo_{\ell}$ by $\fo^{(L)}$,
thus
$\fc_{\ell}
\fc^{(R)}=
\fo^{(L)}\fo_{\ell}=I_L.$
As 
$\ul{X}\in\Omega_{sm}(\cR)$, 
it follows from Lemma 
\ref{lem:6Aug19b}
that
\begin{multline*}
\Delta_{j_k}\ldots\Delta_{j_1}
\cR(\ul{X},I_m\otimes \ul{Y},
\ldots,I_m\otimes \ul{Y},
I_m\otimes \ul{Y})
(Z^1,\ldots,Z^k)\\=(I_m\otimes C)
\Lambda_{\ul{\bA},\ul{Y}}(\ul{X})^{-1}
(Z^1,\ldots,Z^{k-1})\bA^{\omega_1}(Z^k)\bB_{j_k}
\end{multline*}
for every choice of 
$k\ge 1,\,1\le j_1,\ldots,j_k\le d$,
and
$Z^1,\ldots,Z^k\in\cE_{sm}$,
therefore when   considering all words 
$\omega\in\cG_d$ with $|\omega|\le\ell$ 
and putting them all together (as blocks) in a (row) matrix, we get 
\begin{align}
\label{eq:3Mar19a}
\Delta_{\cR}^{(\ell)}(\ul{X})
=(I_m\otimes C)
\Lambda_{\ul{\bA},\ul{Y}}
(\ul{X})^{-1}\fc_{\ell},
\end{align}
where 
\begin{multline*}
\Delta_{\cR}^{(\ell)}(\ul{X}):=row
\begin{bmatrix}
\Delta_{\cR}^{(\omega)}(\ul{X})\\
\end{bmatrix}_{|\omega|\le\ell} 
\text{ and }\\
\Delta_{\cR}^{(\omega)}(\ul{X}):=
row
\begin{bmatrix}
\Delta^{\omega}\cR
\Big(\ul{X},I_m\otimes \ul{Y},\ldots, 
I_m\otimes \ul{Y}\Big)
\Big(Z^1,\ldots,Z^{|\omega|}\Big)\\
\end{bmatrix}_{Z^1,\ldots,Z^{|\omega|}\in
\cE_{sm}}.
\end{multline*}
\iffalse
\begin{align*}
\Lambda_{\cR,\omega_2}(X):=\
\begin{bmatrix}
\Delta_1\cR(X,Y)(E_{1,1})^T\\
\vdots\\
\Delta_1\cR(X,Y)(E_{sm,sm})^T\\
\vdots\\
\Delta_d\cR(X,Y)(E_{1,1})^T\\
\vdots\\
\Delta_d\cR(X,Y)(E_{sm,sm})^T\\
\Delta_1\Delta_1
\cR(X,Y,Y)(E_{1,1},E_{1,1})^T\\
\vdots\\
\Delta_1\Delta_1
\cR(X,Y,Y)(E_{sm,sm},E_{sm,sm})^T\\
\Delta_d\Delta_d
\cR(X,Y,Y)(E_{1,1},E_{1,1})^T\\
\vdots\\
\Delta_d\Delta_d(
\cR(X,Y,Y)(E_{sm,sm},E_{sm,sm})^T\\
\vdots\\
\Delta^{\omega_2g_r}
\cR(X,Y^{\odot_{sm}|\omega_2|})
(E_{sm,sm}^{\odot_{sm}|\omega_2|})^T\\
\end{bmatrix}^T.
\end{align*} \fi
Multiplying both sides of
(\ref{eq:3Mar19a}) on the right by 
$\fc^{(R)}$, to obtain
\begin{align}
\label{eq:14Oct17b}
\Delta_{\cR}^{(\ell)}(\ul{X})
\fc^{(R)}=(I_m\otimes C)
\Lambda_{\ul{\bA},\ul{Y}}(\ul{X})^{-1}.
\end{align} 
Next, we consider all higher derivatives w.r.t
words of length at most 
$\ell$ 
of the nc function
$\Delta_{\cR}^{(\ell)}(\cdot)$; by putting them all together (as blocks) in a (column) matrix, and by using 
(\ref{eq:14Oct17b}) 
and
(\ref{eq:6Aug19c}), 
we observe that
\begin{align}
\label{eq:3Mar19c}
\Delta_{\Delta_\cR}^{(\ell)}(\ul{X})\fc^{(R)}
=\fo_{\ell}
\Lambda_{\ul{\bA},\ul{Y}}(\ul{X})^{-1},
\end{align}
where
\begin{multline*}
\Delta_{\Delta_\cR}^{(\ell)}(\ul{X}):=col
\begin{bmatrix}
\Delta_{\Delta_\cR}^{(\omega)}(\ul{X})\\
\end{bmatrix}_{|\omega|\le\ell}
\\ 
\text{ and }
\Delta_{\Delta_\cR}^{(\omega)}(\ul{X})=col
\begin{bmatrix}
\Delta^{\omega}\Delta_{\cR}^{(\ell)}
\Big(\ul{Y}^{\odot_{sm}|\omega|},
\ul{X}\Big)\Big(\ul{Z}^{\odot_{sm}|\omega|}\Big)\\
\end{bmatrix}_{\ul{Z}\in\cE_{sm}^{|\omega|}}.
\end{multline*}
\iffalse
\begin{align}
\Lambda(X):=
\begin{bmatrix}
\Delta_1\Lambda_{
\cR,\omega_2}(Y,X)(E_{1,1})\\
\vdots\\
\Delta_1\Lambda_{
\cR,\omega_2}(Y,X)(E_{sm,sm})\\
\vdots\\
\Delta_d\Lambda_{
\cR,\omega_2}(Y,X)(E_{1,1})\\
\vdots\\
\Delta_d\Lambda_{
\cR,\omega_2}(Y,X)
(E_{sm,sm})\\
\Delta_1\Delta_1\Lambda_{
\cR,\omega_2}(Y,Y,X)
(E_{1,1},E_{1,1})\\
\vdots\\
\Delta_1\Delta_1\Lambda_{
\cR,\omega_2}(Y,Y,X)
(E_{sm,sm},E_{sm,sm})\\
\Delta_d\Delta_d\Lambda_{
\cR,\omega_2}(Y,Y,X)
(E_{1,1},E_{1,1})\\
\vdots\\
\Delta_d\Delta_d\Lambda_{
\cR,\omega_2}(Y,Y,X)
(E_{sm,sm},E_{sm,sm})\\
\vdots\\
\Delta^{\omega_2}\Lambda_{
\cR,\omega_2}\\
\end{bmatrix}.
\end{align}\fi
Multiplying both sides of
(\ref{eq:3Mar19c}) on the left by 
$\fo^{(L)}$, to obtain
$\fo^{(L)}
\Delta_{\Delta_\cR}^{(\ell)}(\ul{X})
\fc^{(R)}=
\Lambda_{\ul{\bA},\ul{Y}}(\ul{X})^{-1}$,
i.e., for every 
$\ul{X}\in \Omega_{sm}(\cR)$ 
we showed that
\begin{align}
\label{eq:3Mar19d}
\fo^{(L)}
\Delta_{\Delta_\cR}^{(\ell)}(\ul{X})
\fc^{(R)}
\Lambda_{\ul{\bA},\ul{Y}}(\ul{X})=I_{Lm}.
\end{align}
Thus, for every invertible
$S\in\KK^{sm\times sm}$ 
such that 
$S\cdot\ul{X}\cdot S^{-1}
\in\Omega_{sm}(\cR)$, 
i.e., such that
\begin{align}
\label{eq:11Sep17b}
\det 
\big(\Lambda_{\ul{\bA},\ul{Y}}
(S\cdot\ul{X}\cdot S^{-1})
\big)=
\det\Big(I_{Lm}-\sum_{k=1}^d 
(SX_kS^{-1}-I_m\otimes
 Y_k)\bA_k\Big)\ne0,
\end{align}
 one can apply 
(\ref{eq:3Mar19d}) 
 for 
 $S\cdot\ul{X}\cdot S^{-1}$ 
 to obtain that 
\begin{align}
\label{eq:11Sep17a}
\fo^{(L)}
\Delta_{\Delta_\cR}^{(\ell)}
(S\cdot\ul{X}\cdot S^{-1})
\fc^{(R)}
\Lambda_{\ul{\bA},\ul{Y}}
(S\cdot\ul{X}\cdot S^{-1})=I_{Lm}.
\end{align}
We proved that 
(\ref{eq:11Sep17a}) 
holds whenever
(\ref{eq:11Sep17b})
holds and next we show that
(\ref{eq:11Sep17a}) holds for any invertible 
$S\in\KK^{sm\times sm}$, even without assuming 
(\ref{eq:11Sep17b}).

As 
$\cR:\Omega(\cR)\rightarrow\KK_{nc}$ 
is a nc function, there exists a unique nc function
$\wt{\cR}:\wt{\Omega}(\cR)\rightarrow 
\KK_{nc}$ such that 
$\wt{\cR}\mid_{\Omega(\cR)}=\cR,$
where 
$\wt{\Omega}(\cR)$
is the similarity invariant envelope of 
$\Omega(\cR)$, 
that is the smallest nc set that contains 
$\Omega(\cR)$ 
and is similarity invariant
(cf. \cite[Appendix A]{KV1}). 
Let
$\varphi:\KK^{sm\times sm}_{inv}
\rightarrow\KK^{Lm\times Lm}$ be the mapping defined by
$$\varphi(S)=\fo^{(L)}
\Delta_{\Delta_{\wt{\cR}}}^{(\ell)}
(S\cdot\ul{X}\cdot S^{-1})
\fc^{(R)}
\Lambda_{\ul{\bA},\ul{Y}}
(S\cdot\ul{X}\cdot S^{-1})-I_{Lm}$$
for every $S\in\KK^{sm\times sm}_{inv}$, 
i.e., for every invertible matrix 
$S\in\KK^{sm\times sm}$,
where 
$\ul{X},I_m\otimes \ul{Y}\in\Omega_{sm}(\cR)$ 
imply that 
$S\cdot\ul{X}\cdot S^{-1},
I_m\otimes\ul{Y}\in\wt{\Omega}_{sm}(\cR)$ 
and hence
$\Delta_{\Delta_{\wt{\cR}}}^{(\ell)}
(S\cdot\ul{X}\cdot S^{-1})$
is well defined. 
If 
$S\in\KK^{sm\times sm}_{inv}$ such that  
(\ref{eq:11Sep17b})
holds, then 
$S\cdot\ul{X}\cdot S^{-1}\in 
\Omega_{sm}(\cR)$ 
which implies that 
$\Delta_{\Delta_{\wt{\cR}}}^{(\ell)}
(S\cdot\ul{X}\cdot S^{-1})
=\Delta_{\Delta_{\cR}}^{(\ell)}
(S\cdot\ul{X}\cdot S^{-1})$ and thus $\varphi(S)=0$.
As the set
$$\cZ=
\big\{ S\in 
\KK^{sm\times sm}: 
\det(S)\ne0,\,
\det\big(\Lambda_{\ul{\bA},\ul{Y}}
(S\cdot\ul{X}\cdot 
S^{-1})\big)\ne0\big\}$$
is a non-empty (since
$I_{sm}\in\cZ$) Zariski open subset of $\KK^{sm\times sm}_{inv}$
on which $\varphi\mid_{\cZ}=0$,
while on the other hand the equation $\varphi(S)=0$  is a polynomial expression
w.r.t (entries of)
$S$ 
and to 
$1/\det(S)$, we deduce that 
$\varphi(S)=0$ holds true for every $S\in\KK^{sm\times sm}_{inv}$.
Therefore the matrix 
$\Lambda_{\ul{\bA},\ul{Y}}
(S\cdot\ul{X}\cdot S^{-1})$
is invertible, which means that 
$S\cdot\ul{X}\cdot S^{-1}\in\Omega_{sm}(\cR)$.
\end{proof}

\section{Minimal Realizations \& NC Rational Functions: 
Full Characterization}
\label{sec:mainR}
\subsection{Linearized lost-abbey conditions w.r.t algebras}
We begin this section by carrying on some of the computations regarding the
$\cL-\cL\cA$ 
conditions, to evaluations over algebras.
Let 
$\cA$
be a unital
$\KK-$algebra.
As in 
\cite{PV1},
we extend the linear mappings 
$$\bA_k:\KK^{s\times s}
\rightarrow\KK^{L\times L}
\text{ and }\bB_k:\KK^{s\times s}
\rightarrow\KK^{L\times s},\,
k=1,\ldots,d$$ 
in a natural way to linear mappings
$$\bA_k^{\cA}:\cA^{\sts}
\rightarrow\cA^{L\times L}
\text{ and }\bB_k^{\cA}:\cA^{\sts}
\rightarrow\cA^{L\times s},\,
k=1,\ldots,d$$ 
by the following rule:
if $\fA\in\cA^{\sts}$, 
it can be written as 
$\fA=\sum_{p,q=1}^s 
E_{pq}\otimes \fa_{pq}$ where 
$\fa_{pq}\in\cA$ and
$E_{pq}\in\cE_s$,
thus we define 
$$(\fA)\bA_k^{\cA}:=
\sum_{p,q=1}^s 
\bA_k(E_{pq})\otimes \fa_{pq}
\in\cA^{L\times L}
\text{ and }
(\fA)\bB_k^{\cA}:=
\sum_{p,q=1}^s 
\bB_k(E_{pq})\otimes \fa_{pq}
\in\cA^{L\times s}$$
for $1\le k\le d$.
It is easy to check that if any of the 
$\cL-\cL\cA$ conditions (cf. equations 
$(\ref{eq:12Jun17a})-(\ref{eq:12Jun17d})$) 
holds over 
$\KK^{\sts}$, 
then it holds also over 
$\cA^{\sts}$.
More precisely, the  
linearized lost-abbey 
($\cL-\cL\cA$) 
conditions w.r.t 
$\cA$ are given by
\begin{align}
\label{eq:12Jun17a3}
S_{\cA}D_{\cA}-D_{\cA}
S_{\cA}=C_{\cA}
\sum_{k=1}^d 
\big([S,Y_k]_{\cA}\big)\bB_k^{\cA},
\end{align}
\begin{align}
\label{eq:12Jun17b3}
(SC)_{\cA}
(\fA_1)\bB_{i_1}^{\cA}-C_{\cA}
\big(S_{\cA}\fA_1\big)\bB_{i_1}^{\cA}
=C_{\cA}\Big(\sum_{k=1}^d
\big([S,Y_k]_{\cA}\big)\bA_k^{\cA}\Big)
(\fA_1)\bB_{i_1}^{\cA},
\end{align}
\begin{align}
\label{eq:12Jun17c3}
(SC)_{\cA}(\fA_1)
\bA_{i_1}^{\cA}-C_{\cA}
\big(S_{\cA}\fA_1\big)\bA_{i_1}^{\cA}
=C_{\cA}\Big(\sum_{k=1}^d
\big([S,Y_k]_{\cA}\big)
\bA_k^{\cA}\Big)(\fA_1)\bA_{i_1}^{\cA},
\end{align}
\begin{align}
\label{eq:12Jun17f3}
(\fA_1S_{\cA})\bB_{i_1}^{\cA}-(\fA_1)\bB_{i_1}^{\cA}S_{\cA}
=(\fA_1)\bA_{i_1}^{\cA}\sum_{k=1}^d
\big([S,Y_k]_{\cA}\big)\bB_k^{\cA},
\end{align}
\begin{align}
\label{eq:12Jun17e3}
(\fA_1S_{\cA})\bA_{i_1}^{\cA}(\fA_2)\bB_{i_2}^{\cA}
-(\fA_1)\bA_{i_1}^{\cA}
(S_{\cA}\fA_2)\bB_{i_2}^{\cA}
=(\fA_1)\bA_{i_1}^{\cA}
\Big( 
\sum_{k=1}^d
\big([S,Y_k]_{\cA}\big)\bA_k^{\cA}\Big)
(\fA_2)\bB_{i_2}^{\cA},
\end{align}
and
\begin{align}
\label{eq:12Jun17d3}
(\fA_1S_{\cA})\bA_{i_1}^{\cA}(\fA_2)\bA_{i_2}^{\cA}
-(\fA_1)\bA_{i_1}^{\cA}
(S_{\cA}\fA_2)\bA_{i_2}^{\cA}
=(\fA_1)\bA_{i_1}^{\cA}
\Big( 
\sum_{k=1}^d
\big([S,Y_k]_{\cA}\big)\bA_k^{\cA}
\Big)(\fA_2)\bA_{i_2}^{\cA},
\end{align}
for every 
$S\in\KK^{\sts},\fA_1,\fA_2\in\cA^{\sts}$,
and 
$1\le i_1,i_2\le d$,
where for the sake of simplicity, 
we frequently use the notation
$$Z_{\cA}:=Z\otimes 
{\bf1}_{\cA}\in
\cA^{\alpha\times\beta},$$
for every 
$\alpha,\beta\in\NN$ 
and 
$Z\in\KK^{\alpha\times\beta}$.
In the following lemma we prove that equation 
$(\ref{eq:12Jun17d})$
implies equation
$(\ref{eq:12Jun17d3})$,
while the proofs for the other 
equations are similar, hence omitted.

\begin{lemma}
Let 
$\cA$ 
be a unital
$\KK-$algebra.
If equation
$(\ref{eq:12Jun17d})$
holds, then
$(\ref{eq:12Jun17d3})$ holds.
\end{lemma}
\begin{proof}
If 
$\fA_1,\fA_2\in\cA^{\sts}$, 
one can write 
$$\fA_1=\sum_{p,q=1}^s 
E_{pq}\otimes \fa^{(1)}_{pq}
\text{ and }\fA_2=\sum_{k,l=1}^s 
E_{kl}\otimes \fa^{(2)}_{kl},$$
where
$E_{pq}\in\cE_s$
and 
$\fa^{(1)}_{pq},\fa^{(2)}_{pq}\in\cA$ 
for 
$1\le p,q\le s$, 
then
\begin{multline*}
\big(\fA_1S_{\cA}\big)
\bA_{i_1}^{\cA}(\fA_2)
\bA_{i_2}^{\cA}-(\fA_1)\bA_{i_1}^{\cA}
\big(S_{\cA}\fA_2\big)\bA_{i_2}^{\cA}
\\=\Big(\sum_{p,q=1}^s 
(E_{pq}S)\otimes \fa^{(1)}_{pq}\Big)
\bA_{i_1}^{\cA}\Big(\sum_{k,l=1}^s
E_{kl}\otimes \fa^{(2)}_{kl}\Big)\bA_{i_2}^{\cA}
\\ -
\Big(\sum_{p,q=1}^sE_{pq}
\otimes \fa^{(1)}_{pq}\Big)
\bA_{i_1}^{\cA}\Big(\sum_{k,l=1}^s
(SE_{kl})\otimes \fa^{(2)}_{kl}\Big)\bA_{i_2}^{\cA}
\\=\sum_{p,q=1}^s\sum_{k,l=1}^s\Big[
\left(\bA_{i_1}(E_{pq}S)\otimes \fa^{(1)}_{pq}
\right)\left(
\bA_{i_2}(E_{kl})\otimes \fa_{kl}^{(2)}
\right)
\\ -
\left(\bA_{i_1}(E_{pq})
\otimes \fa_{pq}^{(1)}\right)
\left(
\bA_{i_2}(SE_{kl})\otimes \fa_{kl}^{(2)}\right)\Big]
\\=\sum_{p,q=1}^s\sum_{k,l=1}^s\Big[
\Big(\bA_{i_1}(E_{pq}S)
\bA_{i_2}(E_{kl})-\bA_{i_1}(E_{pq})
\bA_{i_2}(SE_{kl})\Big)
\otimes
\left(\fa_{pq}^{(1)}
\fa_{kl}^{(2)}\right)\Big]
\end{multline*}
Applying 
$(\ref{eq:12Jun17d})$, 
we get
\begin{multline*}
\big(\fA_1S_{\cA}\big)
\bA_{i_1}^{\cA}(\fA_2)
\bA_{i_2}^{\cA}-(\fA_1)\bA_{i_1}^{\cA}
\big(S_{\cA}\fA_2\big)\bA_{i_2}^{\cA}
\\=\sum_{p,q=1}^s\sum_{k,l=1}^s
\Big[
\Big(\bA_{i_1}(E_{pq})
\Big(\sum_{k=1}^d
\bA_k([S,Y_k]) 
\Big)
\bA_{i_2}(E_{kl})\Big)
\otimes \Big(
\fa^{(1)}_{pq} \fa^{(2)}_{kl}\Big)\Big]
\\=\Big(\sum_{p,q=1}^s 
\bA_{i_1}(E_{pq})\otimes \fa^{(1)}_{pq}\Big)
\Big(\sum_{k=1}^d
\bA_k([S,Y_k])
\otimes {\bf1}_{\cA}\Big)
\sum_{k,l=1}^s \bA_{i_2}(E_{k,l})
\otimes \fa^{(2)}_{kl}
\\=(\fA_1)\bA_{i_1}^{\cA}\Big(
\sum_{k=1}^d
\big([S,Y_k]_{\cA}
\big)\bA_k^{\cA}
\Big)
(\fA_2)\bA_{i_2}^{\cA}.
\end{multline*}
\end{proof} 
For a nc Fornasini--Marchesini realization
$\cR$
that is centred at
$\ul{Y}\in \Omega_s(\cR)$
and described by
$(L;D,C,\ul{\bA},\ul{\bB})$, 
we recall the following definitions: 
\begin{enumerate}
\item[$\bullet$]
The 
\textbf{$\cA-$domain} of $\cR$ 
(denoted by $DOM^{\cA}(\cR)$ in 
\cite[p.9]{PV1}), that is
\begin{equation*}
\Omega_{\cA}(\cR):=\Big\{
\ul{\fA}\in(\cA^{\sts})^d:
\Lambda_{\ul{\bA},\ul{Y}}^{\cA}(\ul{\fA})
\text{ is invertible in }
\cA^{L\times L}\Big\},
\end{equation*}
where 
$\cY_k:=Y_k\otimes {\bf1}_{\cA}$ 
for every
$1\le k\le d$, 
and
$$\Lambda_{\ul{\bA},\ul{Y}}^{\cA}(\ul{\fA}):=
(I_L)_{\cA}-\sum_{k=1}^d
(\fA_k-\cY_k)
\bA_k^{\cA}=(I_L)_\cA-\sum_{k=1}^d 
\big[(\fA_k)\bA_k^{\cA}-
(\bA_k(Y_k))_\cA\big].$$
\item[$\bullet$] 
The $\cA-$evaluation of 
$\cR$ 
at any
$\ul{\fA}=(\fA_1,\ldots,\fA_d)\in 
\Omega_{\cA}(\cR)$, 
which is
\begin{align*}
\cR^{\cA}(\ul{\fA}):=
D_\cA
+C_\cA
\Lambda_{\ul{\bA},\ul{Y}}^{\cA}(\ul{\fA})^{-1}
\sum_{k=1}^d 
\big[(\fA_k)\bB_k^{\cA}-
(\bB_k(Y_k))_\cA\big].
\end{align*}
\end{enumerate}

In the next proposition we show that 
in the special case where 
$\ul{\fA}=(\fA_1,\ldots,\fA_d)=
(I_s\otimes \fa_1,\ldots, 
I_s\otimes \fa_d)\in\Omega_{\cA}(\cR)$ 
for some 
$\fa_1,\ldots,\fa_d\in\cA$, 
the matrix 
$\cR^{\cA}(\ul{\fA})$ 
commutes with all the matrices of the form
$S_{\cA}=S\otimes {\bf1}_{\cA}$, 
where 
$S\in\KK^{\sts}$, 
and hence must be scalar.
This result is very important, as we will see in the proof of Theorem 
\ref{thm:26Jun19a}.
\begin{proposition}
\label{prop:3Jul17a}
Let $\cR$
be a nc Fornasini--Marchesini realization centred at 
$\ul{Y}\in\KK^{\sts}$ 
that is described by 
$(L;D,C,\ul{\bA},\ul{\bB})$,
suppose its coefficients satisfy the 
$\cL-\cL\cA$ conditions, and
let
$\cA$  
be a unital
$\KK-$algebra.
If 
$\ul{\fa}=(\fa_1,\ldots,\fa_d)\in\cA^d$  
such that 
\begin{align}
\label{eq:17Aug18b}
\Lambda_{\ul{\bA},\ul{Y}}^{\cA}(I_s\otimes\ul{\fa})=(I_L)_{\cA}-\sum_{k=1}^d
\big[\bA_k(I_s)\otimes\fa_k-(\bA_k(Y_k))_{\cA}\big]
\end{align}
is invertible in $\cA^{L\times L}$,
then 
$\cR^{\cA}(\ul{\fA})S_{\cA}=
S_{\cA}\cR^{\cA}(\ul{\fA})$
for every
$S\in\KK^{\sts}$, where
$\ul{\fA}=(\fA_1,\ldots,\fA_d)=I_s\otimes \ul{\fa}$.
In particular, there exists an element $f(\ul{\fa})\in\cA$ such that 
$\cR^{\cA}(I_s\otimes\ul{\fa})=I_s\otimes f(\ul{\fa})$. 
\end{proposition}
The ideas of the proof  are very similar to those in the proof of Theorem
\ref{thm:19Oct17a},
with the difference of using the 
$\cL-\cL\cA$ conditions w.r.t algebras now, instead
of the 
$\cL-\cL\cA$ conditions. 
Nevertheless, we present the proof fully detailed.
\begin{proof}
Let 
$S\in\KK^{\sts}$. 
For simplicity, recall the notations
$\cY_i:=Y_i\otimes {\bf1}_{\cA}$
for every
$1\le i\le d$ and  define the matrices
\begin{align*}
\fM_{A}=\sum_{k=1}^d 
\big([S,Y_k]_{\cA}\big)
\bA_k^{\cA}\in\cA^{L\times L}
\text{ and }
\fM_{B}=\sum_{k=1}^d 
\big([S,Y_k]_{\cA}\big)
\bB_k^{\cA}\in\cA^{L\times s}.
\end{align*}
From equations
$(\ref{eq:12Jun17a3})$ 
and 
$(\ref{eq:12Jun17f3})$, 
it follows that 
\begin{multline*}
\cR^{\cA}(\ul{\fA})S_{\cA}
=(DS)_{\cA}
+C_{\cA}
\Lambda_{\ul{\bA},\ul{Y}}^{\cA}(\ul{\fA})^{-1}
\Big(\sum_{i=1}^d
(\fA_i-\cY_i)
\bB_i^{\cA}\Big)
S_{\cA}
\\=(SD)_{\cA}
-C_{\cA}\fM_B
+C_{\cA}
\Lambda_{\ul{\bA},\ul{Y}}^{\cA}(\ul{\fA})^{-1}
\Big[\sum_{i=1}^d 
\big((\fA_i-\cY_i)
S_{\cA}\big)\bB_i^{\cA}
-\sum_{i=1}^d
(\fA_i-\cY_i)\bA_i^{\cA}
\fM_B\Big]
\\=\Big(SD-C
\sum_{k=1}^d 
\bB_k([S,Y_k])
\Big)_{\cA}
+C_{\cA}
\Lambda_{\ul{\bA},\ul{Y}}^{\cA}(\ul{\fA})^{-1}
\\\Big[\sum_{i=1}^d 
\big((\fA_i-\cY_i)
S_{\cA}\big)\bB_i^{\cA}
+\Lambda_{\ul{\bA},\ul{Y}}^{\cA}(\ul{\fA})
\Big(\sum_{k=1}^d 
\bB_k([S,Y_k])\Big)_{\cA}
-\Big(\sum_{k=1}^d 
\bB_k([S,Y_k])\Big)_{\cA}\Big]
\\=(SD)_{\cA}+
C_{\cA}
\Lambda_{\ul{\bA},\ul{Y}}^{\cA}(\ul{\fA})^{-1}
\sum_{i=1}^d 
\big((\fA_i-\cY_i)
S_{\cA}
-[S,Y_i]_{\cA}\big)\bB_i^{\cA}
\\=(SD)_{\cA}
+C_{\cA}
\Lambda_{\ul{\bA},\ul{Y}}^{\cA}(\ul{\fA})^{-1}
\sum_{i=1}^d 
\big(\fA_iS_{\cA}
-(SY_i)_{\cA}\big)\bB_i^{\cA}
\\ =(SD)_{\cA}
+C_{\cA}
\Lambda_{\ul{\bA},\ul{Y}}^{\cA}(\ul{\fA})^{-1}
\sum_{i=1}^d 
\big(S_{\cA}
(\fA_i-\cY_i)\big)
\bB_i^{\cA},
\end{multline*}
i.e.,
\begin{align*}
\cR^{\cA}(\ul{\fA})S_{\cA}
-(SD)_{\cA}
=C_{\cA}
\Lambda_{\ul{\bA},\ul{Y}}^{\cA}(\ul{\fA})^{-1}
\sum_{i=1}^d 
\big(S_{\cA}
(\fA_i-\cY_i)\big)
\bB_i^{\cA}.
\end{align*}
From 
$(\ref{eq:12Jun17e3})$ 
it follows that
\begin{multline*}
\big((\fA_j-\cY_j)
S_{\cA}\big)
\bA_j^{\cA}(\fA_i-\cY_i)
\bB_i^{\cA}-(\fA_j-\cY_j)
\bA_j^{\cA}
\big(S_{\cA}
(\fA_i-\cY_i)\big)\bB_i^{\cA}
\\ =(\fA_j-\cY_j)
\bA_j^{\cA}\fM_A
(\fA_i-\cY_i)\bB_i^{\cA},
\end{multline*}
thus
\begin{multline*}
\Big(\sum_{j=1}^d
\big((\fA_j-\cY_j)S_{\cA}\big)\bA_j^{\cA}\Big)
\sum_{i=1}^d(\fA_i-\cY_i)\bB_i^{\cA}-
\Big(\sum_{j=1}^d
(\fA_j-\cY_j)\bA_j^{\cA}\Big)
\sum_{i=1}^d
\big(S_{\cA}
(\fA_i-\cY_i)\big)\bB_i^{\cA}
\\=\Big(\sum_{j=1}^d(\fA_j-\cY_j)\bA_j^{\cA}\Big)
\fM_A\sum_{i=1}^d(\fA_i-\cY_i)\bB_i^{\cA}
\end{multline*}
and hence
\begin{multline*}
\Lambda_{\ul{\bA},\ul{Y}}^{\cA}(\ul{\fA})
\sum_{i=1}^d
\big(S_{\cA}(\fA_i-\cY_i)\big)\bB_i^{\cA}
-\sum_{i=1}^d
\big(S_{\cA}(\fA_i-\cY_i)\big)\bB_i^{\cA}
\\=
\Big(\sum_{j=1}^d(\fA_j-\cY_j)\bA_j^{\cA}\Big)
\fM_A
\sum_{i=1}^d(\fA_i-\cY_i)\bB_i^{\cA}
-\Big(\sum_{j=1}^d
\big((\fA_j-\cY_j)S_{\cA}\big)\bA_j^{\cA}\Big)
\sum_{i=1}^d(\fA_i-\cY_i)\bB_i^{\cA},
\end{multline*}
which implies that 
\begin{multline*}
\Lambda_{\ul{\bA},\ul{Y}}^{\cA}
(\ul{\fA})^{-1}
\sum_{i=1}^d
\big(S_{\cA}(\fA_i-\cY_i)
\big)\bB_i^{\cA}
\\=\sum_{i=1}^d
\big(S_{\cA}(\fA_i-\cY_i)
\big)\bB_i^{\cA}
-\Lambda_{\ul{\bA},\ul{Y}}^{\cA}
(\ul{\fA})^{-1}
\Big(\sum_{j=1}^d
(\fA_j-\cY_j)\bA_j^{\cA}\Big)
\fM_A\sum_{i=1}^d
(\fA_i-\cY_i)\bB_i^{\cA}
\\+\Lambda_{\ul{\bA},\ul{Y}}^{\cA}
(\ul{\fA})^{-1}
\Big(\sum_{j=1}^d
\big((\fA_j-\cY_j)
S_{\cA}\big)\bA_j^{\cA}\Big)
\sum_{i=1}^d(\fA_i-\cY_i)\bB_i^{\cA}
\\=\sum_{i=1}^d
\big(S_{\cA}(\fA_i-\cY_i)\big)
\bB_i^{\cA}+
\Lambda_{\ul{\bA},\ul{Y}}^{\cA}
(\ul{\fA})^{-1}
\Big(\sum_{j=1}^d
\big((\fA_j-\cY_j)
S_{\cA}\big)\bA_j^{\cA}\Big)
\sum_{i=1}^d
(\fA_i-\cY_i)\bB_i^{\cA}
\\+\fM_A\sum_{i=1}^d
(\fA_i-\cY_i)\bB_i^{\cA}
-\Lambda_{\ul{\bA},\ul{Y}}^{\cA}
(\ul{\fA})^{-1}
\fM_A
\sum_{i=1}^d
(\fA_i-\cY_i)
\bB_i^{\cA}
\\=\sum_{i=1}^d
\big(S_{\cA}(\fA_i-\cY_i)
\big)\bB_i^{\cA}
+\fM_A\sum_{i=1}^d
(\fA_i-\cY_i)\bB_i^{\cA}
\\+\Lambda_{\ul{\bA},\ul{Y}}^{\cA}
(\ul{\fA})^{-1}
\Big(\sum_{j=1}^d
\big((\fA_j-\cY_j)
S_{\cA}
-[S,Y_j]_{\cA}\big)
\bA_j^{\cA}\Big)
\sum_{i=1}^d(\fA_i-\cY_i)\bB_i^{\cA}.
\end{multline*}
Therefore, using 
$(\ref{eq:12Jun17b3})$,
we get
\begin{multline*}
\cR^{\cA}(\ul{\fA})
S_{\cA}-(SD)_{\cA}
\\=
\sum_{i=1}^d C_{\cA}
\Big[\big(S_{\cA}
(\fA_i-\cY_i)\big)\bB_i^{\cA}
+\fM_A(\fA_i-\cY_i)
\bB_i^{\cA}\Big]
\\+C_{\cA}
\Lambda_{\ul{\bA},\ul{Y}}^{\cA}
(\ul{\fA})^{-1}
\Big(\sum_{j=1}^d
\big(\fA_j
S_{\cA}-(SY_j)_{\cA}\big)\bA_j^{\cA}\Big)
\sum_{i=1}^d
(\fA_i-\cY_i)\bB_i^{\cA}
\\=\sum_{i=1}^d 
(SC)_{\cA}(\fA_i-\cY_i)
\bB_i^{\cA}+
C_{\cA}
\Lambda_{\ul{\bA},\ul{Y}}^{\cA}
(\ul{\fA})^{-1}
\Big(\sum_{j=1}^d
\big(S_{\cA}
(\fA_j-\cY_j)\big)\bA_j^{\cA}\Big)
\sum_{i=1}^d(\fA_i-\cY_i)\bB_i^{\cA}.
\end{multline*}
We know, from 
$(\ref{eq:12Jun17d3})$, 
that
\begin{multline*}
\Big(\sum_{i=1}^d 
(\fA_i-\cY_i)\bA_i^{\cA}\Big)
\sum_{j=1}^d
\big(S_{\cA}
(\fA_j-\cY_j)\big)
\bA_j^{\cA}
\\=\Big(\sum_{i=1}^d
\big((\fA_i-\cY_i)S_{\cA}\big)
\bA_i^{\cA}\Big)
\sum_{j=1}^d(\fA_j-\cY_j)\bA_j^{\cA}
-\Big(\sum_{i=1}^d(\fA_i-\cY_i)\bA_i^{\cA}\Big)
\fM_A
\sum_{j=1}^d(\fA_j-
\cY_j)\bA_j^{\cA}
\end{multline*}
and hence, using 
$(\ref{eq:12Jun17c3})$, 
\begin{multline*}
C_{\cA}
\Lambda_{\ul{\bA},\ul{Y}}^{\cA}
(\ul{\fA})^{-1}
\sum_{j=1}^d
\big(S_{\cA}(\fA_j-\cY_j)\big)\bA_j^{\cA}
\\=C_{\cA}
\Big(\sum_{j=1}^d
\big(S_{\cA}
(\fA_j-\cY_j)\big)\bA_j^{\cA}\Big)
+C_{\cA}\fM_A
\sum_{j=1}^d
(\fA_j-\cY_j)\bA_j^{\cA}
\\+C_{\cA}
\Lambda_{\ul{\bA},\ul{Y}}^{\cA}
(\ul{\fA})^{-1}
\Big(\sum_{i=1}^d\big(
(\fA_i-\cY_i)S_{\cA}\big)
\bA_i^{\cA}\Big)
\sum_{j=1}^d
(\fA_j-\cY_j)\bA_j^{\cA}
\\ -C_{\cA}
\Lambda_{\ul{\bA},\ul{Y}}^{\cA}
(\ul{\fA})^{-1}
\fM_A
\sum_{j=1}^d
(\fA_j-\cY_j)\bA_j^{\cA}
\\=(SC)_{\cA}\Big(
\sum_{i=1}^d 
(\fA_i-\cY_i)\bA_i^{\cA}\Big)
+C_{\cA}
\Lambda_{\ul{\bA},\ul{Y}}^{\cA}
(\ul{\fA})^{-1}
\Big(\sum_{i=1}^d
\big(S_{\cA}
(\fA_i-\cY_i)\big)\bA_i^{\cA}\Big)
\sum_{j=1}^d(\fA_j-\cY_j)\bA_j^{\cA},
\end{multline*}
which implies
\begin{align*}
C_{\cA}
\Lambda_{\ul{\bA},\ul{Y}}^{\cA}
(\ul{\fA})^{-1}
\Big(\sum_{j=1}^d
\big(S_{\cA}
(\fA_j-\cY_j)\big)\bA_j^{\cA}\Big) 
\Lambda_{\ul{\bA},\ul{Y}}^{\cA}
(\ul{\fA})=
(SC)_{\cA}
\sum_{i=1}^d 
(\fA_i-\cY_i)\bA_i^{\cA}
\end{align*}
and then
\begin{multline*}
C_{\cA}
\Lambda_{\ul{\bA},\ul{Y}}^{\cA}
(\ul{\fA})^{-1}
\sum_{j=1}^d
\big(S_{\cA}
(\fA_j-\cY_j)\big)\bA_j^{\cA}
=(SC)_{\cA}
\Big(
\sum_{i=1}^d 
(\fA_i-\cY_i)\bA_i^{\cA}\Big)
\Lambda_{\ul{\bA},\ul{Y}}^{\cA}
(\ul{\fA})^{-1}
\\=(SC)_{\cA}
\Lambda_{\ul{\bA},\ul{Y}}^{\cA}
(\ul{\fA})^{-1}
\sum_{i=1}^d
(\fA_i-\cY_i)\bA_i^{\cA}.
\end{multline*}
Finally, we see that
\begin{multline*}
\cR^{\cA}(\ul{\fA})S_{\cA}-(SD)_{\cA}
=\sum_{i=1}^d 
(SC)_{\cA}
(\fA_i-\cY_i)\bB_i^{\cA}
\\ +(SC)_{\cA}
\Lambda_{\ul{\bA},\ul{Y}}^{\cA}
(\ul{\fA})^{-1}
\Big(\sum_{i=1}^d
(\fA_i-\cY_i)\bA_i^{\cA}\Big)
\sum_{i=1}^d
(\fA_i-\cY_i)\bB_i^{\cA}
\\=(SC)_{\cA}
\Lambda_{\ul{\bA},\ul{Y}}^{\cA}
(\ul{\fA})^{-1}
\sum_{i=1}^d
(\fA_i-\cY_i)\bB_i^{\cA},
\end{multline*}
i.e.,  
$\cR^{\cA}(\ul{\fA})
S_{\cA}=
S_{\cA}\cR^{\cA}(\ul{\fA}).$ 
Then it is easily seen
--- just by choosing 
$S=E_{\alpha\beta}$ 
for all 
$1\le \alpha,\beta\le s$ ---  
that $\cR^{\cA}(\ul{\fA})$ 
is a scalar matrix over 
$\cA$, 
so for every 
$\ul{\fa}\in\cA^d$ such that
$I_s\otimes\ul{\fa}\in\Omega_{\cA}(\cR)$, 
there exists 
$f(\ul{\fa})\in\cA$ 
satisfying 
$\cR^{\cA}(I_s\otimes\ul{\fa})=I_s\otimes f(\ul{\fa})$.
\end{proof}

\subsection{Conclusions and main results}
%\label{subsec:RealNCRF}
Towards the final step of viewing a nc Fornasini--Marchesini realization
$\cR$, that is controllable and observable,
as a (restriction of a) nc rational function, we use the following
technical yet important lemma, see
\cite[Lemma 3.9]{V1}.
\begin{lemma}
\label{lem:25Jun19a}
Let 
$\fr=
\big(\fr_{ij}\big)_{1\le i,j\le n}$ 
be an 
$\ntn$ matrix over 
$\KK\plangle\ul{x}\prangle$ 
and let 
$\ul{X}\in dom(\fr)
:=\bigcap _{i,j=1,\ldots,n}
dom(\fr_{ij}).$ If 
$\det
\big(\fr(\ul{X})\big)\ne0$, 
then there exist nc rational expressions 
$S_{ij},
(i,j=1,\ldots,n)$ 
such that 
$\ul{X}\in \bigcap_{i,j=1,\ldots,n}
dom(S_{ij})$ 
and 
$S_{ij}$ 
represents 
$(\fr^{-1})_{ij}$ for every
$1\le i,j\le n$. 
\end{lemma}
Given a nc Fornasini--Marchesini realization $\cR$ centred at 
$\ul{Y}=(Y_1,\ldots,Y_d)\in(\KK^{\sts})^d$,
define the following two matrices of nc rational functions:
\begin{align}
\label{eq:26Jun19b}
\delta_{\cR}
(x_1,\ldots,x_d):=
I_L-\sum_{k=1}^d
\big[\bA_k(I_s)x_k-\bA_k(Y_k)\big]
\end{align}
in
$\KK\plangle\ul{x}\prangle^{\LTL}$,
and
\begin{multline}
\label{eq:22May20a}
\fr_{\cR}
(x_1,\ldots,x_d)=D \\ +C
\Big(I_L-\sum_{k=1}^d 
\big[\bA_k(I_s)x_k-\bA_k(Y_k)\big]\Big)^{-1}
\sum_{k=1}^d 
\big[
\bB_k(I_s)x_k-\bB_k(Y_k)\big]
\end{multline} 
in
$\KK\plangle\ul{x}\prangle^{\sts}$. 
We will show in the proof of Theorem 
\ref{thm:26Jun19a}
below that the matrix 
$\delta_{\cR}$, appearing in the formula of $\fr_{\cR}$, 
is indeed invertible over 
$\KK\plangle\ul{x}\prangle^{\LTL}$. 
\begin{theorem}
\label{thm:26Jun19a}
Let $\cR$
be a nc Fornasini--Marchesini realization
that is centred at 
$\ul{Y}\in\KK^{\sts}$ 
and described by  
$(L;D,C,\ul{\bA},\ul{\bB})$.
Suppose
$\cR$
is controllable and observable, 
and its coefficients satisfy the 
$\cL-\cL\cA$ conditions.
Then there exists 
$f\in\KK\plangle\ul{x}\prangle$ 
such that
\begin{align}
\label{eq:2Jul19a}
\fr_{\cR}=I_s\otimes f.
\end{align}
Moreover, 
$\Omega_{sm}(\cR)\subseteq dom_{sm}(f)$
and
$\cR(\ul{X})=f(\ul{X})$ 
for every
$m\in\NN$ and 
$\ul{X}\in\Omega_{sm}(\cR)$.
In particular, 
$\ul{Y}\in dom_s(f)$. 
\end{theorem}
\begin{remark}
\label{rem:11Aug19a}
The equality in 
(\ref{eq:2Jul19a})
holds in $\KK\plangle\ul{x}\prangle^{\sts}$, 
in the sense that 
$\ul{X}\in dom(\fr_{\cR})$ 
if and only if 
$\ul{X}\in \bigcap_{1\le i,j\le s}dom(R_{ij})$,
where 
$R_{ii}$ (for $1\le i\le s$) 
are nc rational expressions which represent 
$f$, 
and 
$R_{ij}$ (for 
$1\le i\ne j\le s$) 
are nc rational expressions which represent 
$0$. 
In that case, the evaluation is given by 
$\fr_{\cR}(\ul{X})=
\big(R_{ij}(\ul{X})
\big)_{1\le i,j\le s}$.
\end{remark}
\begin{proof}
In Theorem
\ref{thm:19Oct17a} we showed that
$\ul{Y}\in\Omega_s(\cR)$, while 
$\Omega(\cR)$ 
is closed under direct sums 
and similarities, where
$\cR$ 
is the nc function given by 
(\ref{eq:26Jun19c}), 
thus 
$I_s\otimes \ul{Y}\in 
\Omega_{s^2}(\cR)$ 
and as 
$\ul{Y}\otimes I_s=  E(s,s)(I_s\otimes\ul{Y})E(s,s)^{-1}$ 
(cf. equation 
(\ref{eq:24Jan19a})), we get that 
$\ul{Y}\otimes I_s\in 
\Omega_{s^2}(\cR)$, i.e, that the matrix
\begin{multline*}
\Lambda_{\ul{\bA},\ul{Y}}
(\ul{Y}\otimes I_s)=
I_{Ls}-\sum_{k=1}^d 
(Y_k\otimes I_s-I_s\otimes Y_k)
\bA_k
\\ =I_{Ls}-\sum_{k=1}^d
\big[Y_k\otimes\bA_k(I_s)
-I_s\otimes \bA_k(Y_k)
\big]
\end{multline*}
is invertible, which implies that the matrix
\begin{multline*}
\delta_{\cR}(\ul{Y})=I_{L}\otimes I_s-\sum_{k=1}^d
\big[ 
\bA_k(I_s)\otimes Y_k-
\bA_k(Y_k)\otimes I_s
\big]
\\ =E(s,L)
\Lambda_{\ul{\bA},\ul{Y}}
(\ul{Y}\otimes I_s) 
E(s,L)^{-1}
\end{multline*}
is invertible.
It is well known (e.g, it follows immediately from Lemma 
\ref{lem:25Jun19a}) 
that an element from
$\KK\plangle\ul{x}\prangle^{\LTL}$ is invertible in 
$\KK\plangle\ul{x}\prangle^{\LTL}$ 
if and only if its evaluation at some point is invertible. 
Therefore 
$\delta_{\cR}$ 
is invertible in 
$\KK\plangle\ul{x}\prangle^{\LTL}$, since the matrix 
$\delta_{\cR}(\ul{Y})$ is invertible.

By applying Proposition
\ref{prop:3Jul17a}
with the free field
$\cA=\KK\plangle\ul{x}\prangle$ and 
$\fa_1=x_1,\ldots,\fa_d=x_d$, 
the invertibility of $\delta_{\cR}$ 
in 
$\KK\plangle\ul{x}\prangle^{\LTL}$
---
which is equivalent for equation 
(\ref{eq:17Aug18b}) ---
implies that there exists 
$f\in\KK\plangle\ul{x}\prangle$ such that 
$\cR^{\cA}(I_s\otimes\ul{x})=I_s\otimes f(\ul{x})$, i.e., 
$\fr_{\cR}=I_s\otimes f.$

Next, let 
$\ul{X}\in\Omega_{sm}(\cR)$, 
similarly to the arguments above, we get that
$\delta_{\cR}(\ul{X})$ 
is invertible, thus it follows from Lemma
\ref{lem:25Jun19a} 
that  there exist nc rational expressions
$S_{ij}$ 
(for
$1\le i,j\le L)$, 
such that 
$\ul{X}\in\bigcap_{1\le i,
j\le L}dom(S_{ij})$ and the 
$\LTL$ matrix valued nc rational expression
$(S_{ij})_{1\le i,j\le L}$ 
represents 
$\delta_{\cR}^{-1}$.
Therefore, the 
$\sts$ 
matrix of nc rational expressions given by
\begin{multline*}
\wt{R}(\ul{x})=D+C
\begin{pmatrix}
S_{ij}(\ul{x})\\
\end{pmatrix}_{1\le i,j\le L}
\begin{pmatrix}
p_{ij}(\ul{x})\\
\end{pmatrix}_{1\le i\le L,\, 
1\le j\le s}
\\ =
\begin{pmatrix}
d_{ij}+\sum_{k,\ell=1}^L 
c_{ik}S_{k\ell}(\ul{x})
p_{\ell j}(\ul{x})\\
\end{pmatrix}_{1\le i,j\le s},
\end{multline*}
where     
$\begin{pmatrix}
p_{ij}\\
\end{pmatrix}_{1\le i\le L,\, 1\le j\le s}
=\sum_{k=1}^d
\big[\bB_k(I_s)x_k-\bB_k(Y_k)\big]$ 
is an 
$L\times s$ 
matrix of nc polynomials,
represents the 
$\sts$ 
matrix of nc rational functions 
$\fr_{\cR}$.
As
$\ul{X}\in\bigcap_{1\le i,j\le L}dom(S_{ij})$, it follows that 
$$\ul{X}\in dom_{sm}(\ul{e}_1^T\wt{R}\ul{e}_1)
= dom_{sm}\Big(d_{11}+\sum_{k,\ell=1}^L
c_{1k}S_{k\ell}p_{\ell1}\Big),$$
however 
$\ul{e}_1^T\wt{R}\ul{e}_1$ 
represents the nc rational function 
$\ul{e}_1^T\fr_{\cR}\ul{e}_1=\ul{e}_1^T(I_s\otimes f)\ul{e}_1=f$, 
hence 
$\ul{X}\in dom_{sm}(f)$. 
Moreover, as
\begin{multline*}
\delta_{\cR}(\ul{X})=I_{Lsm}-\sum_{k=1}^d
\big[\bA_k(I_s)\otimes X_k-
\bA_k(Y_k)\otimes I_{sm}\big]
\\=
E(L,sm)^{-1}
\Big(I_{Lsm}-
\sum_{k=1}^d
\big[X_k\otimes \bA_k(I_s)-I_{sm}
\otimes\bA_k(Y_k)\big]\Big)
E(L,sm)
\\ =E(L,sm)^{-1}
\Lambda_{\ul{\bA},\ul{Y}}
(\ul{X}\otimes I_s)
E(L,sm),
\end{multline*}
one can evaluate 
$\fr_{\cR}(\ul{X})$ in the following way
\begin{multline*}
\fr_{\cR}(\ul{X})=
D\otimes I_{sm}
+(C\otimes I_{sm})
\delta_{\cR}(\ul{X})^{-1}
\sum_{k=1}^d 
\big[\bB_k(I_s)\otimes X_k-
\bB_k(Y_k)\otimes I_{sm}\big]
\\=D\otimes I_{sm}+
(C\otimes I_{sm})
E(L,sm)^{-1}
\Lambda_{\ul{\bA},\ul{Y}}
(\ul{X}\otimes I_s)^{-1}
E(L,sm)
\\ \Big(
\sum_{k=1}^d 
E(L,sm)^{-1}\big[X_k\otimes 
\bB_k(I_s)-I_{n}\otimes\bB_k(Y_k)\big]E(s,sm)
\Big)
\\=
E(s,sm)^{-1}
\Big(I_{sm}\otimes 
D+(I_{sm}\otimes C)
\Lambda_{\ul{\bA},\ul{Y}}
(\ul{X}\otimes I_s)^{-1}
\sum_{k=1}^d (X_k\otimes 
I_s-I_{sm}\otimes Y_k)\bB_k
\Big)
\\ E(s,sm)
=E(s,sm)^{-1} 
\cR(\ul{X}\otimes I_s)
E(s,sm)
\\ =E(s,sm)^{-1}\cR
\big(E(s,sm)\cdot(I_s\otimes \ul{X})
\cdot E(s,sm)^{-1}
\big)
E(s,sm)
\\ =\cR(I_s\otimes\ul{X})
=I_s\otimes \cR(\ul{X}),
\end{multline*}
while 
$\fr_{\cR}=I_s\otimes f$ implies 
$\fr_{\cR}(\ul{X})
=I_s\otimes f(\ul{X})$, 
therefore 
$f(\ul{X})=\cR(\ul{X})$.
\end{proof}

Before we prove one of the main results of the paper, we prove two useful
properties of nc rational functions and expressions:
\begin{proposition}
\label{prop:4Jul19a}
Let $\fR_1,\fR_2\in\KK\plangle\ul{x}\prangle$, 
let $R$ 
be a nc rational expression in
$x_1,\ldots,x_d$
over
$\KK,\,
n\in\NN$, 
and 
$\ul{Z}\in(\KK^{\ntn})^d$. 
\\1.
If 
$I_m\otimes\ul{Z}\in dom_{mn}(R)$
for some $m\in\NN$, then 
$\ul{Z}\in dom_n(R)$.
\\2.
If
$\fR_1(\ul{X})=\fR_2(\ul{X})$ for every 
$\ul{X}\in dom(\fR_1)\cap dom(\fR_2)$, 
then $\fR_1=\fR_2$.
\end{proposition}
Using the terminology in
\cite[pp. 95--96]{KV1}, 
item $1$ says that 
$dom(R)$ is a radical set.
\begin{proof}
1. By synthesis: we only show the step of going from a nc rational expression
to its inverse, as the proof for the other parts (going from two nc rational
expressions to their sum and product, as well as the case of nc polynomial)
are either similar or trivial. 
Suppose that our statement is known for a nc rational expressions
$R$ 
and let 
$\ul{Z}$ 
be such that 
$I_m\otimes \ul{Z}\in 
dom_{mn}(R^{-1})$ 
for some 
$m\in\NN$.
Then 
$I_m\otimes 
\ul{Z}\in 
dom_{mn}(R)$ 
and 
$R(I_m\otimes\ul{Z})$ 
is invertible, 
but from the assumption (on 
$R$) we get that 
$\ul{Z}\in dom_n(R)$. 
As 
$R$ 
respects direct sums we get that
$R(I_m\otimes\ul{Z})=
I_m\otimes R(\ul{Z})$ 
is invertible, hence 
$R(\ul{Z})$
is invertible as well, i.e., 
$\ul{Z}\in 
dom_{n}(R^{-1})$ 
as needed. 
\\\\
2. Let $R_1\in\fR_1$ 
and 
$R_2\in\fR_2$ be (non-degenerate) nc rational expressions. 
For every 
$\ul{X}\in dom(R_1)\cap dom(R_2)$, we have
$\ul{X}\in dom(\fR_1)\cap dom(\fR_2)$ and hence $R_1(\ul{X})=\fR_1(\ul{X})=\fR_2(\ul{X})=R_2(\ul{X})$,
i.e., $R_1$ and $R_2$
are
$(\KK^d)_{nc}-$evaluation equivalent, so they represent the same nc rational
function, meaning $\fR_1=\fR_2$.
\end{proof}
We now show that Theorems 
\ref{thm:26Jun19a}
and
\ref{thm:19Oct17a}, 
imply one of our main results, 
that is --- roughly speaking --- we can evaluate
a nc rational function and its
domain, using any of
its minimal realizations,  centred at any chosen point from its domain:
\begin{theorem}
\label{thm:3Jun19a}
Let 
$\fR\in\KK\plangle x_1,\ldots,x_d\prangle$. 
For every 
integer
$s\in\NN$,
a point 
$\ul{Y}\in dom_{s}(\fR)$, and
a minimal nc Fornasini--Marchesini realization 
$\cR$ 
centred at 
$\ul{Y}$ 
of 
$\fR$,
we have
\begin{enumerate}
\item
$dom_{sm}(\fR)=\Omega_{sm}(\cR)$ 
and 
$\fR(\ul{X})=\cR(\ul{X})$, 
for every 
$\ul{X}\in dom_{sm}(\fR)$
and
$m\in\NN$.
\item
$dom_n(\fR)=\big\{ 
\ul{X}\in(\KK^{\ntn})^d: 
I_s\otimes\ul{X}\in\Omega_{sn}(\cR)\big\}$ 
and
$I_s\otimes
\fR(\ul{X})=\cR(I_s\otimes \ul{X})$,
for every $\ul{X}\in dom_n(\fR)$ and $n\in\NN$.
\iffalse
\item[3.]
$\ul{\fa}\in dom_{\cA}(\fR)\iff 
I_s\otimes\ul{\fa}\in \Omega_{\cA}(\cR)$; in that case 
$$I_s\otimes
\fR^{\cA}(\ul{\fa})
=\cR^{\cA}(I_s\otimes\ul{\fa}).$$
\fi
\end{enumerate}
\end{theorem}
\begin{proof}
In Theorem
\ref{thm:MainThmFirstPaper}
we proved that
$$dom_{sm}(\fR)\subseteq\Omega_{sm}(\cR)
\text{ and }\fR(\ul{X})=\cR(\ul{X}),\,
\forall\ul{X}\in dom_{sm}(\fR),\,
\forall m\in\NN.$$
Since 
$\fR$
admits the realization 
$\cR$, it follows from Lemma 
\ref{lem:13Dec17a} 
that the coefficients of 
$\cR$
must satisfy the 
$\cL-\cL\cA$ 
conditions, thus Theorem 
\ref{thm:26Jun19a}
guarantees the existence of a nc rational function 
$f\in\KK\plangle\ul{x}\prangle$ 
for which 
$$\Omega_{sm}(\cR)\subseteq dom_{sm}(f)
\text{ and }f(\ul{X})=\cR(\ul{X}),\,
\forall \ul{X}\in\Omega_{sm}(\cR),\,
\forall m\in\NN.$$ 
Therefore, we have
\begin{align}
\label{eq:10Nov19c}
dom_{sm}(\fR)\subseteq dom_{sm}(f)
\text{ and }\fR(\ul{X})=f(\ul{X})
\text{ for every }\ul{X}\in dom_{sm}(\fR).
\end{align}
Next, we show that 
(\ref{eq:10Nov19c})
is true for
$sm$
replaced by any 
$n\in\NN$. 
If 
$\ul{X}\in dom_n(\fR)$, 
it follows from Theorem
\ref{thm:MainThmFirstPaper}
that 
$I_s\otimes\ul{X}\in\Omega_{sn}(\cR)$ 
and 
$I_s\otimes \fR(\ul{X})= \cR(I_s\otimes\ul{X})$,
hence 
$I_s\otimes\ul{X}\in 
dom_{sn}(f)$ and $f(I_s\otimes\ul{X})=\cR(I_s\otimes\ul{X}).$ 
From the first part of Proposition 
\ref{prop:4Jul19a}
it follows that $\ul{X}\in dom_n(f)$ 
and also that
$I_s\otimes f(\ul{X})
=f(I_s\otimes\ul{X})=\cR(I_s\otimes\ul{X})
=I_s\otimes\fR(\ul{X}),$ so
$f(\ul{X})=\fR(\ul{X})$,
i.e., we showed that 
$$dom_n(\fR)\subseteq dom_n(f) 
\text{ and }
\fR(\ul{X})=f(\ul{X}),\, 
\forall 
\ul{X}\in dom_n(\fR),\, \forall n\in\NN.$$
Thus from the second part of Proposition
\ref{prop:4Jul19a}, we obtain that $\fR=f$ and thus 
$$dom_{sm}(\fR)=dom_{sm}(f)=\Omega_{sm}(\cR)
\text{ and }\fR(\ul{X})=f(\ul{X})
\text{ for every }\ul{X}\in dom_{sm}(\fR).$$
Finally, if 
$I_s\otimes \ul{X}\in\Omega_{sn}(\cR)
=dom_{sn}(\fR)$, then from the first part of Proposition 
\ref{prop:4Jul19a}, we have
$\ul{X}\in dom_n(\fR)$, 
hence 
$I_s\otimes \fR(\ul{X})=
\fR(I_s\otimes\ul{X})=
\cR(I_s\otimes \ul{X})$.
\end{proof} 

A direct consequence of the results above, which yields a 
representation  that is independent of matrix centre for all nc rational
functions, in the spirit of Cohn and Reutenauer, is now presented.
\begin{cor}
\label{cor:25Mar21a}
Let 
$\fR\in\KK\plangle x_1,\ldots,x_d\prangle$.
\begin{itemize}
\item[1.]
There exist 
$d_0\in\KK$,
$\ul{c}\in\KK^{1\times L}$,
$\ul{b_0},\ldots,\ul{b_d}\in\KK^{L\times 1}$,
and 
$M_0,\ldots,M_d\in\KK^{\LTL},$
where $L$ is the McMillan degree of 
$\fR$,
such that 
\begin{align}
\label{eq:26Mar21a}
\fR(x_1,\ldots,x_d)=
d_0+\ul{c}
\Big(M_0-\sum_{k=1}^d M_kx_k\Big)^{-1}
\Big(\ul{b_0}-\sum_{k=1}^d \ul{b_k}x_k\Big)
\end{align}
in the free skew field and
$\ul{X}\in dom_n(\fR)$ if and only if $\det\Big(
M_0\otimes I_n-\sum_{k=1}^d M_k\otimes X_k\Big)\ne0$. 
\item[2.]
There exist 
$\wt{\ul{c}}\in\KK^{1\times(L+1)}$, $\wt{M_0},
\ldots,\wt{M_d}\in\KK^{(L+1)\times(L+1)}$,
and $\wt{\ul{b}}\in\KK^{(L+1)\times 1}$
such that
\begin{align}
\label{eq:26Mar21b}
\fR(x_1,\ldots,x_d)=\wt{\ul{c}}
\Big(\wt{M_0}-\sum_{k=1}^d
\wt{M_k}x_k
\Big)^{-1}
\wt{\ul{b}}
\end{align}
in the free skew field and $\ul{X}\in dom_n(\fR)$
if and only if
$\det\Big(
\wt{M_0}\otimes I_n-\sum_{k=1}^d
\wt{M_k}\otimes X_k\Big)\ne0$.
\end{itemize}
\end{cor}
The equality in
(\ref{eq:26Mar21a}) means that
for every 
$n\in\NN$
and
$\ul{X}\in dom_n(\fR)$, we have
\begin{align}
\label{eq:25Mar21a}
\fR(\ul{X})=d_0\otimes I_n+(\ul{c}\otimes I_n)
\Big(M_0\otimes I_n-\sum_{k=1}^d M_k\otimes X_k\Big)^{-1}
\Big( \ul{b_0}\otimes I_n-
\sum_{k=1}^d\ul{b_k}\otimes X_k\Big).
\end{align}
The representation in
(\ref{eq:26Mar21b})
is well known in the literature; it was used extensively by Fliess and 
presented precisely in the papers
\cite{CR1,CR2}. We notice that such a representation does not involve a matrix
centre at all, in contrary to our theory, while the domain of the rational
function coincides with the invertibility set of the corresponding pencil.
The representation in
(\ref{eq:26Mar21b}) is clearly not minimal, but it is also not that far from
being minimal as the state space dimension differs only by $1$ from the minimal
dimension.
On the other hand, the representation in
(\ref{eq:26Mar21a}) is minimal and good in terms of the precise domain of
the function, with the disadvantage that it has not been studied earlier.
\begin{proof}
Let 
$\ul{Y}\in dom_s(\fR)$ be any arbitrary point from the domain of 
$\fR$ and as in Theorem 
\ref{thm:3Jun19a}, let 
$\cR$
be a minimal nc Fornasini--Marchesini realization of $\fR$
that is centred at 
$\ul{Y}$. From the proof of Theorem 
\ref{thm:3Jun19a} it follows that for every 
$\ul{X}\in dom_n(\fR)$ we have
\begin{multline*}
I_s\otimes \fR(\ul{X})=\fr_{\cR}(\ul{X})
\\=
D\otimes I_{n}
+(C\otimes I_{n})
\Big(I_{Ln}-\sum_{k=1}^d
\big[\bA_k(I_s)\otimes X_k-
\bA_k(Y_k)\otimes I_{n}\big]\Big)^{-1}
\\ \sum_{k=1}^d 
\big[\bB_k(I_s)\otimes X_k-
\bB_k(Y_k)\otimes I_{n}\big]
\end{multline*}
and hence 
\begin{multline*}
\fR(\ul{X})=(\ul{e_1}^T\otimes I_n)\fr_{\cR}
(\ul{X})(\ul{e_1}\otimes I_n)
\\=(\ul{e_1}^TD\ul{e_1})\otimes I_n
+
\big(\big[\ul{e_1}^TC\big]\otimes I_n\big)
\Big(\Big[I_L+\sum_{k=1}^d\bA_k(Y_k)
\Big]\otimes I_n-
\sum_{k=1}^d
\bA_k(I_s)\otimes X_k
\Big)^{-1}\\
\Big(
\Big[-\sum_{k=1}^d\bB_k(Y_k)\ul{e_1}\Big]\otimes I_n+\sum_{k=1}^d
\big[\bB_k(I_s)\ul{e_1}\big]\otimes X_k
\Big),
\end{multline*}
that is exactly the form in 
(\ref{eq:25Mar21a}) with
$d_0=\ul{e_1}^TD\ul{e_1}$,
$\ul{c}=\ul{e_1}^TC$,
$M_0=I_L+\sum_{k=1}^d
\bA_k(Y_k)$,
$\ul{b_0}=-\sum_{k=1}^d 
\bB_k(Y_k)\ul{e_1}$, and 
$M_j=\bA_j(I_s)$
and 
$\ul{b_j}=- \bB_j(I_s)\ul{e_1}$
for 
$j=1,\ldots,d$.

Finally, similarly to the proof of
Theorem 
\ref{thm:26Jun19a}, it follows that 
\begin{multline*}
\det\Big(
M_0\otimes I_n-\sum_{k=1}^d M_k\otimes X_k\Big)\ne0
\\ \iff
\det
\Big(I_{Ln}-\sum_{k=1}^d
\big[\bA_k(I_s)\otimes X_k-
\bA_k(Y_k)\otimes I_{n}\big]\Big)\ne0
\\\iff 
\det\Big(\Lambda_{\ul{\bA},\ul{Y}}(\ul{X}\otimes I_s)\Big)\ne0
\iff \ul{X}\otimes I_s \in \Omega_{sn}(\cR)
\\\iff I_s\otimes \ul{X}\in \Omega_{sn}(\cR)\iff \ul{X}\in dom_n(\fR),
\end{multline*}
where we used the fact that $\Omega(\cR)$ is invariant under similarity and
also the second part of Theorem \ref{thm:3Jun19a}.
Finally, by letting 
\begin{multline*}
\wt{\ul{c}}=\begin{bmatrix}
-\ul{c}&d_0\\\end{bmatrix},\, 
\wt{\ul{b}}=\begin{bmatrix}
0_{L\times 1}\\
1\\
\end{bmatrix},\,
\wt{M_0}=
\begin{bmatrix}
M_0&\ul{b_0}\\
0&1\\
\end{bmatrix},
\\ \text{ and }
\wt{M_j}=
\begin{bmatrix}
M_j&\ul{b_j}\\
0&0\\
\end{bmatrix} 
\text{ for every }
j=1,\ldots,d,
\end{multline*}
one can transform from the representation in
(\ref{eq:26Mar21a})
to the one in  
(\ref{eq:26Mar21b}), without changing the invertibility set of the pencils.
\end{proof}

Here is another immediate consequence of Theorems
\ref{thm:26Jun19a},
\ref{thm:3Jun19a}, and
Lemma
\ref{lem:13Dec17a}, 
which states that given a 
nc Fornasini--Marchesini realization that is controllable and observable,
it is the 
realization of a nc rational function if and 
only if its coefficients satisfy the 
$\cL-\cL\cA$ 
conditions:
\begin{theorem}
\label{thm:5Nov19a}
Let  
$\cR$ 
be a nc Fornasini--Marchesini realization centred at 
$\ul{Y}\in(\KK^{\sts})^d$ 
that is described by 
$(L;D,C,\ul{\bA},\ul{\bB})$, and suppose  $\cR$ is both controllable and
observable. 
\iffalse
given by
\begin{align*}
\cR(X)= D+C
\left(I_{L}-\sum_{i=1}^d 
\bA_i(X_i-Y_i)\right)^{-1}
\sum_{i=1}^d \bB_i(X_i- Y_i).
\end{align*}
\fi
The following are equivalent:
\begin{enumerate}
\item
There exists 
$\fR\in\KK\plangle \ul{x}\prangle$
that is regular at 
$\ul{Y}$, 
such that
$dom_{sm}(\fR)=\Omega_{sm}(\cR)$
and 
$\fR(\ul{X})=\cR(\ul{X})$,
for every 
$\ul{X}\in dom_{sm}(\fR)$ 
and 
$m\in\NN$.
\item
The coefficients of $\cR$
satisfy the 
$\cL-\cL\cA$ 
conditions (cf. equations 
(\ref{eq:12Jun17a})-(\ref{eq:12Jun17d})).
\item
There exists
$\fR\in\KK\plangle \ul{x}\prangle$
that is regular at 
$\ul{Y}$, 
such that
$dom_{n}(\fR)=
\{\ul{X}\in(\KK^{\ntn})^d: I_s\otimes\ul{X}
\in\Omega_{sn}(\cR)\}$
and 
$I_s\otimes \fR(\ul{X})=\cR(I_s\otimes\ul{X})$,
for every 
$\ul{X}\in dom_{n}(\fR)$
and 
$n\in\NN$.
\end{enumerate}
\end{theorem}
\begin{proof}
\uline{\textbf{1$\Longrightarrow$2:}}
This is an immediate corollary of Lemma
\ref{lem:13Dec17a}, as $\fR$
admits the realization $\cR$.  
\\\textbf{\uline{2$\Longrightarrow$1:}}
Suppose that the coefficients of $\cR$ satisfy the 
$\cL-\cL\cA$ 
conditions. 
By applying Theorem
\ref{thm:26Jun19a}, 
there exists 
$f\in\KK\plangle\ul{x}\prangle$
such that
\begin{align}
\label{eq:5May20a}
\Omega_{sm}(\cR)\subseteq dom_{sm}(f) 
\text{ and }
\cR(\ul{X})=f(\ul{X}),\,
\forall\ul{X}\in\Omega_{sm}(\cR)
,m\in\NN.
\end{align}
In particular we know that $\ul{Y}\in dom_s(f)$,
thus from Corollary 
\ref{cor:ExistAndUnique}
and Theorem 
\ref{thm:3Jun19a},
there exists a minimal nc Fornasini--Marchesini realization 
$\wt{\cR}$ of $f$ that is centred at
$\ul{Y}$, 
satisfying 
\begin{align}
\label{eq:5May20b}
dom_{sm}(f)=\Omega_{sm}(\wt{\cR})
\text{ and }
f(\ul{X})=\wt{\cR}(\ul{X}),\, 
\forall\ul{X}\in dom_{sm}(f),
m\in\NN.
\end{align}
From 
$(\ref{eq:5May20a})$ and 
$(\ref{eq:5May20b})$, it follows that
\begin{align*}
\Omega_{sm}(\cR)\subseteq\Omega_{sm}(\wt{\cR})
\text{ and }\cR(\ul{X})=\wt{\cR}(\ul{X}),\,
\forall\ul{X}\in\Omega_{sm}(\cR),
m\in\NN.
\end{align*} 
Since 
$Nilp_{sm}(\ul{Y})\subseteq \Omega_{sm}(\cR)$ for every
$m\in\NN$, we obtain that
\begin{align*}
\cR(\ul{X})
=\wt{\cR}(\ul{X}),\,
\forall\ul{X}\in Nilp_{sm}(\ul{Y}),
m\in\NN
\end{align*} 
and as both 
$\cR$
and
$\wt{\cR}$
are minimal nc Fornasini--Marchesini realizations centred at 
$\ul{Y}$, we can use the arguments which appear in 
the proof of \cite[Theorem 2.13]{PV1}
(which is based on the fact that the two realizations give
the same value on the nilpotent ball around 
$\ul{Y}$ and the uniqueness of
the Taylor--Taylor coefficients of their power series expansions around $\ul{Y}$),
to deduce that 
$\cR$
and
$\wt{\cR}$
are similar and hence
$\Omega_{sm}(\cR)
=\Omega_{sm}(\wt{\cR})$
for every 
$m\in\NN$.
In conclusion, 
we get that 
$dom_{sm}(f)
=\Omega_{sm}(\cR)$ 
and 
$f(\ul{X})=\cR(\ul{X})$ 
for every 
$\ul{X}\in dom_{sm}(f)$ 
and 
$m\in\NN$, 
as needed.
\\\textbf{\uline{(1)$\Longrightarrow$(3):}}
It follows from Theorem
\ref{thm:3Jun19a}, 
as 
$\fR$
admits the minimal realization
$\cR$.
\\\textbf{\uline{(3)$\Longrightarrow$(1)}}
Let 
$n=sm$ 
and let 
$\ul{X}\in 
dom_{sm}(\fR)$, 
then 
$I_s\otimes \ul{X}\in 
\Omega_{s^2m}(\cR)$ 
and 
$I_s\otimes \fR(\ul{X})
=\cR(I_s\otimes \ul{X})$.
As
$(I_s\otimes X_k-
I_{sm}\otimes Y_k)\bA_k
=I_s\otimes 
\big((X_k-I_m\otimes Y_k)
\bA_k\big)$, we have
\begin{align*}
\Lambda_{\ul{\bA},\ul{Y}}
(I_s\otimes\ul{X})=
I_s\otimes 
\Big( I_{Lm}
-\sum_{k=1}^d
(X_k-I_m\otimes Y_k)
\bA_k 
\Big)=
I_s\otimes 
\Lambda_{\ul{\bA},\ul{Y}}
(\ul{X})
\end{align*}
and thus
\begin{multline*}
I_s\otimes\ul{X}\in \Omega_{s^2m}(\cR)
\Longrightarrow 
\det
\big(
\Lambda_{\ul{\bA},\ul{Y}}
(I_s\otimes\ul{X})
\big)\ne0
\\ \Longrightarrow 
\det
\big(
\Lambda_{\ul{\bA},\ul{Y}}
(\ul{X})\big)
\ne0
\Longrightarrow 
\ul{X}\in\Omega_{sm}(\cR).
\end{multline*}
As $\cR$ 
respects direct sums (cf. Theorem 
\ref{thm:19Oct17a}),  
$I_s\otimes\fR(\ul{X})=\cR(I_s\otimes\ul{X})=
I_s\otimes \cR(\ul{X})$
and hence
$\fR(\ul{X})=\cR(\ul{X})$.
We showed that for every $\ul{X}\in dom_{sm}(\fR)$, we have 
$\ul{X}\in\Omega_{sm}(\cR)$ and that $\fR(\ul{X})=\cR(\ul{X})$,
i.e., we showed that 
$\fR$
admits the minimal nc Fornasini--Marchesini realization 
$\cR$ and thus Theorem 
\ref{thm:3Jun19a} 
implies 
the assertion in 
\textbf{1}.
\end{proof}
Using our theory of minimal nc Fornasini--Marchesini realizations of nc rational
functions, 
we now provide a short proof for the fact that the domain of regularity of
a nc rational function is
an upper admissible, similarity invariant nc set. 
Notice that proving the upper admissibility is 
a highly non trivial thing to do, however with our methods the proof becomes
easy.
\begin{cor}
\label{cor:12May20a}
If 
$\fR\in\KK\plangle\ul{x}\prangle$, 
then its domain of regularity is an upper admissible, similarity invariant
nc subset of 
 $(\KK^d)_{nc}$.
\end{cor}
\begin{proof}
Fix a point
$\ul{Y}\in dom_s(\fR)$
and a minimal nc Fornasini--Marchesini realization 
$\cR$ 
of
$\fR$, 
that is centred at 
$\ul{Y}$; 
by using Corollary
\ref{cor:ExistAndUnique} 
for instance. 
 
\textbf{\uline{$dom(\fR)$ 
is similarity invariant:}}
Let
$\ul{X}\in dom_n(\fR)$ 
and 
$T\in\KK^{\ntn}$ be invertible.
It follows from Theorem
\ref{thm:3Jun19a} that
$I_s\otimes \ul{X}\in\Omega_{sn}(\cR)$, 
while 
from (the similarity invariance of $\Omega(\cR)$, proved in) Theorem
\ref{thm:6Aug19a} 
it follows that 
$$I_s\otimes(T^{-1}\cdot\ul{X}\cdot T)
=(I_s\otimes T)^{-1}\cdot
(I_s\otimes \ul{X})\cdot
(I_s\otimes T)\in \Omega_{sn}(\cR).$$ 
By using Theorem 
\ref{thm:3Jun19a} 
again, we get  that 
$T^{-1}\cdot\ul{X}\cdot T\in dom_n(\fR)$.

\underline{\textbf{$dom(\fR)$ 
is upper admissible:}}
Let
$\ul{X}\in dom_n(\fR),\,
\ul{X}^\prime\in dom_m(\fR)$, 
and 
$\ul{Z}\in (\KK^{\ntm})^d$.
Theorem
\ref{thm:3Jun19a}
implies that
$I_s\otimes\ul{X}
\in \Omega_{sn}(\cR)$
and
$I_s\otimes\ul{X}^\prime\in\Omega_{sm}(\cR)$,
while
from (the similarity invariance of $\Omega(\cR)$, proved in) Theorem
\ref{thm:6Aug19a} we have  
$\ul{X}\otimes I_s\in\Omega_{sn}(\cR)$
and
$\ul{X}^\prime\otimes I_s\in\Omega_{sm}(\cR)$.
Next, from (the upper admissibility of 
$\Omega(\cR)$, proved in) 
Theorem
\ref{thm:19Oct17a}, we get 
$$\begin{bmatrix}
\ul{X}&\ul{Z}\\
0&\ul{X}^\prime\\
\end{bmatrix}\otimes I_s=
\begin{bmatrix}
\ul{X}\otimes I_s&\ul{Z}\otimes I_s\\
0&\ul{X}^\prime\otimes I_s\\
\end{bmatrix}\in\Omega_{s(n+m)}(\cR)
$$
and thus the similarity invariance of 
$\Omega(\cR)$ 
implies that
$I_s\otimes \begin{bmatrix}
\ul{X}&\ul{Z}\\
0&\ul{X}^\prime\\
\end{bmatrix}\in\Omega_{s(n+m)}(\cR)$.
Finally, Proposition
\ref{prop:4Jul19a}
and Theorem 
\ref{prop:4Jul19a} 
imply that 
$\begin{bmatrix}
\ul{X}&\ul{Z}\\
0&\ul{X}^\prime\\
\end{bmatrix}\in dom_{n+m}(\fR)$.
\end{proof}
\subsection{Evaluations over algebras}
\label{subsec:MainAlg}
For the sake of completion, one would like to get a similar result to
Theorem
\ref{thm:3Jun19a}
for evaluations w.r.t stably finite algebras. 
To do so, we must introduce the following definition of a matrix domain of
a nc rational function w.r.t an algebra.
For every 
$\fR\in\KK\plangle\ul{x}\prangle$,  define the matrix domain of the function
$\fR$ w.r.t an algebra $\cA$, 
also called the\textbf{ matrix 
$\cA-$domain} of $\fR$, by
\begin{multline}
\label{eq:5Nov19b}
dom_{\cA}^{Mat}(\fR):=
\big\{
\ul{\fa}\in\cA^d: \ul{\fa}\in dom_{\cA}^{Mat}(R)
\text{ for some $1\times1$ 
matrix valued}\\
\text{ nc rational expression $R$ which represents
$\fR$}\big\}.
\end{multline}
Here $dom_{\cA}^{Mat}(R)$ 
is defined  just as in 
\cite[Definition 1.1]{PV1},
using the idea of synthesis, with the only difference
that the synthesis may involve  matrix valued nc rational expressions; e.g.
$R(x_1,x_2)=
\begin{pmatrix}
1&0\\
\end{pmatrix}
\begin{pmatrix}
1&x_1\\
x_2&1\\
\end{pmatrix}^{-1} 
\begin{pmatrix}
x_1\\x_2\\
\end{pmatrix}$ 
with
$dom_{\cA}^{Mat}(R)
=\Big\{(\fa_1,\fa_2)\in\cA^2:
\begin{pmatrix}
1_{\cA}&\fa_1\\
\fa_2&1_{\cA}\\
\end{pmatrix}
\text{ is invertible in }\cA^{2\times2} 
\Big\}$.
As every nc rational expression is a $1\times1$
matrix valued nc rational expression, we automatically have 
$$dom_{\cA}(\fR)\subseteq dom_{\cA}^{Mat}(\fR).$$ 
\iffalse
Let 
$\fR\in\KK\plangle x_1,\ldots,x_d\prangle$ 
be a nc rational function,
$\cR$ be any minimal nc Fornasini--Marchesini realization of 
$\fR$ centred at 
$\ul{Y}\in dom_s(\fR)$ and 
$\cA$
be a unital stably finite 
$\KK-$algebra. 
If 
$\ul{\fa}\in\cA^d$ such that
$\ul{\fa}\in dom_{\cA}(\fR)$, then we know from
Theorem
\ref{thm:30Jan18c}
that
$I_s\otimes\ul{\fa}\in 
\Omega_{\cA}(\cR)$
and 
$I_s\otimes
\fR^{\cA}(\ul{\fa})
=\cR^{\cA}(I_s\otimes\ul{\fa}).$
However, the converse direction that 
$I_s\otimes\ul{\fa}\in\Omega_{\cA}(\cR)$ 
implies 
$\ul{\fa}\in dom^{\cA}(\fR)$ 
is no longer true. 
\fi
\iffalse
\begin{conj}
\label{conj:1Jul19a}
Let  
$\fr\in\KK\plangle\ul{x}\prangle^{\ntn}$ 
and 
$\ul{\fA}\in dom_{\cA}(\fr)$.
If 
$\fr^{\cA}(\ul{\fA})$ is invertible (in 
$\cA^{\ntn}$), 
then there exist nc rational expressions 
$S_{ij}$ 
such that 
$$\ul{\fA}\in dom_{\cA}(S_{ij}) 
\text{ and }
\fr^{-1}=\big(\fS_{ij}\big)_{ij}.$$
\end{conj} 
\fi
For more details on matrix 
valued nc rational 
expressions and functions,
see 
\cite{KV4}.
\begin{definition}
Let $\cA$ be a unital $\KK-$algebra.
We say that $\cA$ \textbf{has property} $\cI$ 
(here $\cI$ 
stands for inversion), if
for every 
$n\in\NN$ and
$\fA=(\fa_{ij})_{1\le i,j\le n}\in\cA^{\ntn}$
that is invertible in
$\cA^{\ntn}$, there exist nc rational expressions 
$R_{ij}$
(for 
$1\le i,j\le n$) in 
$n^2$ 
nc variables 
$\{x_{ij}:1\le i,j\le n\}$, 
such that 
$$\ul{\fa}:=(\fa_{11},\ldots,\fa_{nn})\in dom_{\cA}(R_{ij})
\text{ and }
(\fA^{-1})_{i,j}=R_{ij}^{\cA}(\ul{\fa}),$$
for every 
$1\le i,j\le n$.
\end{definition}
We know certain families of algebras which satisfy Property 
$\cI$, those are when 
$\cA$ 
is either commutative, a skew field or a matrix algebra $\cA=\KK^{\ntn}$
for some $n\in\NN$.
The first case is true due to 
using determinants and an
analog of the Cramer's rule.
The second case can be treated by 
using quasideterminants 
(as developed and discussed in
\cite{GGRW,GR1}), 
or by using Schur complements 
(as in the proof of 
\cite[Lemma 3.9]{V1}
with a smart way of 
choosing any non-zero element of the matrix as a pivot). 
The third case follows directly from
\cite[Lemma 3.9]{V1}.

It is easily seen that once $\cA$
satisfies Property $\cI$, we have 
\begin{align}
\label{eq:19Aug20a}
dom_{\cA}(\fR)=dom_{\cA}^{Mat}(\fR)
\end{align}
for every nc rational function $\fR$. 
Moreover, we get the precise description of 
all algebras with Property 
$\cI$, by 
using the domain and the matrix domain 
of matrix valued nc rational functions
w.r.t $\cA$:
\begin{proposition}
\label{prop:11Sep20a}
Let $\cA$ be a unital $\KK-$algebra. Then $\cA$ has Property $\cI$ if and only if
$dom_{\cA}(\fr)=dom^{Mat}_{\cA}(\fr)$ 
for every matrix valued nc rational function 
$\fr$.
\end{proposition}
Here the $\cA-$domain and the matrix 
$\cA-$domain, respectively, of a matrix valued nc rational function
$\fr=(\fr_{ij})$, 
are defined as the intersection of all the 
$\cA-$domains and matrix $\cA-$domains, respectively, of the nc rational functions $\fr_{ij}$.
\begin{proof}
If 
$\cA$ 
has Property 
$\cI$, 
then it follows from the 
definition of the matrix 
$\cA-$domain that for every matrix valued nc rational function 
$\fr$ 
we have 
$dom_{\cA}(\fr)
=dom^{Mat}_{\cA}(\fr)$.

On the other hand, 
suppose it is true that
$dom_{\cA}(\fr)=dom_{\cA}^{Mat}(\fr)$ for every 
matrix valued nc rational function $\fr$ 
and let 
$\fA=
(\fa_{ij})_{1\le i,j\le n}\in\cA^{\ntn}$ 
be invertible in 
$\cA^{\ntn}$. 
Consider the 
$\ntn$ 
matrix valued nc rational function
$$\fr(x_{11},\ldots,x_{nn})=
\big(\fr_{ij}(x_{11},\ldots,x_{nn})\big)_{1\le i,j\le n}=
\begin{pmatrix}
x_{11}&\ldots&x_{1n}\\
\vdots&&\vdots\\
x_{n1}&\ldots&x_{nn}\\
\end{pmatrix}^{-1},$$
which is invertible as the matrix $(x_{ij})_{1\le i,j\le n}$ is invertible over the free algebra
$\KK\langle x_{11},\ldots,x_{nn} \rangle$.
As 
$\fA$ 
is invertible, we know that
$\ul{\fa}=(\fa_{11},\ldots,\fa_{nn})\in dom^{Mat}_{\cA}(\fr)$, so by the assumption we get 
$\ul{\fa}\in dom_{\cA}(\fr)$,
which implies that 
$\ul{\fa}\in dom_{\cA}(\fr_{ij})$ 
and 
$\big(\fr^{\cA}(\ul{\fa})\big)_{i,j}=
\fr_{ij}^{\cA}(\ul{\fa})$ 
for every 
$1\le i,j\le n$. 
Thus there exist nc rational expressions $R_{ij}$ such that 
$\ul{\fa}\in dom_{\cA}(R_{ij})$ and
$\fr^{\cA}_{ij}(\ul{\fa})=R_{ij}^{\cA}(\ul{\fa})$ for every $1\le i,j\le n$, hence 
$$(\fA^{-1})_{i,j}=
\big(\fr^{\cA}(\ul{\fa})\big)_{i,j}=\fr_{ij}^{\cA}(\ul{\fa})=R_{ij}^{\cA}(\ul{\fa}).$$
\end{proof}
It is not true that every stably finite algebra satisfies Property
$\cI$, 
as the following example --- kindly provided by 
I. Klep and J. Vol\v{c}i\v{c} ---
shows. 
\begin{example}
\label{ex:10Nov19a}
Let
$\fX=
\begin{pmatrix}
x_{11}&x_{12}\\
x_{21}&x_{22}\\
\end{pmatrix},
\mathfrak{Z}=
\begin{pmatrix}
z_{11}&z_{12}\\
z_{21}&z_{22}\\
\end{pmatrix},$
and define 
\begin{align}
\cA:=\faktor{\KK\langle 
\ul{x},\ul{z}
\rangle}{\cJ}=
\Big\{ p+\cJ: p\in\KK\langle\ul{x},
\ul{z}\rangle\Big\},
\end{align}
where 
$\ul{x}=(x_{11},x_{12},x_{21},x_{22}),\,
\ul{z}=(z_{11},z_{12},z_{21},z_{22})$, and
$\cJ$
is the two sided ideal in 
$\KK\langle\ul{x},
\ul{z}\rangle$
generated by the $8$ equations obtained from
$\fX\fZ=\fZ\fX=I_2$. 
For convenience
denote the $8$ equations by 
$h_j(\ul{x},\ul{z})=0$, 
for 
$1\le j\le 8$, 
hence
$\cJ=\langle h_j:1\le j\le 8\rangle$.

It is known that 
$\cA$
is a free ideal ring (e.g., see
\cite[Theorem 5.3.9]{Co95} and 
\cite[Theorem 6.1]{Ber74}) 
and hence \textbf{\uline{$\cA$
is stably finite}},
as it can be embedded into a 
skew field 
which is trivially stably finite.

The natural mapping 
$\varphi:\KK\langle
\ul{x}\rangle\rightarrow
\cA$
given by
$\varphi(p)=p+\cJ$
for every 
$p\in\KK\langle\ul{x}\rangle$,
is \uline{an 
\textbf{embedding} from $\KK\langle
\ul{x}\rangle$ into $\cA$, i.e., 
$\ker(\varphi)=\{0_{\cA}\}$}:
Let 
$f\in\KK\langle\ul{x}\rangle$ such that
$\varphi(f)=0_{\cA}$, 
then 
\begin{align}
\label{eq:1Sep20a}
f(\ul{x})=\sum_{i=1}^m
p_i(\ul{x},\ul{z})h_{k_i}(\ul{x},\ul{z})
q_i(\ul{x},\ul{z})
\end{align}
for some 
$m\in\NN,\,
p_i,q_i\in
\KK\langle\ul{x},\ul{z}\rangle$
and $k_i\in \{1,\ldots,8\}$
for any
$1\le i\le m$. 
Therefore for every
$2\times 2$ 
block matrix
$X=
\big[X_{ij}\big]_{1\le i,j\le 2}
$ invertible over 
$\KK$, 
we let 
$X^{-1}=\big[Z_{ij}\big]_{1\le i,j\le 2}
$, and obtain that
$$
f(\ul{X})=\sum_{i=1}^m
p_i(\ul{X},\ul{Z})h_{k_i}(\ul{X},\ul{Z})
q_i(\ul{X},\ul{Z})=0,
$$
where 
$\ul{X}=(X_{11},X_{12},X_{21},X_{22})$
and
$\ul{Z}=(Z_{11},Z_{12},Z_{21},Z_{22})$.
As 
$f$ 
vanishes on a set of matrices that is Zariski dense, 
it follows that 
$f$ 
vanishes on all matrices and therefore 
$f=0_{\KK\langle\ul{x}\rangle}$. 

\uline{If
$f\in\KK\langle\ul{x}\rangle$
is non-constant, then $\varphi(f)$ is not invertible in $\cA$}:
Suppose 
$f\in\KK\langle\ul{x}\rangle$ is non-constant and that $\varphi(f)$
is invertible in
$\cA$. Thus there exist 
$g\in\KK\langle\ul{x},\ul{z}\rangle$
such that 
$(g+\cJ)\varphi(f)
=\bf{1}_{\cA}$
which implies that 
$fg-1\in\cJ.$
Therefore, similarly to
(\ref{eq:1Sep20a}), 
we get that
whenever the matrix 
$[X_{ij}]_{1\le i,j\le 2}$ is invertible over $\KK$,  
so is the matrix $f(\ul{X})$.
However, this is a contradiction to a claim proved in
\cite[page 79]{V1}, which is the following:
For every non-constant 
$f\in\KK\langle\ul{x}\rangle$, there exists a tuple of matrices 
$\ul{X}=(X_{11},X_{12},X_{21},X_{22})$ for which 
the matrix
$[X_{ij}]_{1\le i,j\le 2}$
is invertible over $\KK$, but $f(\ul{X})$ 
is not invertible over 
$\KK$.

Finally, the algebra $\cA$ 
\uline{\textbf{does not satisfy Property }
$\cI$}:
First, it is obvious that 
$$\fA:=\varphi(\fX)=
\big(\varphi(x_{ij})\big)_{1\le i,j\le 2}
=
\big( x_{ij}+\cJ\big)_{1\le i,j\le 2}
$$is invertible in
$\cA^{2\times2}$, with its inverse being equal
$\big(z_{ij}+\cJ\big)_{1\le i,j\le 2}
$. Now, suppose  
$\cA$
satisfies Property 
$\cI$, therefore there exist 
nc rational expressions 
$R_{ij}$ 
such that
$\ul{\fa}\in dom_{\cA}(R_{ij})$ and 
$R_{ij}^{\cA}(\ul{\fa})
=(\fA^{-1})_{ij}$
for every 
$1\le i,j\le 2$, where 
$\ul{\fa}=(\fa_{11},\fa_{12},\fa_{21},\fa_{22})\in\cA^4$
is given by 
$\fa_{ij}= x_{ij}+\cJ$.
If the expression
$R_{ij}$ contains at least one inversion, then let us take the innermost nested inverse, say 
$\wt{R}^{-1}$ with non-constant nc polynomial $\wt{R}$, 
which appears in $R_{ij}$, 
so we must have 
$\ul{\fa}\in dom_{\cA}(\wt{R}^{-1})$.
Thus
$\wt{R}^{\cA}(\ul{\fa})$
is invertible in $\cA$, 
while a simple calculation shows that 
$\wt{R}^{\cA}(\ul{\fa})
=\wt{R}^{\cA}(\ul{x}+\cJ)
=\wt{R}(\ul{x})+\cJ
=\varphi(\wt{R}),$
hence $\varphi(\wt{R})$ is invertible in $\cA$ and $\wt{R}\in \KK\langle\ul{x}\rangle$ is non-constant, which is a contradiction .
Therefore the \textbf{expressions $R_{ij}$
must be polynomials}, 
so we get that
$(\fA^{-1})_{ij}=R_{ij}^{\cA}(\ul{\fa})=\varphi(R_{ij})$ 
which implies that 
$$
\boldsymbol{1}_{\cA^{2\times2}}=\fA \fA^{-1}=
\big(\varphi(x_{ij})\big)_{1\le i,j\le 2}\cdot
\big(\varphi(R_{ij})\big)_{1\le i,j\le 2},$$
but as 
$\varphi$ is an an embedding, we must have 
$I_2=
\big( x_{ij}\big)_{1\le i,j\le 2}\cdot
\big( R_{ij}\big)_{1\le i,j\le 2},$
which is a contradiction, 
since the matrix
$\big(x_{ij}\big)_{1\le i,j\le 2}
$ is not invertible over 
$\KK\langle\ul{x}\rangle$. 
\end{example}
In the next theorem we realize (up to a tensor product with the identity
matrix) the 
matrix $\cA-$domain of a nc rational function, as the 
$\cA-$domain of any of its minimal realizations, 
centred at any point from its domain of regularity.
\begin{theorem}
\label{thm:10Sep19b}
Let 
$\fR\in\KK\plangle \ul{x}\prangle$. 
For every integer
$s\in\NN$,
a point 
$\ul{Y}\in dom_{s}(\fR)$,
a minimal nc Fornasini--Marchesini realization 
$\cR$ 
centred at 
$\ul{Y}$ 
of 
$\fR$,
and a unital stably finite 
$\KK-$algebra
$\cA$, we have 
\begin{align}
\label{eq:10Nov19d}
dom_{\cA}^{Mat}(\fR)=
\big\{\ul{\fa}\in\cA^d: 
I_s\otimes\ul{\fa}\in\Omega_{\cA}(\cR)\big\}
\end{align} 
and 
$I_s\otimes
\fR^{\cA}(\ul{\fa})
=\cR^{\cA}(I_s\otimes\ul{\fa})$
for every
$\ul{\fa}\in dom^{Mat}_{\cA}(\fR)$.
\end{theorem}
Most of the results in \cite{PV1}
can be modified and 
proven with the setting of 
the matrix domain (cf. 
\cite{KV4}), instead of 
the usual domain of regularity. We will skip most of the details, as the
proofs are pretty much the same as the proofs in the scalar case, 
yet give the reader the instructions to what exact changes have to be done.
\begin{proof}
A parallel version of
\cite[Theorem 2.4]{PV1} ---
that is the existence of a realization formula for any nc rational expression, that is centred at any point from its domain of regularity ---
can be proven for the matrix domain, via synthesis, while adding the part
where block matrices are allowed, however, the proof of 
\cite[Theorem 4.2]{PV1}
covers that part. So one can show the existence of a minimal nc Fornasini--Marchesini
realization
$\wt{\cR}$,
centred at any point from $dom(\fR)$, 
of
$\fR$ w.r.t $\cA$,
in the sense that
$\ul{\fa}\in dom_{\cA}^{Mat}(\fR)$ 
implies 
$I_s\otimes\ul{\fa}\in 
\Omega_{\cA}(\cR)$ 
and $I_s\otimes \fR^{\cA}(\ul{\fa})=
\wt{\cR}^{\cA}(I_s\otimes\ul{\fa})$.
Furthermore, it can be shown then that it holds for every minimal nc Fornasini--Marchesini
realization of 
$\fR$ centred at any point in
$dom(\fR)$.
This proves the first inclusion of 
(\ref{eq:10Nov19d}), that is
$$dom^{Mat}_{\cA}(\fR)\subseteq
\big\{\ul{\fa}\in\cA^d: 
I_s\otimes\ul{\fa}\in\Omega_{\cA}(\cR)\big\}.$$

On the other hand, let
$\ul{\fa}\in\cA^d$ 
and suppose that
$I_s\otimes\ul{\fa}\in\Omega_{\cA}(\cR)$,
thus
the matrix
\begin{multline*}
\Lambda_{\ul{\bA},\ul{Y}}^{\cA}
(I_s\otimes\ul{\fa})= 
I_L\otimes 
{\bf1}_{\cA}-\sum_{k=1}^d
\big[
(I_s\otimes \fa_k)
\bA_k^{\cA}-\bA_k(Y_k)
\otimes {\bf1}_{\cA}
\big]
\\=I_L\otimes {\bf1}_{\cA}
-\sum_{k=1}^d
\big[
\bA_k(I_s)\otimes 
\fa_k-\bA_k(Y_k)\otimes 
{\bf1}_{\cA}
\big]
=\delta_{\cR}^{\cA}(\ul{\fa})
\end{multline*} 
is invertible in 
$\cA^{\LTL}$.
Therefore
$\ul{\fa}\in dom_{\cA}^{Mat}
(\ul{e}_1^T\fr_{\cR}\ul{e}_1),$ 
where we recall that
$\fr_{\cR}$ 
is the 
$\sts$ 
matrix of nc rational functions given in
(\ref{eq:22May20a}).
However, as shown in the proof of 
Theorem
\ref{thm:26Jun19a},
there exists
$f\in\KK\plangle\ul{x}\prangle$
such that
$\fr_{\cR}=I_s\otimes f$, 
while in the proof of Theorem
\ref{thm:3Jun19a} we then showed that $f=\fR$, therefore 
$\fr_{\cR}=I_s\otimes \fR$. Finally, as 
$\ul{e}_1^T\fr_{\cR}\ul{e}_1=\fR$, 
we obtain that
$\ul{\fa}\in dom^{Mat}_{\cA}(\fR)$. 
\end{proof}
As an immediate consequence of 
Theorem
\ref{thm:10Sep19b}
and of
(\ref{eq:19Aug20a}), 
we get the following: 
\begin{cor}
\label{cor:28Aug20a}
Let 
$\fR\in\KK\plangle \ul{x}\prangle$. 
For every integer
$s\in\NN$,
a point 
$\ul{Y}\in dom_{s}(\fR)$,
a minimal nc Fornasini--Marchesini realization 
$\cR$ 
centred at 
$\ul{Y}$ 
of 
$\fR$,
and a unital stably finite 
$\KK-$algebra
$\cA$ which satisfies Property $\cI$, we have 
\begin{align}
%\label{eq:10Nov19d}
dom_{\cA}(\fR)=
\big\{\ul{\fa}\in\cA^d: 
I_s\otimes\ul{\fa}\in\Omega_{\cA}(\cR)\big\}
\end{align} 
and 
$I_s\otimes
\fR^{\cA}(\ul{\fa})
=\cR^{\cA}(I_s\otimes\ul{\fa})$
for every
$\ul{\fa}\in dom^{Mat}_{\cA}(\fR)$.
\end{cor}
\begin{remark}
\label{rem:16May20a}
Let 
$n\in\NN$
and consider the unital
$\KK-$algebra 
$\cA_n=\KK^{\ntn}$. 
It is easily seen that
$dom_{\cA_n}(R)=dom_n(R)$
and 
$R(\ul{\fa})=R^{\cA_n}(\ul{\fa})$,
for every
nc rational expression $R$
and 
$\ul{\fa}\in dom_n(R)$.
Therefore, Corollary
\ref{cor:28Aug20a}
actually gives us another way of evaluating nc rational
functions anywhere on their domain, 
that is by using the 
$\cA_n-$evaluation of the functions. 
\end{remark}

\section{Stable Extended Domain}
\label{sec:stable}
In this last section we use results 
from previous sections to show that the so called 
stable extended domain of a 
nc rational function, coincides with the domain 
of regularity of the function. 
We begin by recalling the definitions and 
some properties of the extended and stable 
extended domains of  nc rational expressions 
and functions, see
\cite{KV1,V1}
for more details.
\\\\
Let 
$\ul{\Xi}=(\Xi_1,\ldots,\Xi_d)$
be the
$d-$tuple of  
$\ntn$ 
generic matrices, i.e., the matrices whose 
entries are independent commuting variables. Let 
$R$ 
be a non-degenerate nc rational expression. 
For every 
$n\in\NN$ 
such that 
$dom_n(R)\ne\emptyset$, 
let 
$R[n]:=R(\ul{\Xi})$, 
that is an 
$\ntn$ 
matrix whose entries are
rational functions in 
$dn^2$ 
(commutative) variables. If 
$R_1$
and
$R_2$ 
are 
$(\KK^d)_{nc}-$evaluation equivalent nc 
rational expressions, then 
$R_1[n]=R_2[n]$
for every 
$n\in\NN$ 
such that 
$dom_n(R_1),dom_n(R_2)\ne\emptyset$. 
Therefore, if 
$\fR\in\KK\plangle\ul{x}\prangle$ 
and 
$dom_n(\fR)\ne\emptyset$, 
we define 
$\fR[n]:=R[n]$ 
for any 
$R\in\fR$ 
such that 
$dom_n(R)\ne\emptyset$.  
Let the \textbf{extended domain} of
$\fR$ 
be
$$edom(\fR):=\coprod_{n=1}^\infty 
edom_n(\fR),$$
where 
$edom_n(\fR)$ 
is defined as the intersection of the 
domains of all entries in 
$\fR[n]$, 
for such 
$n\in\NN$ 
with 
$dom_n(\fR)\ne\emptyset$ 
and 
$edom_n(\fR)=\emptyset$ 
otherwise, thus 
$edom_n(\fR)$ 
is a Zariski open set in
$(\KK^{\ntn})^d$.  
This is the definition of the extended domain 
(at the level of  
$d-$tuples of 
$\ntn$ matrices) as it appears in
\cite{V1}, 
however in
\cite{KV1}
there is a different definition for the 
extended domain,
while it turns out that the stable 
extended domains  coming from these two 
different definitions do coincide, see Remark
\ref{rem:12Apr20a}. 

As pointed out in 
\cite{V1}, 
the extended domain of regularity of a nc 
rational function 
$\fR$ 
is 
\textbf{not} 
closed 
under direct sums, however this can be 
fixed by considering the 
\textbf{stable extended domain}
of 
$\fR$, 
that is
$$edom^{st}(\fR):=\coprod_{n=1}^\infty
edom^{st}_n(\fR),$$
where
$$edom^{st}_n(\fR):=
\Big\{\ul{X}\in 
(\KK^{\ntn})^d:I_m\otimes
\ul{X}\in edom_{mn}(\fR)
\text{ for every }m\in\NN \Big\}.$$
Thus, we have the relations
\begin{align}
\label{eq:11Nov19e}
dom(\fR)\subseteq edom^{st}
(\fR)\subseteq edom(\fR),
\end{align}
while in 
\cite[Theorem 3.10]{V1} 
it was shown that
\begin{align}
\label{eq:5Nov19c}
dom(\fR)=edom^{st}(\fR)
\end{align} 
for every 
$\fR\in\KK\plangle\ul{x}\prangle$ 
with
$dom_1(\fR)\ne\emptyset$.
The proof of
(\ref{eq:5Nov19c}) 
is done by 
considering a descriptor realization and applying the ideas from
\cite{KV1},  that is showing that both the 
domain and the stable extended domain of a 
nc rational function, that is regular 
at a scalar point, are equal to the 
invertibility set of the pencil which appears 
from such a minimal realization. 
\\
\\
The purpose of this section is to generalize 
the equality in
(\ref{eq:5Nov19c})
for an arbitrary 
nc rational function 
$\fR\in\KK\plangle\ul{x}\prangle$, 
by showing that 
$edom^{st}(\fR)$ 
coincides with the invertibility set of a  
pencil which corresponds to a minimal nc 
Fornasini--Marchesini realization of 
$\fR$, 
while on the other hand the invertibility 
set coincides with 
$dom(\fR)$, as we already  showed in Theorem
\ref{thm:3Jun19a}.
The generalization takes most of 
its ideas from 
\cite{KV1} and
\cite{V1}, 
where the case of nc rational functions 
which are regular at a scalar point is 
considered, while invoking our results 
on realizations with a centre of an arbitrary size
and the corresponding generalizations
stated in Section
\ref{subsec:cal}. 

In the spirit of 
\cite[Lemma 3.4]{V1}, 
we first show that the stable extended domain 
of any nc rational function is an upper 
(and lower) admissible nc set.
\begin{lemma}
\label{lem:26Sep19a}
Let 
$\fR\in\KK\plangle\ul{x}\prangle$, 
then  
$edom^{st}(\fR)$ 
is a lower (and upper) admissible nc set.
\end{lemma}
\begin{proof}
Let
$\ul{X}^1=(X^1_1,\ldots,X^1_d)
\in edom^{st}_{n_1}(\fR)$
and
$\ul{X}^2=(X^2_1,\ldots,X^2_d)
\in dom_{n_2}^{st}(\fR)$.
As 
$edom^{st}(\fR)$ 
is closed under direct sums (see
\cite[Proposition 3.3]{V1} 
for a proof), we have 
$\ul{X}^1\oplus \ul{X}^2\in 
edom^{st}_{n_1+n_2}(\fR)$ 
and hence
$\ul{X}^1\oplus \ul{X}^2\in 
edom_{n_1+n_2}(\fR)$.
Let
$\ul{\Xi}^{11}=(\Xi^{11}_1,
\ldots,\Xi^{11}_d)$ 
be a 
$d-$tuple of 
$n_1\times n_1$ 
generic matrices,
$\ul{\Xi}^{22}=
(\Xi^{22}_1,\ldots,\Xi^{22}_d)$ 
be a $d-$tuple of
$n_2\times n_2$ 
generic matrices and 
$\ul{\Xi}^{21}
=(\Xi_1^{21},\ldots,\Xi_d^{21})$ 
be a 
$d-$tuple of 
$n_2\times n_1$ 
generic matrices.
Due to the simple fact that by 
inverting, taking products and taking sums 
of %upper 
lower
triangular block matrices, the 
outcome diagonal blocks only depend on the 
initial diagonal blocks, the denominators 
of the entries in the matrix
$$\fR[n_1+n_2]
\begin{bmatrix}
\ul{\Xi}^{11}&\ul{0}\\
\ul{\Xi}^{21}&\ul{\Xi}^{22}\\
\end{bmatrix}=\fR[n_1+n_2]
\Big(\begin{bmatrix}
\Xi^{11}_1&0\\
\Xi_1^{21}&\Xi^{22}_1\\
\end{bmatrix},\ldots,
\begin{bmatrix}
\Xi^{11}_d&0\\
\Xi_d^{21}&\Xi^{22}_d\\
\end{bmatrix}\Big)$$
are independent of (the entries of matrices in) 
$\ul{\Xi}^{21}$, 
therefore 
$\ul{X}^1\oplus \ul{X}^2
\in edom_{n_1+n_2}(\fR)$ 
implies that
$$\Big(
\begin{bmatrix}
X^1_1&0\\
Z_1&X^2_1\\
\end{bmatrix},\ldots,
\begin{bmatrix}
X^1_d&0\\
Z_d&X^2_d\\
\end{bmatrix}
\Big)=
\begin{bmatrix}
\ul{X}^1&\ul{0}\\
\ul{Z}&\ul{X}^2\\
\end{bmatrix}
\in edom_{n_1+n_2}(\fR)$$
for every 
$\ul{Z}=(Z_1,\ldots,Z_d)\in 
(\KK^{n_2\times n_1})^d$.

Finally, for every 
$\ell\ge1$, 
it is easily seen that
$I_{\ell}\otimes \ul{X}^1\in 
edom_{n_1\ell}^{st}(\fR)$ 
and
$I_{\ell}\otimes\ul{X}^2\in 
edom^{st}_{n_2\ell}(\fR)$, 
these imply by the first part of the proof that
$$\begin{bmatrix}
I_{\ell}\otimes\ul{X}^1&\ul{0}\\
E(n_2,\ell)^T(\ul{Z}\otimes 
I_{\ell})E(n_1,\ell)
&I_{\ell}\otimes\ul{X}^2\\
\end{bmatrix}\in 
edom_{(n_1+n_2)\ell}(\fR),$$
while on the other hand
\begin{multline*}
I_{\ell}\otimes
\begin{bmatrix}
\ul{X}^1&\ul{0}\\
\ul{Z}&\ul{X}^2\\
\end{bmatrix}=E(\ell,n_1+n_2)^T  
\begin{bmatrix}
E(n_1,\ell)&0\\
0&E(n_2,\ell)\\
\end{bmatrix}
\\
\begin{bmatrix}
I_{\ell}\otimes\ul{X}^1&\ul{0}\\
E(n_2,\ell)^T(Z\otimes I_{\ell})
E(n_1,\ell)&I_{\ell}\otimes\ul{X}^2\\
\end{bmatrix}
\begin{bmatrix}
E(n_1,\ell)&0\\
0&E(n_2,\ell)\\
\end{bmatrix}^T
E(\ell,n_1+n_2).
\end{multline*}
Using the fact that 
$edom(\fR)$ 
is closed under simultaneous conjugation, 
we conclude that 
\begin{multline*}
I_{\ell}\otimes
\begin{bmatrix}
\ul{X}^1&\ul{0}\\
\ul{Z}&\ul{X}^2\\
\end{bmatrix}=\Big(
I_{\ell}\otimes
\begin{bmatrix}
X^1_1&0\\
Z_1&X^2_1\\
\end{bmatrix},\ldots, 
I_{\ell}\otimes
\begin{bmatrix}
X^1_d&0\\
Z_d&X^2_d\\
\end{bmatrix}
\Big)
\\ \in edom_{(n_1+n_2)\ell}(\fR),
\,\forall\ell \ge 1,
\end{multline*}
i.e., 
$\begin{bmatrix}
\ul{X}^1&\ul{0}\\
\ul{Z}&\ul{X}^2\\
\end{bmatrix}\in 
edom^{st}_{n_1+n_2}(\fR)$ 
for every 
$\ul{Z}\in(\KK^{n_2\times n_1})^d.$
Similar arguments for %lower 
upper
block triangular 
matrices imply that
$edom^{st}(\fR)$ 
is also %lower 
upper
admissible.
\end{proof}

\begin{remark}
\label{rem:5Nov19b}
The proof of Proposition 3.3 in
\cite{V1} 
uses results from
\cite{HMS} for descriptor realizations (cf. 
\cite[Lemma 3.2]{V1}), 
however if one wants to be self contained,
we can do that 
by using similar arguments regarding nc 
Fornasini--Marchesini realizations
(e.g., see Corollary 
\ref{cor:ExistAndUnique}
and
Theorem
\ref{thm:MainThmFirstPaper},
as the main use of 
realizations in the proof is the existence of 
such a realization with a good domination 
on the domain of the corresponding expression
which represents the function). 
\end{remark}
As
$edom^{st}(\fR)$ 
is an upper admissible nc set, we can now 
consider the nc difference-differential 
calculus of the nc function
$f$, defined on each level by
\begin{align}
\label{eq:11Nov19a}
f[n]=\fR[n]\restriction_{edom^{st}_n(\fR)},
\,\forall n\in\NN
\end{align}
on its domain that is
$edom^{st}(\fR)$.
\begin{lemma}
\label{lem:30Oct19a}
Let
$\fR\in\KK\plangle\ul{x}\prangle$
and let
$f$ 
be as in
(\ref{eq:11Nov19a}).
For every two integers
$k_1,k_2\ge0$ 
such that 
$k_1+k_2>0,\,
n\in\NN,
\,\ul{Y}^1,\ldots,
\ul{Y}^{k_1+k_2}\in 
dom_n(\fR)$,
$Z^1,\ldots,Z^{k_1+k_2}
\in\KK^{\ntn},$
and 
$\omega\in\cG_d$ 
of length 
$k_1+k_2$, 
we have that
\begin{align}
\label{eq:4Nov19a}
\Delta^{\omega} f
(\ul{Y}^1,\ldots,\ul{Y}^{k_1},
\ul{X},\ul{Y}^{k_1+1},
\ldots,\ul{Y}^{k_1+k_2})
(Z^1,\ldots,Z^{k_1+k_2})
\end{align}
is a matrix of rational functions which
are all regular on 
$edom^{st}_n(\fR)$.
\end{lemma}
\begin{proof}
Let
$\ul{W}\in edom^{st}_{n}(\fR)$.
By setting
$\ul{Z}^j=\ul{e}_j^T\otimes Z^j$ 
(for $1\le j\le k_1+k_2$) 
and applying Lemma
\ref{lem:26Sep19a},
we have
\begin{align*}
P_{\ul{W}}:=\begin{pmatrix}
\ul{Y}^1&\ul{Z}^1&\ul{0}&\ldots
&\ldots&\ldots&\ul{0}\\
\ul{0}&\ddots&\ddots&
&&&\vdots\\
\vdots&\ddots&\ul{Y}^{k_1}&\ul{Z}^{k_1}
&\ul{0}&&\vdots\\
\vdots&&\ul{0}&\ul{W}
&\ul{Z}^{k_1+1}&\ddots&\vdots\\
\vdots&&&\ul{0}
&\ul{Y}^{k_1+1}&\ddots&\ul{0}\\
\vdots&&&
&\ddots&\ddots&\ul{Z}^{k_1+k_2}\\
\ul{0}&\ldots&\ldots&\ldots
&\ldots&\ul{0}&\ul{Y}^{k_1+k_2}\\
\end{pmatrix}
\in edom^{st}_{n(k_1+k_2+1)}(\fR),
\end{align*}
which means that
$f(\ul{\Xi})=\fR[n(k_1+k_2+1)](\ul{\Xi})$
is a matrix of rational functions (in
$$dn^2(k_1+k_2+1)^2$$ 
commuting variables) which are all regular at 
$P_{\ul{W}}$, 
where
$\ul{\Xi}$ 
is a 
$d-$tuple of generic matrices of size 
$n(k_1+k_2+1)\times n(k_1+k_2+1)$.
In particular, we can fix all of them, 
except for the ones which correspond 
to the location of 
$\ul{W}$
in the block matrix, to obtain that also 
\begin{align}
\label{eq:12Apr20b}
f[n(k_1+k_2+1)]
\begin{pmatrix}
\ul{Y}^1&\ul{Z}^1&\ul{0}&\ldots
&\ldots&\ldots&\ul{0}\\
\ul{0}&\ddots&\ddots&
&&&\vdots\\
\vdots&\ddots&\ul{Y}^{k_1}&\ul{Z}^{k_1}
&\ul{0}&&\vdots\\
\vdots&&\ul{0}&\ul{\Xi}^\prime
&\ul{Z}^{k_1+1}&\ddots&\vdots\\
\vdots&&&\ul{0}
&\ul{Y}^{k_1+1}&\ddots&\ul{0}\\
\vdots&&&
&\ddots&\ddots&\ul{Z}^{k_1+k_2}\\
\ul{0}&\ldots&\ldots&\ldots
&\ldots&\ul{0}&\ul{Y}^{k_1+k_2}\\
\end{pmatrix}
\end{align}
is a matrix of rational functions (in 
$dn^2$ 
commuting variables) which are all regular at 
$\ul{W}$,
where
$\ul{\Xi}^\prime$
is a 
$d-$tuple of generic matrices of size 
$\ntn$.
However, due to
\cite[Theorem 3.11]{KV3},
we know that the upper most right block matrix of
the matrix in
(\ref{eq:12Apr20b})
is equal to
\begin{align*}
\Delta^{\omega} f
(\ul{Y}^1,\ldots,\ul{Y}^{k_1},
\ul{\Xi}^\prime,\ul{Y}^{k_1+1},
\ldots,\ul{Y}^{k_1+k_2})
(Z^1,\ldots,Z^{k_1+k_2}),
\end{align*}
hence
the matrix in
(\ref{eq:4Nov19a}) 
consists of rational functions, all regular at
$\ul{W}$, 
as needed. 
\end{proof}
Now we are ready to prove that the stable 
extended domain of a nc rational function 
coincides with its domain of regularity, 
as well as with the invertibility set of 
any of its minimal realizations centred 
at a point in 
$dom_s(\fR)$, \textbf{first 
on all levels which are multiples of 
$s$.}
\begin{theorem}
\label{thm:31Jan18a}
Let
$\fR\in\KK\plangle\ul{x}\prangle$
and let 
\iffalse
\begin{align*}
\cR(X_1,\ldots,X_d)= D+C
\Big(I_{L}-\sum_{k=1}^d 
\bA_k(X_k-Y_k)\Big)^{-1}
\sum_{k=1}^d \bB_k(X_k- Y_k)
\end{align*}
\fi
$\cR$
be a minimal nc Fornasini--Marchesini realization
of 
$\fR$, 
centred at
$\ul{Y}\in dom_s(\fR)$. 
Then 
\begin{align}
\label{eq:11No19c}
edom^{st}_{sm}(\fR)= 
\Omega_{sm}(\cR)
=dom_{sm}(\fR),\,
\forall m\in\NN.
\end{align}
\end{theorem}
\begin{proof}
As 
$\cR$ 
is a minimal nc Fornasini--Marchesini realization of 
$\fR$, 
it follows from Theorem
\ref{thm:5Nov19a}
that the coefficients of 
$\cR$ 
satisfy the 
$\cL-\cL\cA$ 
conditions (cf. equations
$(\ref{eq:12Jun17a})-(\ref{eq:12Jun17d})$), 
while Theorem
\ref{thm:19Oct17a} 
implies that 
$\cR:\Omega(\cR)
\rightarrow\KK_{nc}$ 
is a nc function. In the proof of Theorem
\ref{thm:6Aug19a},
(cf. equation 
(\ref{eq:3Mar19d})),
we showed
that
$\varphi(\ul{W})=0$ 
for every 
$m\in\NN$
and
$\ul{W}\in\Omega_{sm}(\cR)$, 
where
\begin{align}
\label{eq:26Sep19d}
\varphi(\ul{X}):=\fo^{(L)}
\Delta_{\Delta_\cR}^{(\ell)}(\ul{X})
\fc^{(R)}
\Lambda_{\ul{\bA},\ul{Y}}
(\ul{X})-I_{Lm},
\end{align} 
$\fo^{(L)}$ 
and 
$\fc^{(R)}$
are constant matrices, 
$\Lambda_{\ul{\bA},\ul{Y}}$ 
is the generalized linear pencil centred at
$\ul{Y}$, 
and 
$\Delta_{\Delta_{\cR}}^{(\ell)}$ 
is a block matrix with entries of the form
$$\Delta^{\omega}\cR(I_m\otimes\ul{Y},
\ldots,I_m\otimes\ul{Y},
\ul{X},I_m\otimes\ul{Y},
\ldots,I_m\otimes\ul{Y})
(Z^1,\ldots,Z^{k_1+k_2}),$$
where
$Z^1,\ldots,Z^{k_1+k_2}\in\cE_{sm}$. 
However, from Theorem
\ref{thm:3Jun19a} we know that
$\cR(\ul{X})=\fR(\ul{X})$ 
for every 
$\ul{X}\in\Omega_{sm}(\cR)$ 
and also that 
$\Omega_{sm}(\cR)=
dom_{sm}(\fR)\subseteq 
edom^{st}_{sm}(\fR)$, 
thus 
$\cR(\ul{X})=
\fR\restriction_{edom^{st}(\fR)}
(\ul{X})=f(\ul{X})$ 
for every
$\ul{X}\in\Omega_{sm}(\cR)$ 
and
$m\in\NN$,  
therefore 
\begin{align}
\label{eq:12Nov19a}
\wt{\varphi}(\ul{X})=0
\end{align}
for every
$\ul{X}\in\Omega_{sm}(\cR)$, 
where
$$\wt{\varphi}(\ul{X})
=\fo^{(L)}
\Delta_{\Delta_f}^{(\ell)}(\ul{X})
\fc^{(R)}
\Lambda_{\ul{\bA},\ul{Y}}
(\ul{X})-I_{Lm}=
\begin{bmatrix}
\wt{\varphi}_{ij}(\ul{X})
\end{bmatrix}_{1\le i,j\le Lm}.$$
By applying Lemma
\ref{lem:30Oct19a}
with
$n=sm$
and
$\ul{Y}^1=\ldots
=\ul{Y}^{k_1+k_2}
=I_m\otimes\ul{Y}$, 
we know that
\begin{align*}
\Delta^{\omega} f
(I_m\otimes\ul{Y},
\ldots,I_m\otimes\ul{Y},
\ul{X},I_m\otimes\ul{Y},
\ldots,I_m\otimes\ul{Y})
(Z^1,\ldots,Z^{k_1+k_2})
\end{align*}
are all matrices of rational functions which 
are regular on
$edom^{st}_{sm}(\fR)$,
therefore
all 
$\wt{\varphi}_{ij}$ 
are
rational functions, regular on
$edom^{st}_{sm}(\fR)$, 
for every 
$1\le i,j\le Lm$.

 As seen in Theorem
\ref{thm:3Jun19a},
$\Omega_{sm}(\cR)=dom_{sm}(\fR)$ 
together with the trivial inclusion 
$dom_{sm}(\fR)\subseteq 
edom^{st}_{sm}(\fR)$, 
we know that 
$\Omega_{sm}(\cR)$
is a non-empty Zariski open set and hence  
Zariski dense in 
$edom_{sm}^{st}(\fR)$. 
Therefore, as all the entries 
$\wt{\varphi}_{ij}$
of 
$\wt{\varphi}$
are rational functions which are regular on
$edom^{st}_{sm}(\fR)$ 
and the equality in 
(\ref{eq:12Nov19a})
holds in 
$\Omega_{sm}(\cR)$ ---
which is Zariski dense subset of 
$edom^{st}_{sm}(\fR)$ --- 
it follows that 
(\ref{eq:12Nov19a}) 
holds for every
$\ul{X}\in edom^{st}_{sm}(\fR)$. 
To conclude, we showed that if 
$\ul{X}\in edom^{st}_{sm}(\fR)$, 
then 
$\Lambda_{\ul{\bA},\ul{Y}}(\ul{X})$ 
is invertible, i.e., that 
$edom^{st}_{sm}
(\fR)\subseteq 
\Omega_{sm}(\cR)$.

On the other hand, we know from
Theorem
\ref{thm:3Jun19a} 
that 
$$\Omega_{sm}(\cR)
=dom_{sm}(\fR)\subseteq 
edom^{st}_{sm}(\fR),$$
hence the equality in 
(\ref{eq:11No19c}).
\end{proof}

Finally, we show that the stable extended domain of a nc rational function
coincides with its domain of  regularity \textbf{on all levels}, hence coincide. Moreover,
at each level they can be identified with the domain of the realization,
i.e., 
the invertibility set 
$\Omega(\cR)$, up to some tensoring with the identity matrix.
\begin{cor}
\label{cor:10Nov19b}
Let
$\fR\in\KK\plangle\ul{x}\prangle$
and let 
\iffalse\begin{align*}
\cR(X_1,\ldots,X_d)= D+C
\Big(I_{L}-\sum_{k=1}^d 
\bA_k(X_k-Y_k)\Big)^{-1}
\sum_{k=1}^d \bB_k(X_k- Y_k)
\end{align*}\fi
$\cR$
be a minimal nc Fornasini--Marchesini realization
of $\fR$, centred at
$\ul{Y}\in dom_s(\fR)$. 
Then 
\begin{align}
\label{eq:11No19d}
edom^{st}_{n}(\fR)
=dom_{n}(\fR)=
\big\{ \ul{X}\in(\KK^{\ntn})^d:
I_s\otimes\ul{X}\in\Omega_{sn}(\cR)\big\}
\end{align}
for every $n\in\NN$. Moreover, we get the equality
\begin{align}
\label{eq:31Jan18b}
dom(\fR)
=edom^{st}(\fR).
\end{align}
\end{cor}
\begin{proof}
\iffalse
For every 
$s\in\NN$ 
such that $dom_s(\fR)\ne\emptyset$, 
there exists $\ul{Y}\in dom_s(\fR)$; let $R\in\fR$ 
be a nc rational expression 
with $\ul{Y}\in dom_s(R)$, 
hence from Theorem 
\ref{}
there exists a minimal nc Fornasini--Marchesini  
realization
$\cR$ 
of 
$R$ 
that is centred at 
$\ul{Y}$.
\fi
From Theorem
\ref{thm:3Jun19a} and the relations in
(\ref{eq:11Nov19e}), 
we know that
$$\big\{
\ul{X}\in(\KK^{\ntn})^d:I_s
\otimes\ul{X}\in\Omega_{sn}(\cR)\big\}
=dom_n(\fR)\subseteq edom^{st}_n(\fR).$$
On the other hand, if 
$\ul{X}\in edom^{st}_n(\fR)$, 
then as $edom^{st}(\fR)$ 
is closed under direct sums
$I_s\otimes\ul{X}\in edom^{st}_{sn}(\fR)$, 
while Theorem
\ref{thm:31Jan18a} 
implies 
$I_s\otimes\ul{X}\in\Omega_{sn}(\cR)$, as needed.
\iffalse
However, the other inclusion 
$dom_{s}(\fR)\subseteq edom^{st}_{s}(\fR)$
is true by definition, hence 
$$edom^{st}_{s}(\fR)=dom_{s}(\fR).$$
The last equality holds for any $s\in\NN$
such that
$dom_s(\fR)\ne\emptyset$ 
and as 
$dom_s(\fR)\ne\emptyset$
if and only if
$edom^{st}_s(\fR)\ne\emptyset$, 
we get
\fi
Finally, it follows that
$$edom^{st}(\fR)=\coprod_{n\in\NN}edom^{st}_n(\fR)
=\coprod_{n\in\NN}
dom_n(\fR)=dom(\fR).$$
\end{proof}
\begin{remark}
\label{rem:12Apr20a}
The definition of the extended domain in
\cite{KV1}
is allegedly different than the one in
\cite{V1}. By the discussion in
\cite{KV1}, the only suspected case in which the extended domains might be
different is when 
$\KK$ is a finite field, $dom_n(\fR)=\emptyset$ but $\fR$ can be evaluated
on $d-$tuples of
$\ntn$ generic matrices. 
This issue might be solved when moving to the stable extended domain, by
taking amplifications 
as described in the proof of
Corollary 
\ref{cor:10Nov19b}.
The idea is 
that if  one follows the definition of the extended domain from
\cite{KV1}, 
then the corresponding stable extended domain contains the domain of regularity
and
is closed under amplifications, so similarly to Corollary 
\ref{cor:10Nov19b} 
we know that it coincides with the domain of regularity of the function.
\end{remark}

\end{document}